%% file: nonspecial.tex
\IfFileExists{nag.sty}{\RequirePackage[l2tabu, orthodox, abort, 
]{nag}}{}   

\documentclass[openany,  
final, 
]   
        {amsart}

\usepackage{amsmath, amsfonts, amsthm} 
\usepackage{amssymb, latexsym}
\usepackage[all]{onlyamsmath}
\usepackage{fixltx2e}
\usepackage{stmaryrd}   

\usepackage{empheq} 

\usepackage[obeyFinal]{todonotes}

\usepackage[
        final]          
        {hyperref}   

\usepackage{showkeys}	
\makeatletter
\@ifpackagelater{showkeys}{2011/11/24}{
	\typeout{Using showkeys v3.16 (or later); this is fine.}
}{
	\@ifpackagelater{showkeys}{2006/01/10}{
		\@ifpackagewith{showkeys}{notcite}{
			\typeout{Using showkeys v3.14 or 3.15 with notcite option; this is fine.}
		}{
			\PackageError{showkeys}{
				Bad version of showkeys, 3.14 or 3.15%
			}{%
				Bug introduced in showkeys v3.14, \MessageBreak
				reported in tools/4173 in LaTeX bugs database: \MessageBreak
				http://www.latex-project.org/cgi-bin/ltxbugs2html?pr=tools/4173 \MessageBreak
				You must either upgrade the showkeys package to [2011/11/24 v3.16] \MessageBreak
				or use it with the [notcite] option.
			}
		}
	}{
		\typeout{Using showkeys v3.13 or earlier (such as v3.12 on Coxeter); this is fine.}
	}
}
\makeatother

\newcommand{\pred}[1]{#1\mathord{\downarrow}}	
\newcommand{\cone}[1]{#1\mathord{\uparrow}}	

\newtheorem{theorem}{Theorem}
\newtheorem{lemma}[theorem]{Lemma}
\newtheorem{remark}[theorem]{Remark}
\newtheorem{corollary}[theorem]{Corollary}
\newtheorem{example}[theorem]{Example}
\newtheorem{question}{Question}

\newtheorem*{main}{Main Theorem}
\newtheorem*{unnumbered}{Theorem}

\newtheorem{claim}{Claim}[theorem]

\newtheorem{definition}{Definition}

\DeclareMathOperator{\cf}{cf}
\DeclareMathOperator{\height}{ht}	

\DeclareMathOperator{\range}{range}

\newcounter{condition}

\title[Stationary Trees and Baumgartner-Hajnal-Todorcevic Theorem]%
        {A Theory of Stationary Trees 
		and the \\
	Balanced Baumgartner-Hajnal-Todorcevic Theorem for Trees}
\author{Ari Meir Brodsky
}
\date{Updated \today\ \copyright}

\address{Department of Mathematics\\
            University of Toronto\\
            Toronto, Ontario, Canada M5S~2E4}
\email{\href{mailto:ari.brodsky@utoronto.ca}{ari.brodsky@utoronto.ca}}

\newcommand{\textofBHTtrees}{%
Let $\kappa$ be any infinite regular cardinal,
let $\xi$ be any ordinal such that $2^{\left|\xi\right|} < \kappa$, and let $k$ be any natural number.
Then
\[
\text{non-$\left(2^{<\kappa}\right)$-special tree } \to \left(\kappa + \xi \right)^2_k.
\]
}

\newcommand{\textofCorBHTpo}{%
Let $\kappa$ be any infinite regular cardinal,
let $\xi$ be any ordinal such that $2^{\left|\xi\right|} < \kappa$, and let $k$ be any natural number.
Let $P$ be a partially ordered set such that $P \to (2^{<\kappa})^1_{2^{<\kappa}}$.
Then
\[
P 
\to \left(\kappa + \xi \right)^2_k.
\]
}

\begin{document}

\begin{abstract}
Building on early work by Stevo Todorcevic, 
we describe a theory of 
stationary subtrees of 
trees of successor-cardinal height. 
We define the diagonal union of subsets of a tree, as well as normal ideals on a tree, 
and we characterize arbitrary subsets of a non-special tree as being either stationary or non-stationary. 

We then use this theory to prove the following partition relation for trees:
\begin{main}
\textofBHTtrees
\end{main}



This is a generalization to trees of the Balanced Baumgartner-Hajnal-Todorcevic Theorem, 
which we recover by applying the above 
to the cardinal $(2^{<\kappa})^+$,
the simplest example of a non-$(2^{<\kappa})$-special tree.

As a corollary, we obtain a general result for partially ordered sets:

\begin{unnumbered}
\textofCorBHTpo
\end{unnumbered}
\end{abstract}

\maketitle

\tableofcontents

\input{intro_Acta_MH.tex}

\input{notation_BHT_trees.tex}

\section{A Theory of Stationary Trees}

\label{section:stationary}

In this 
section,
we discuss how some standard concepts that are defined on ordinals, such as regressive functions,
normal ideals, diagonal unions, and stationary sets can be generalized to nonspecial trees.

\subsection{The Ideal of Special Subtrees of a Tree}

\todo{Consider reversing the order of this subsection and the following subsection.}
%
%
Suppose we fix an infinite cardinal $\kappa$ and a tree of height $\kappa^+$.
What is the correct analogue in $T$ of the ideal of bounded sets in $\kappa^+$?
What is the correct analogue in $T$ of the ideal of nonstationary sets in $\kappa^+$?

As an analogue to the ideal of bounded sets in $\kappa^+$, 
we consider the collection of $\kappa$-special subtrees
of $T$:

\begin{definition}\label{def-special-subtree}
Let $T$ be a tree of height $\kappa^+$.
We say that $U \subseteq T$ is 
a \emph{$\kappa$-special subtree of $T$} if $U$ 
can be written as a union of $\leq \kappa$ many antichains.
That is, $U$ is a $\kappa$-special subtree of $T$ if
\[
U = \bigcup_{\alpha < \kappa} A_\alpha,
\]
where each $A_\alpha \subseteq T$ is an antichain, or equivalently, if
\[
\exists f : U \to \kappa \left(\forall t, u \in U \right) \left[ t <_T u \implies f(t) \neq f(u) \right].
\]
\end{definition}

The collection of $\kappa$-special subtrees of $T$ is clearly a $\kappa^+$-complete ideal on $T$,
and it is proper iff $T$ is itself non-$\kappa$-special.

The cardinal $\kappa^+$ itself is an example of a non-$\kappa$-special tree of height $\kappa^+$.
Letting $T = \kappa^+$, 
we see that the $\kappa$-special subtrees of $\kappa^+$ are precisely the bounded subsets of $\kappa^+$, 
supporting the choice of analogue.

The next important concept on cardinals that we would like to generalize to trees
is the concept of club, stationary, and nonstationary sets.
The problem is that we cannot reasonably define a club subset of a tree 
in a way that is analogous to a club subset of a cardinal.%
\footnote{A natural attempt would be to consider the collection of \emph{closed cofinal} subsets of a tree.
The problem is that this collection is not necessarily a filter base, that is, it is not necessarily \emph{directed}.
For example:
Consider $\sigma\mathbb Q$ to be the collection all (nonempty) \emph{bounded} well-ordered sequences of rationals,
ordered by end-extension.
This is a nonspecial tree, as mentioned in the Introduction. 
Define the two sets
\begin{align*}
C_1 &= \left\{ s \in \sigma\mathbb Q : \sup(s) \in (n, n+1] \text{ for some even integer } n \right\};	\\
C_2 &= \left\{ s \in \sigma\mathbb Q : \sup(s) \in (n, n+1] \text{ for some odd integer } n \right\}.
\end{align*}
Both $C_1$ and $C_2$ are closed cofinal subsets of $\sigma\mathbb Q$,
but $C_1 \cap C_2$ is empty (and therefore not cofinal in $\sigma\mathbb Q$).}
Instead,
\todo{As an alternative, 
Prof.~Tall recommends investigating  the idea of \emph{stationary coding sets}, on $\mathcal P_\kappa(\lambda)$,
introduced in Zwicker's paper.  
This requires further investigation.}
we recall the alternate characterization of stationary and nonstationary subsets 
given by Neumer in~\cite{Neumer}:

\begin{theorem}[Neumer's Theorem]\label{Neumer}
For a regular uncountable cardinal $\lambda$,
and a set $X \subseteq \lambda$,
the following are equivalent:
\begin{enumerate}
\item $X$ intersects every club set of $\lambda$;
\item For every regressive function $f : X \to \lambda$,
there is some $\alpha < \lambda$ such that $f^{-1}(\alpha)$ is unbounded below $\lambda$.
(In the terminology of diagonal unions:  
$X \notin \bigtriangledown \mathcal I$, where $\mathcal I$ is the ideal of bounded subsets of $\lambda$.)
\end{enumerate}
\end{theorem}

We will use this characterization to motivate similar definitions on trees.
First, a few preliminaries:

\subsection{Regressive Functions and Diagonal Unions on Trees}

We begin by formalizing the following definition,
as mentioned in Section~\ref{Section:notation}:


\begin{definition}
For any tree $T$ and node $t \in T$, we define:
\begin{align*}
\pred t &= \left\{ s \in T : s <_T t \right\}	\\
\cone t &= 
	\begin{cases}
		\left\{ s \in T : t <_T s \right\}	&\text{if $t \neq \emptyset$}	\\
		T 					&\text{if $t = \emptyset$}.
	\end{cases}
\end{align*}
\end{definition}

Following immediately from the definition is:
\begin{lemma}
For any $A \subseteq T$ and $t \in T$ we have:
\[
A \cap \cone t = 
	\begin{cases}
		\left\{ s \in A : t <_T s \right\}	&\text{if } t \neq \emptyset	\\
		A 					&\text{if } t = \emptyset.
	\end{cases}
\]
\end{lemma}

We now define the diagonal union of subsets of a tree, indexed by nodes
of the tree.
This is a generalization of the corresponding definition for subsets of a cardinal.

\begin{definition}
Let $T$ be a tree.  For a collection of subsets of $T$ indexed by
nodes of $T$, i.e.
\[
\left< A_t \right>_{t \in T} \subseteq \mathcal P(T),
\]
we define its \emph{diagonal union} to be
\[
\bigtriangledown_{t \in T} A_t = \bigcup_{t \in T} \left( A_t \cap \cone t\right).
\]
\end{definition}
Note that we use $\bigtriangledown$, rather than $\sum$  used by some texts such as~\cite{Jech}.

The following lemma supplies some elementary observations about the
diagonal union operation.  They all reflect the basic intuition that when
taking the diagonal union of sets $A_t$, the only part of each $A_t$
that contributes to the result is the part within 
$\cone t$.

\begin{lemma}\label{diagu}
For any tree $T$ and any collection
\[
\left< A_t \right>_{t \in T} \subseteq \mathcal P(T),
\]
we have:
\begin{align*}
\bigtriangledown_{t \in T} A_t &= \left\{ s \in T : s \in A_\emptyset \cup \bigcup_{t
<_T s} A_t \right\}; \tag{$*$} \label{diagu-alt-def}   \\
\bigtriangledown_{t \in T} A_t &= \bigtriangledown_{t \in T} \left(
A_t \cap \cone t\right);    \\
\bigtriangledown_{t \in T} A_t &= \bigtriangledown_{t \in T} \left(
A_t \cup \pred t \cup \left( \left\{ t \right\} \setminus \left\{ \emptyset \right\} \right) \right);    \\
\bigtriangledown_{t \in T} A_t &= \bigtriangledown_{t \in T} \left(
A_t \cup \left( T \setminus \cone t \right) \right);    \\
\bigtriangledown_{t \in T} A_t &= \bigtriangledown_{t \in T} \left(
A_t \cup X_t\right), \text{ where each $X_t \subseteq T \setminus \cone t$};    \\
\bigtriangledown_{t \in T} A_t &= \bigtriangledown_{t \in T} \left(
A_t \setminus X_t\right), \text{ where each $X_t \subseteq T \setminus \cone t$};    \\
\bigtriangledown_{t \in T} A_t &= \bigtriangledown_{t \in T} \left(
\bigcup_{s \leq_T t} A_s \right);							\\
\bigtriangledown_{t \in T} A_t 
	&= \bigtriangledown_{t \in T} \left( A_t \setminus \bigcup_{s <_T t} A_s \right).
\end{align*}

\begin{proof}[Proof of~\protect\eqref{diagu-alt-def}]
\begin{align*}
\bigtriangledown_{t \in T} A_t
    &= \bigcup_{t \in T} \left(A_t \cap \cone t \right). \\
    &= \left\{ s \in T : \left(\exists t \in T\right) \left[s \in
    A_t \cap \cone t\right] \right\} \\
    &= \left\{ s \in T : \left(\exists t \in T\right) 
			\left[s \in A_t \text{ and } \left(t <_T s \text{ or } t = \emptyset \right) \right] \right\} \\
    &= \left\{ s \in T : 
			s \in A_\emptyset \text{ or } 
	\left(\exists t <_T s\right) s \in A_t \right\} \\
    &= \left\{ s \in T : s \in A_\emptyset \cup \bigcup_{t <_T s} A_t \right\}
\qedhere
\end{align*}
\end{proof}
\end{lemma}

\begin{lemma}
For any tree $T$, if the collections
\[
\left< A_t \right>_{t \in T}, \left<B_t\right>_{t \in T} \subseteq \mathcal P(T)
\]
are such that for all $t \in T$ we have $A_t \subseteq B_t$, then
\[
\bigtriangledown_{t \in T} A_t \subseteq \bigtriangledown_{t\in T} B_t.
\]
\end{lemma}

\begin{lemma}
For any tree $T$, any index set $J$, and collections
\[
\left< A^j_t \right>_{j \in J, t \in T} \subseteq \mathcal P(T),
\]
we have
\[
\bigcup_{j \in J} \left( \bigtriangledown_{t \in T} A^j_t \right) =
\bigtriangledown_{t \in T} \left( \bigcup_{j \in J} A^j_t \right).
\]

\begin{proof}
\begin{align*}
\bigcup_{j \in J} \left( \bigtriangledown_{t \in T} A^j_t \right)
    &= \bigcup_{j \in J} \left( \bigcup_{t \in T}
                \left( A^j_t \cap \cone t \right) \right) \\
    &= \bigcup_{t \in T} \left( \bigcup_{j \in J}
                \left( A^j_t \cap \cone t \right) \right) \\
    &= \bigcup_{t \in T} \left( \left( \bigcup_{j \in J}
                A^j_t \right) \cap \cone t \right) \\
    &= \bigtriangledown_{t \in T} \left( \bigcup_{j \in J}
                A^j_t \right)   \qedhere
\end{align*}
\end{proof}
\end{lemma}

\begin{definition}
Let $\mathcal I \subseteq \mathcal P(T)$ be an ideal.  We
define
\[
\bigtriangledown \mathcal I = \left\{ \bigtriangledown_{t \in T} A_t
: \left< A_t \right>_{t \in T} \subseteq \mathcal I \right\}.
\]
\end{definition}

Some easy facts about $\bigtriangledown \mathcal I$:

\begin{lemma}\label{diagu-is-ideal}
If $\mathcal I$ is any ideal on $T$, then $\mathcal I \subseteq
\bigtriangledown \mathcal I$, and $\bigtriangledown \mathcal I$ is
also an ideal, though not necessarily proper.  Furthermore, for any
cardinal $\lambda$, if $\mathcal I$ is $\lambda$-complete, then so is
$\bigtriangledown \mathcal I$.
\end{lemma}

Notice that the statement $\mathcal I \subseteq \bigtriangledown \mathcal I$ of Lemma~\ref{diagu-is-ideal}
relies crucially on our earlier convention that $\emptyset \in \cone\emptyset$.
Otherwise any set containing the root would never be in $\bigtriangledown \mathcal I$.

%

\begin{lemma}
If $\mathcal I_1, \mathcal I_2 \subseteq \mathcal P(T)$ are two ideals such that $\mathcal I_1 \subseteq \mathcal I_2$,
then $\bigtriangledown \mathcal I_1 \subseteq \bigtriangledown \mathcal I_2$.
\end{lemma}

\begin{definition} \cite[Section~1]{Stevo-SSTC}
Let $X \subseteq T$.  A function $f : X \to T$ is \emph{regressive}
if
\[
\left(\forall t \in X \setminus \{\emptyset\} \right) f(t) <_T t.
\]
\end{definition}

\begin{definition} (cf.~ \cite[p.~7]{BTW})
Let $X \subseteq T$, and let $\mathcal I \subseteq \mathcal P(T)$ be
an ideal on $T$.  A function $f : X \to T$ is called \emph{$\mathcal
I$-small} if
\[
\left(\forall t \in T\right) \left[f^{-1}(t) \in \mathcal I \right].
\]
In words, a function is $\mathcal I$-small iff it is constant only
on $\mathcal I$-sets.  A function is not $\mathcal I$-small iff it
is constant on some $\mathcal I^+$-set.
\end{definition}

\begin{lemma} (cf.~\cite[p.~9]{BTW}) \label{diagu-equiv-regressive}
Let $T$ be a tree, and let $\mathcal I \subseteq \mathcal P(T)$ be
an ideal on $T$.  Then
\[
\bigtriangledown \mathcal I = \left\{ X \subseteq T : \exists \text{
$\mathcal I$-small regressive } f : X \to T \right\}.
\]

\begin{proof}\hfill
\begin{description}
\item[$\subseteq$] Let $X \in \bigtriangledown \mathcal I$.  Then we
can write
\[
X = \bigtriangledown_{t \in T} X_t = \bigcup_{t \in T} \left(X_t
\cap \cone t \right),
\]
where each $X_t \in \mathcal I$.
Define $f : X \to T$ by setting, for each $s \in X$, $f(s) = t$,
where we choose some $t$ such that $s \in X_t \cap \cone t$.
It is clear that $f$ is regressive.  Furthermore, for any $t \in T$,
\[
f^{-1}(t) \subseteq X_t \in \mathcal I,
\]
so $f^{-1}(t) \in \mathcal I$, showing that $f$ is $\mathcal
I$-small.
\item[$\supseteq$]  Let $X \subseteq T$, and fix an $\mathcal
I$-small regressive function $f : X \to T$.  For each $t \in T$,
define
\[
X_t = f^{-1}(t).
\]
Since $f$ is $\mathcal I$-small, each $X_t \in \mathcal I$.  Since
$f$ is regressive, we have $X_t \subseteq \cone t$ for each $t \in T$.
We then have
\begin{align*}
X &= \bigcup_{t \in T} f^{-1}(t) \\
  &= \bigcup_{t \in T}X_t \\
  &= \bigcup_{t\in T} \left(X_t \cap \cone t\right) \\
  &= \bigtriangledown_{t\in T} X_t \in \bigtriangledown \mathcal I.
  \qedhere
\end{align*}
\end{description}
\end{proof}
\end{lemma}

Notice that in the proof of Lemma~\ref{diagu-equiv-regressive}, 
the special treatment of $\emptyset$ in the definition of $\cone \emptyset$ corresponds to the exclusion
of $\emptyset$ from the requirement that $f(t) <_T t$ in the definition of regressive function.

Taking complements, we have:

\begin{corollary}\label{diagu-complement}
For any ideal $\mathcal I \subseteq \mathcal P(T)$, we have
\[
\left( \bigtriangledown \mathcal I \right)^+ = \left\{ X \subseteq T
: \left( \forall \text{ regressive } f: X \to T\right) \left(\exists
t \in T \right) \left[f^{-1}(t) \in \mathcal I^+\right] \right\}.
\]
In words, a set $X$ is $(\bigtriangledown \mathcal I)$-positive iff
every regressive function on $X$ is constant on an $\mathcal
I^+$-set.
\end{corollary}

%

\begin{corollary}
For any ideal $\mathcal I \subseteq \mathcal P(T)$, the following
are equivalent:
\begin{enumerate}
\item $\mathcal I$ is closed under diagonal unions, that is,
$\bigtriangledown \mathcal I = \mathcal I$;
\item If $X \in \mathcal I^+$, and $f : X \to T$ is a regressive
function, then $f$ must be constant on some $\mathcal I^+$-set, that
is, $(\exists t\in T) f^{-1}(t) \in \mathcal I^+$.
\end{enumerate}
\end{corollary}

\begin{definition}
An ideal $\mathcal I$ on $T$ is \emph{normal} if it is closed under
diagonal unions (that is, $\bigtriangledown \mathcal I = \mathcal
I$), or equivalently, if every regressive function on an $\mathcal
I^+$ set must be constant on an $\mathcal I^+$ set.
\end{definition}

A natural question arises:  For a given ideal, how many times must
we iterate the diagonal union operation $\bigtriangledown$ before
the operation stabilizes and we obtain a normal ideal? 
In particular, when is $\bigtriangledown$ idempotent? 
The following lemma gives us a substantial class of ideals for which the
answer is \emph{one}, and this will be a useful tool in later proofs:

\begin{lemma}[Idempotence Lemma]\label{diagu-idempotent}
Let $\lambda = \height (T)$, and suppose $\lambda$ is any 
cardinal. If $\mathcal I$ is a $\lambda$-complete ideal on $T$, then
$\bigtriangledown \bigtriangledown \mathcal I = \bigtriangledown
\mathcal I$, that is, $\bigtriangledown \mathcal I$ is normal.

\begin{proof}
$\bigtriangledown \mathcal I \subseteq \bigtriangledown
\bigtriangledown \mathcal I$ is always true, so we must show
$\bigtriangledown \bigtriangledown \mathcal I \subseteq
\bigtriangledown \mathcal I$.  Let $X \in \bigtriangledown
\bigtriangledown \mathcal I$.  We must show $X \in \bigtriangledown
\mathcal I$.

As $X \in \bigtriangledown \bigtriangledown \mathcal I$, we can
write
\[
X = \bigtriangledown_{t \in T} A_t,
\]
where each $A_t \in \bigtriangledown \mathcal I$.  For each $t \in
T$, we can write
\[
A_t = \bigtriangledown_{s \in T} B_t^s,
\]
where each $B_t^s \in \mathcal I$.

Notice that for each $t \in T$, the only part of $A_t$ that
contributes to $X$ is the part within $\cone t$.  For each
$s, t \in T$, the only part of $B_t^s$ that contributes to $A_t$ is
the part within $\cone s$.  We therefore have:
\begin{itemize}
\item If $s$ and $t$ are incomparable in $T$, we have $\cone s \cap \cone t
= \emptyset$, so $B_t^s$ does not contribute anything to $X$;
\item If $t \leq_T s$ then $\cone s \cap \cone t = \cone s$, so the only part of
$B_t^s$ that contributes to $X$ is within $\cone s$;
\item If $s \leq_T t$ then $\cone s \cap \cone t = \cone t$, so the only part of
$B_t^s$ that contributes to $X$ is within $\cone t$.
\end{itemize}

With this in mind, we collect the sets $B_t^s$ whose contribution to
$X$ lies within any $\cone r$.  We define, for
each $r \in T$,
\[
D_r = \bigcup_{t \leq_T r} B_t^r \cup \bigcup_{s \leq_T r} B_r^s.
\]
Since $\mathcal I$ is $\lambda$-complete and each $r$ has height $<\lambda$, it
is clear that $D_r \in \mathcal I$.

\begin{claim}
We have
\[
X = \bigtriangledown_{r \in T} D_r.
\]

\begin{proof}
\begin{align*}
X &= \bigtriangledown_{t \in T} A_t \\
  &= \bigtriangledown_{t \in T} \bigtriangledown_{s \in T} B_t^s \\
  &= \bigtriangledown_{t \in T} \bigcup_{s \in T}
                \left( B_t^s \cap \cone s\right) \\
  &= \bigcup_{t \in T} \left[ \bigcup_{s \in T}
                \left( B_t^s \cap \cone s\right) \cap \cone t\right] \\
  &= \bigcup_{t \in T} \bigcup_{s \in T}
                \left( B_t^s \cap \cone s \cap \cone t\right) \\
  &= \bigcup_{t, s \in T}
                \left( B_t^s \cap \cone s \cap \cone t\right) \\
  &= \bigcup_{\substack{t, s \in T \\ t \leq_T s}} \left( B_t^s \cap \cone s\right) \cup
     \bigcup_{\substack{t, s \in T \\ s \leq_T t}} \left( B_t^s \cap \cone t\right) \\
  &= \bigcup_{r \in T} \left[ \bigcup_{t \leq_T r} \left( B_t^r \cap \cone r\right) \cup
     \bigcup_{s \leq_T r} \left( B_r^s \cap \cone r\right)  \right] \\
  &= \bigcup_{r \in T} \left[ \left( \bigcup_{t \leq_T r} B_t^r \cup
     \bigcup_{s \leq_T r} B_r^s \right) \cap \cone r \right] \\
  &= \bigcup_{r \in T} \left( D_r \cap \cone r\right) \\
  &= \bigtriangledown_{r \in T} D_r. \qedhere
\end{align*}
\end{proof}
\end{claim}

It follows that $X \in \bigtriangledown \mathcal I$, as required.
\end{proof}
\end{lemma}

\subsection{The Ideal of Nonstationary Subtrees of a Tree}



Armed with Neumer's characterization of stationary (and nonstationary) subsets of a cardinal in terms of diagonal unions
(Theorem~\ref{Neumer}),
we now explore an analogue for trees of this concept, 
using the new concepts we have introduced in the previous subsection:

\begin{definition}
Let $B \subseteq T$, where $T$ is a tree of height $\kappa^+$.  
We say that $B$ is a \emph{nonstationary subtree of $T$} if we can write
\[
B = \bigtriangledown_{t \in T} A_t,
\]
where each $A_t$ is a $\kappa$-special subtree of $T$.
We may, for emphasis, refer to $B$ as \emph{$\kappa$-nonstationary}.
If $B$ cannot be written this way, then $B$ is a \emph{stationary subtree of $T$}.

We define $NS^T_\kappa$ to be the collection of nonstationary subtrees of $T$.
That is, $NS^T_\kappa$ is the diagonal union of the ideal of $\kappa$-special subtrees of $T$.
(The subscript $\kappa$ is for emphasis and may sometimes be omitted.)
\end{definition}

Our definitions here are new, and in particular are different from Todorcevic's earlier use of 
$I_T$ in~\cite{Stevo-SSTC} and $NS_T$ in~\cite{Stevo-PRPOS}.
Todorcevic defines $NS_T$ as an ideal on the cardinal $\kappa^+$, 
consisting of 
subsets of $\kappa^+$ that are said to be 
nonstationary \emph{in} or \emph{with respect to} $T$,
while we define $NS^T_\kappa$ as an ideal on the tree $T$ itself,
consisting of sets that are nonstationary \emph{subsets of} $T$.
For any set $X \subseteq \kappa^+$, the statement 
$T \upharpoonright X \in NS^T_\kappa$ in our notation means the same thing as $X \in NS_T$ of~\cite{Stevo-PRPOS}.
However, our definitions will allow greater flexibility in stating and proving the relevant results.
In particular, we can discuss the membership of 
arbitrary subsets of the tree in the ideal $NS^T_\kappa$, 
rather than only those of the form $T \upharpoonright X$ for some $X \subseteq \kappa^+$.

In the case that $T = \kappa^+$, 
the fact that $NS^T_\kappa$ is identical to the collection of nonstationary sets in the usual sense
(that is, sets whose complements include a club subset of $\kappa^+$)
is Theorem~\ref{Neumer} (Neumer's Theorem), 
so the analogue is correct.
In fact, more can be said about the analogue:
In what may be historically the first use%
\footnote{See~\cite[footnote~214]{Larson-sets-extensions} and~\cite[p.~105]{Jech}.} 
of the word \emph{stationary} 
(actually, the French word \emph{stationnaire}) in the context of regressive functions,
G\'erard Bloch~\cite{Bloch} \emph{defines} a set $A \subseteq \omega_1$ to be stationary if
every regressive function on $A$ is constant on an uncountable set,
and then states as a \emph{theorem} that a set is stationary if its complement includes no club subset,
rather than using the latter characterization as the definition of stationary 
as would be done 
nowadays
(cf.~
\cite[p.~251]{Stevo-SSTC}, \cite[Prop.~I.2.1(i)]{BTW}).
So the extension to stationary subtrees of a tree really is a direct generalization 
of the original definition of stationary subsets of a cardinal!

The following lemma collects facts about $NS^T_\kappa$ that follow easily from Lemma~\ref{diagu-is-ideal}:

\begin{lemma}\label{special-is-nonstationary}
Fix a tree $T$ of height $\kappa^+$.  
Then every $\kappa$-special subtree of $T$ is a nonstationary subtree.
Furthermore, $NS^T_\kappa$ is a $\kappa^+$-complete ideal on $T$.
\end{lemma}

The converse of the first conclusion of Lemma~\ref{special-is-nonstationary} is false.
In the special case where $T$ is just the cardinal $\kappa^+$,
there exist unbounded nonstationary subsets of $\kappa^+$ 
(for example, 
the set of successor ordinals less that $\kappa^+$),
so any such set is a nonstationary subtree of $\kappa^+$ that is not $\kappa$-special.
This also means that the ideal of bounded subsets of $\kappa^+$ is not normal,
so that in general the ideal of $\kappa$-special subtrees of a tree $T$ is not a normal ideal.
However, we do have the following generalization to trees of Fodor's Theorem:

\begin{theorem}\label{nonstationary-normal}
For any tree $T$ of height $\kappa^+$, the ideal $NS^T_\kappa$ is a normal ideal on $T$.

\begin{proof}
This follows from the Idempotence Lemma (Lemma~\ref{diagu-idempotent}), 
since the ideal of $\kappa$-special subtrees is $\kappa^+$-complete.
\end{proof}
\end{theorem}

Theorem~\ref{nonstationary-normal} tells us that
$\bigtriangledown NS^T_\kappa = NS^T_\kappa$.
Equivalently:
If $B$ is a stationary subtree of $T$, 
meaning that every regressive function on $B$ is constant on a non-$\kappa$-special subtree of $T$,
then in fact every regressive function on $B$ is constant on a stationary subtree of $T$.
So for any tree $T$ of height $\kappa^+$,
the main tool for extracting subtrees using regressive functions should be  the ideal $NS^T_\kappa$,
rather than the ideal of $\kappa$-special subtrees of $T$.

Theorem~\ref{nonstationary-normal} is stated without proof as~\cite[Theorem~13]{Stevo-PRPOS},
and the special case for trees of height $\omega_1$ is proven as~\cite[Theorem~2.2(i)]{Stevo-SSTC}.
The simplicity of our proof, compared to the one in~\cite{Stevo-SSTC}, 
is a result of our new definitions and machinery that we have built up to this point.


The ideal $NS^T_\kappa$ will be useful if we know that it is proper.
When can we guarantee that $T \notin NS^T_\kappa$?
The following lemma will be a crucial ingredient in the proof of Theorem~\ref{pressing-down-trees}:

\begin{lemma}\label{special-above-antichain}
Let $A \subseteq T$ be an antichain.  
For each $t \in A$, fix a $\kappa$-special subtree $X_t \subseteq \cone t$.  Then
\[
\bigcup_{t \in A} X_t 
\]
is also a $\kappa$-special subtree of $T$.
That is, a union of $\kappa$-special subtrees above pairwise incompatible nodes is also a $\kappa$-special subtree.
\end{lemma}

While Lemma~\ref{special-above-antichain} is easily seen to be true,
what is significant about it is the precision of its hypotheses.
If instead of each $X_t$ being a $\kappa$-special subtree 
we require it to be a union of at most $\kappa$ \emph{levels} of the tree,
even if we require $A$ to consist of nodes on a single level,
we do not get the result that the union of all $X_t$ is a union of $\kappa$ levels of the tree.
So in the development of our theory we cannot replace the ideal of $\kappa$-special subtrees with
the ideal of subtrees consisting of (at most) $\kappa$ levels of the tree,
even though the latter is also a $\kappa^+$-complete ideal on the tree, 
and may appear to be a reasonable generalization to trees 
of the concept of bounded subsets of $\kappa^+$.


Similarly, if we try to generalize to trees of height a limit cardinal $\lambda$ rather than $\kappa^+$, 
replacing the ideal of $\kappa$-special subtrees with 
the ideal of subtrees that are unions of strictly fewer than $\lambda$ antichains,
we do not get an analogue of Lemma~\ref{special-above-antichain}
(even if the height is a regular limit cardinal),
and this is why Theorem~\ref{pressing-down-trees} is not valid for trees of limit cardinal height.%
\footnote{It is possible to extend the definitions of \emph{special} and \emph{nonspecial} to trees with height
an arbitrary regular cardinal, as Todorcevic does in~\cite[Chapter~6]{Stevo-walks}.
The essential difficulty is that Theorem~\ref{pressing-down-trees} doesn't hold for trees of limit-cardinal height,
but this is overcome by starting with the characterization $T \notin NS^T_\kappa$
in Theorem~\ref{pressing-down-trees} as the \emph{definition} of nonspecial,
rather than Definition~\ref{def-special-subtree}.
In this exposition, we have chosen to restrict our investigation to trees of successor-cardinal height.
}

Obviously, if a tree is special, then all of its subtrees are special and therefore nonstationary.
Theorem~\ref{pressing-down-trees} gives the converse, 
establishing the significance of using a nonspecial tree as our ambient space.
It is a generalization to nonspecial trees of a theorem of Dushnik~\cite{Dushnik} on successor cardinals%
\footnote{Though he does not say so, Dushnik's proof works for regular limit cardinals as well.
Nevertheless, it does not generalize to trees of regular-limit-cardinal height,
due to the failure of Lemma~\ref{special-above-antichain} in that case, as we have explained.},
which itself was a generalization of Alexandroff and Urysohn's theorem~\cite{AU} on $\omega_1$.

The proofs in \cite{AU} and \cite{Dushnik} are substantially different from each other,
and each one of them has been generalized to prove theorems for which the other method would not be suitable.
The main ingredient in~\cite{Dushnik} is a cardinality argument, 
and this is the proof that extends to trees, where the focus will be on counting antichains, as we shall see.
On the other hand, the main argument of~\cite{AU} involves cofinality, 
and this is the argument that is adaptable to prove Theorem~\ref{Neumer} (Neumer's Theorem),
but does not extend easily to trees.

The case of Theorem~\ref{pressing-down-trees} for nonspecial trees of height $\omega_1$ 
is proven as~\cite[Theorem~2.4]{Stevo-SSTC}.
The general case is subsumed by~\cite[Theorem~14]{Stevo-PRPOS},
but we present the theorem and its proof here, for several reasons:
to indicate the generality of Dushnik's technique as it applies to trees of successor-cardinal height,
to isolate this theorem and its proof from the harder portion of~\cite[Theorem~14]{Stevo-PRPOS}
(which we state later as Theorem~\ref{stationary subtree}),
and to show how the statement of the theorem and its proof are affected 
by our new terminology and notation.

\begin{theorem}[Pressing-Down Lemma for Trees] \label{pressing-down-trees}
Suppose $T$ is a non-$\kappa$-special tree.
Then $NS^T_\kappa$ is a proper ideal on $T$, that is, $T \notin NS^T_\kappa$.

\begin{proof}
Fix a non-$\kappa$-special tree $T$, and
suppose $
\left< X_t \right>_{t \in T} 
$ is any indexed collection of $\kappa$-special subtrees of $T$.
We will show that
\[
T \neq \bigtriangledown_{t \in T} X_t.
\]

We define a sequence of subtrees of $T$ by recursion on $n<\omega$, as follows:
Let
\[
S_0 = \left\{\emptyset\right\},
\]
and for $n < \omega$, define
\[
S_{n+1} = \bigcup_{t \in S_n} \left( X_t \cap \cone t\right).
\]

\begin{claim}
For all $n<\omega$, $S_n$ is a $\kappa$-special subtree of $T$.

\begin{proof}
We prove this claim by induction on $n$.  Certainly $S_0 =
\{\emptyset\}$ is a $\kappa$-special subtree as it contains only one element.
Now fix $n<\omega$ and suppose $S_n$ is a $\kappa$-special subtree.  We need to
show that $S_{n+1}$ is $\kappa$-special.

Since $S_n$ is $\kappa$-special, we can write
\[
S_n = \bigcup_{\alpha<\kappa} A_\alpha,
\]
where each $A_\alpha$ is an antichain.
For each $t \in T$ we know that $X_t \cap \cone t$ is a $\kappa$-special subtree of $\cone t$,
so for each $\alpha < \kappa$, Lemma~\ref{special-above-antichain} tells us that
\[
\bigcup_{t \in A_\alpha} \left( X_t \cap \cone t\right) 
\]
is a $\kappa$-special subtree of $T$.
We then have
\[
S_{n+1} = \bigcup_{\alpha<\kappa} \bigcup_{t \in A_\alpha} \left( X_t \cap \cone t\right),
\]
so that $S_{n+1}$ is a union of $\kappa$ many $\kappa$-special subtrees, and is therefore $\kappa$-special,
%
%
completing the induction.
%
%
%
%
%
\end{proof}
\end{claim}

Since $T$ is a non-$\kappa$-special tree, 
and a union of countably many $\kappa$-special subtrees is also $\kappa$-special,
we have
\[
T \setminus \bigcup_{n < \omega} S_n \neq \emptyset,
\]
so we fix a $<_T$-minimal element $s$ of that set.

\begin{claim}
We have
\[
s \notin \bigtriangledown_{t \in T} X_t.
\]

\begin{proof}
By equivalence~\eqref{diagu-alt-def} of Lemma~\ref{diagu},
we need to show that
\[
s \notin X_\emptyset \cup \bigcup_{t <_T s} X_t.
\]
Since $\emptyset \in S_0$, we have $s \neq \emptyset$,
so we just need to show that for any $t <_T s$, we have $s \notin X_t$.
So suppose $t <_T s$.
Since $s$ was minimally not in any $S_n$, we must have $t \in S_n$ for some $n< \omega$.
If $s$ were in $X_t$, then by definition of $S_{n+1}$ we would have $s \in S_{n+1}$,
contradicting the choice of $s$.
So $s$ is not in any relevant $X_t$, as required.
\end{proof}
\end{claim}

We have thus found $s \in T$ that is not in the diagonal union of the $\kappa$-special sets $X_t$, 
as required to show that $T \notin NS^T_\kappa$.
\end{proof}
\end{theorem}

What other nonstationary subtrees can we come up with?

\begin{lemma}\label{isolated-nonstationary}
Let $T$ be any tree of height $\kappa^+$, and let $S \subseteq T$ be any subtree.
Then the set of isolated points%
\footnote{\label{tree-topology}%
The topology on $T$ is the \emph{tree topology}, defined by any of the following equivalent formulations:
\begin{enumerate}
\item The tree topology has, as its basic open sets, 
all chains $C \subseteq T$ such that $\height_T''C$ is open as a set of ordinals.
\item \cite[p.~14]{Rudin-lectures} The tree topology has, as its basic open sets,
the singleton root $\{\emptyset\}$ as well as all intervals (chains) of the form $(s,t]$ for $s <_T t$ in $T$.
\item \cite[p.~244]{Stevo-TLOS} The tree topology has, as its basic open sets,
all intervals (chains) of the form $(s,t]$ for $s <_T t$ in $T \cup \{-\infty\}$.
\item \cite[Definition~III.5.15]{New-nen} 
A set $U \subseteq T$ is open in the tree topology iff for all $t \in U$ with height a limit ordinal,
\[
\left(\exists s <_T t\right) \left[\cone{s} \cap \pred{t} \subseteq U\right].
\]
\item A point $t \in T$ is a limit point of a set $X \subseteq T$ in the tree topology iff 
($\height_T (t)$ is a limit ordinal and)
$X \cap \pred{t}$ is unbounded below $t$.  
\end{enumerate}
The fact that the tree topology doesn't have an easily intuitive definition in terms of basic open sets
(as seen especially by the awkward semi-open intervals in (2) and (3) above)
seems to relate to the fact that there is no obvious order topology on an arbitrary partial order.
} 
of $S$ is a nonstationary subtree of $T$.

\begin{proof}
Let $R$ be the set of isolated points of $S$.
Define a function $f : R \to T$ by setting, for $t \in R$, 
\[
f(t) = \sup \left(S \cap \pred t \right),
\]
where the sup is taken along the chain $\pred t \cup \{t\}$.


For any $t \in R$,
$t$ is an isolated point of $S$, so $S \cap \pred t$ must be bounded below $t$, so that $f(t) <_T t$.
This shows that $f$ is regressive.

\begin{claim}
For each $s \in T$, $f^{-1}(s)$ is an antichain.

\begin{proof}
If $t_1 <_T t_2$ are both in $R$, then $f(t_1) <_T t_1 \leq_T f(t_2)$.
\end{proof}
\end{claim}

So $R$ is a diagonal union of antichains, and is therefore a nonstationary subtree of $T$, as required.
\end{proof}
\end{lemma}

What do we know about the status of sets of the form $T \upharpoonright X$, for some $X \subseteq \kappa^+$,
with respect to the ideal $NS^T_\kappa$?
The following facts are straightforward:

\begin{lemma}\label{clear subtrees}
Let $T$ be any tree of height $\kappa^+$, and let $X, C \subseteq \kappa^+$.  Then:
\begin{enumerate}
\item If $\left|X\right| \leq \kappa$ then $T \upharpoonright X$ is a $\kappa$-special subtree of $T$.
\item If $X$ is a nonstationary subset of $\kappa^+$, then $T \upharpoonright X \in NS^T_\kappa$.
\item In particular, the set of successor nodes
 of $T$ is a nonstationary subtree of $T$.
\item If $C$ is a club subset of $\kappa^+$, 
then $T \upharpoonright C \in (NS^T_\kappa)^*$.
\item If $T$ is a non-$\kappa$-special tree 
and $C$ is a club subset of $\kappa^+$, then $T \upharpoonright C \notin NS^T_\kappa$.
\end{enumerate}

\begin{proof}\hfill
\begin{enumerate}
\item $T \upharpoonright X$ is a union of $\left|X\right|$ antichains.
\item (cf.~\cite[p.~251]{Stevo-SSTC}) 
Let $X$ be a nonstationary subset of $\kappa^+$.
By
Theorem~\ref{Neumer} (Neumer's characterization of nonstationary subsets of a cardinal),
we can choose a regressive function $f : X \to \kappa^+$ such that 
$\left|f^{-1}(\alpha) \right| < \kappa^+$ for every $\alpha < \kappa^+$.
This induces a regressive function $f_T : T \upharpoonright X \to T$, as follows:
%
%
For every $t \in T \upharpoonright X$, 
let $f_T(t) \leq_T t$ be such that 
\[
\height_T(f_T(t)) = f(\height_T(t)).
\]
The function $f_T$ is regressive,
and for each $s \in T$ the set $f_T^{-1}(s)$ is a $\kappa$-special subtree by part~(1), 
since it is a subset of $T \upharpoonright f^{-1}(\height_T(s))$.
It follows that $T \upharpoonright X$ is a nonstationary subtree of $T$, as required.

\item This follows from part~(2), 
since the set of successor ordinals below $\kappa^+$ is a nonstationary subset of $\kappa^+$.
Alternatively, the successor nodes are precisely the isolated points of $T$, 
so we can apply Lemma~\ref{isolated-nonstationary} to the whole tree $T$.

\item We have 
\[
T \setminus \left(T \upharpoonright C\right)  = T \upharpoonright \left(\kappa^+ \setminus C\right) \in NS^T_\kappa
\]
by part (2), 
so $T \upharpoonright C$ is the complement of an ideal set and therefore in the filter.
\item By the 
Pressing-Down Lemma for Trees (Theorem~\ref{pressing-down-trees}), 
$NS^T_\kappa$ is a proper ideal on $T$,
so that 
\[
\left(NS^T_\kappa\right)^* \subseteq \left(NS^T_\kappa\right)^+,
\]
giving the required result. 
\qedhere
\end{enumerate}
\end{proof}
\end{lemma}

It is a standard textbook theorem (see e.g.~\cite[Lemma~III.6.9]{New-nen}) 
that for any regular infinite cardinal $\theta < \kappa^+$, the set
\[
S^{\kappa^+}_\theta = \left\{ \gamma < \kappa^+ : \cf(\gamma) = \theta \right\}
\]
is a stationary subset of $\kappa^+$.
A partial analogue to this theorem for trees is:

\begin{theorem}\cite[Theorem~14, (2) $\implies$ (3)]{Stevo-PRPOS}\label{stationary subtree}
If $T$ is a non-$\kappa$-special tree, then the subtree
\[
T \upharpoonright S^{\kappa^+}_{\cf(\kappa)} = 
	\left\{ t \in T : \cf(\height_T(t)) = \cf(\kappa) \right\}
\]
is a stationary subtree of $T$.
\end{theorem}

Of course, in the case where $T$ has height $\omega_1$ (that is, where $\kappa = \omega$),
Theorem~\ref{stationary subtree} provides no new information, 
because the set of ordinals with countable cofinality is just the set of limit ordinals below $\omega_1$
and is therefore a club subset of $\omega_1$,
so that Lemma~\ref{clear subtrees}(5) applies.
But when $\kappa > \omega$, 
Theorem~\ref{stationary subtree} provides a nontrivial example of a stationary subtree of $T$ 
whose complement is not (necessarily) nonstationary.

\input{BHT_generalization_to_trees.tex}

\input{very_nice_BHT.tex}

\input{BHT_main_proof.tex}




\makeatletter
\providecommand\@dotsep{5}
\makeatother
\listoftodos\relax

\end{document}

%% file: intro_Acta_MH.tex


\section{Introdiction and Background}

\subsection{Partition Calculus}

Partition calculus, as a discipline within set theory, was developed
by Erd\H{o}s and Rado in their seminal paper \cite{E-R}, appearing more than fifty
years ago.
It 
offers a rich theory with many surprising and deep
results,
surveyed in texts such as~
\cite{Williams} and~\cite{EHMR}. 
However, the primary focus of the early development of partition
calculus was 
exclusively on linear (total) order types,
including cardinals and ordinals as specific examples.  
It wasn't
until the 1980s that Todorcevic \cite{Stevo-PRPOS} pioneered the
systematic study of partition relations for partially ordered sets,
although the extension of the partition calculus to non-linear order
types began with Galvin \cite{Galvin} and the idea was anticipated
even by Erd\H os and Rado~\cite[p.~430]{E-R}.  

As we will see (\autoref{trees-to-partial-orders}), 
Todorcevic showed that partition relations for partially ordered sets in general can be reduced to 
the corresponding partition relations for trees.
Furthermore, as Todorcevic writes in \cite[p.~13]{Stevo-PRPOS}, 
\begin{quote}
It turns out that partition relations for trees are very natural generalizations of partition relations for cardinals 
and that several well-known partition relations for cardinals are straightforward consequences 
of the corresponding relations for trees.
\end{quote}
This motivates our continuing of Todorcevic's study of the partition calculus 
for trees.


\subsection{Nonspecial Trees and Todorcevic's Paradigm Shift}

The systematic study of set-theoretic trees was pioneered by \DJ uro Kurepa in the 1930s~\cite{Kurepa-thesis},
in the context of examining Souslin's Problem.%
\footnote{See 
Todorcevic's description of Kurepa's work on trees in~\cite[pp.~6--11]{Kurepa-selected},
as well as the survey article~\cite{Stevo-TLOS} covering Kurepa's work and related material.}
Kurepa showed 
that Aronszajn trees can be constructed without assuming any special axioms,
but the existence of a Souslin tree is equivalent to the failure of Souslin's Hypothesis.
When constructing an Aronszajn tree, a natural question to ask 
is whether the tree is Souslin.
Kurepa observed that an Aronszajn tree may fail to be Souslin for a very \emph{special} reason:
it may be able to be written as a union of countably many antichains.

We now know that Souslin's Problem is independent of the usual ZFC axioms.
In particular, 
Baumgartner, Malitz, and Reinhardt showed~\cite[Theorem~4]{BMR} that assuming MA$_{\aleph_1}$,
not only are there no Souslin trees, but every Aronszajn tree is special.
This may give the impression that nonspecial trees are somewhat pathological.
However, this is only because until now we have restricted our attention to Aronszajn trees,
so that our understanding of special and nonspecial trees in somewhat incomplete.

In Kurepa's work on trees, motivated by the quest to resolve Souslin's Problem, 
the main classification of trees was by their width~\cite[\S8.A.11, pp.~75--76]{Kurepa-thesis} 
\cite[pp.~71--72]{Kurepa-selected},
with a special focus on Aronszajn trees.
So 
the distinction between special and nonspecial was generally considered (by Kurepa and his successors)
only for Aronszajn trees.

But being Aronszajn is mainly a condition on the width of the tree,
the cardinality of its levels;
being special or non-special is a distinction in the number of its antichains,
in some sense related to 
the height of the tree.
We can consider one without the other.

It was Stevo Todorcevic who pioneered the systematic study of nonspecial trees without regard to their width, 
in his early work in the late 1970s~\cite{Stevo-SSTC, Stevo-PRPOS}. 
With this paradigm shift, 
he was the first to properly understand the notion of 
nonspecial trees and put it into the right context inside the whole set theory.  
We can forget about trees being Aronszajn or Souslin, and simply define what it means for trees of height $\omega_1$
to be special or nonspecial, regardless of their width:
\footnote{Unfortunately,
it remains common \cite[Definition~III.5.16]{New-nen} \cite[p.~117]{Jech} \cite[p.~41]{J-W}
to define special Aronszajn trees only, rather than defining special and nonspecial trees more broadly
as introduced by Todorcevic.}

\begin{definition}\cite[p.~250]{Stevo-SSTC}
A tree $T$ is a \emph{special tree} if it
can be written as a union of countably many antichains.
%
Otherwise, $T$ is a \emph{nonspecial tree}.
\end{definition}

In some sense the class of nonspecial trees
represents a natural generalization of the
first uncountable ordinal $\omega_1$, 
which in turn can be considered the simplest example of a nonspecial tree.
Todorcevic showed that many partition relations known to be true for $\omega_1$ 
are true for 
nonspecial trees as well.
And in contrast to our previous observation 
that
nonspecial \emph{Aronszajn} trees may not exist, 
Kurepa showed \cite[Theorem~1]{Kurepa-1956} \cite[p.~236]{Kurepa-selected}
that there \emph{does} exist a nonspecial tree with no uncountable chain, namely $w\mathbb Q$
(and its variant $\sigma \mathbb Q$),
the collection of all 
well-ordered subsets of $\mathbb Q$, ordered by end-extension.
Thus our 
generalization from $\omega_1$ to nonspecial trees is not vacuous.

We can examine a similar generalization for heights greater than $\omega_1$:


\begin{definition}\cite[p.~246]{Stevo-TLOS}, \cite[p.~4, p.~15ff.]{Stevo-PRPOS}
For any infinite cardinal $\kappa$,
a tree $T$ is a \emph{$\kappa$-special tree} if it
can be written as a union of $\leq \kappa$ many antichains.
Otherwise, $T$ is a \emph{non-$\kappa$-special tree}.
\end{definition}

Again,
the class of non-$\kappa$-special trees
represents a natural generalization of the
ordinal $\kappa^+$, 
which in turn can be considered the simplest example of a non-$\kappa$-special tree.
And again, 
Todorcevic showed that many partition relations known to be true for an arbitrary successor cardinal $\kappa^+$
are true for 
non-$\kappa$-special trees as well.

In \autoref{section:stationary} we will describe a new theory of stationary subtrees of a nonspecial tree.
We will define the diagonal union of subsets of a tree, as well as normal ideals on a tree, 
and we characterize arbitrary subsets of a non-special tree as being either stationary or non-stationary. 

In subsequent sections, we will
use this theory to prove the following partition relation for trees,
which is a generalization to trees of the 
Balanced Baumgartner-Hajnal-Todorcevic Theorem for cardinals~\cite[Theorem~3.1]{BHT}:

\begin{theorem}[Main Theorem]
\label{BHT-trees-regular}
\textofBHTtrees
\end{theorem}

%% file: notation_BHT_trees.tex
\section{Notation}
\label{Section:notation}

Our set-theoretic notation and terminology will generally follow standard conventions, 
such as in~\cite{EHMR, Jech, J-W, New-nen, Stevo-TLOS, Williams}.
For clarity and definiteness, and in some cases to resolve conflicts between the various texts, we state the following:

For cardinals $\nu$ and $\kappa$, where $\nu \geq 2$ and $\kappa$ is infinite, we define%
\footnote{
Some older texts use 
$\nu^{\underset{\smile}{\kappa}}$ instead of $\nu^{<\kappa}$,
such as~\cite{EHMR}, \cite{Stevo-TLOS}, and~\cite{Stevo-PRPOS}.}
\[
\nu^{<\kappa} = \sup_{\mu < \kappa} \nu^\mu, 
\]
where the exponentiation is \emph{cardinal exponentiation},
and the supremum is taken over \emph{cardinals} $\mu < \kappa$.

Following~\cite{BHT}, 
we define\footnote{In~\cite{EHMR}, this would be denoted $L_3(\kappa)$.} 
$\log \kappa$ (for an infinite cardinal $\kappa$) to be the smallest cardinal $\tau$ such that $2^\tau \geq \kappa$.
So for any ordinal $\xi$, we have
\[
\xi < \log \kappa \iff 2^{\left|\xi\right|} < \kappa \iff m^{\left|\xi\right|} < \kappa \text{ for any finite } m,
\]
and in particular, the hypothesis on $\xi$ in the Main Theorem~\ref{BHT-trees-regular} can be stated as
$\xi < \log \kappa$.

If $\mathcal A \subseteq \mathcal P(X)$ is any set algebra (field of sets) over some set $X$,
then we follow the convention in~\cite[p.~171, Definition~29.5(i),(ii)]{EHMR}, 
\cite[Section~13.1]{J-W}, and~\cite{BHT}
that a sub-collection $I \subseteq \mathcal A$ can be an ideal in $\mathcal A$
even if $X \in I$ (so that $I = \mathcal A$).
If, in fact, $X \notin I$, then the ideal is called \emph{proper}.
A similar allowance is made in the definition of a filter.
This will allow us to define ideals and their corresponding filters without verifying that they are proper.

We will always assume $T$ is a tree with order relation $<_T$.

Following~\cite[p.~239]{Stevo-TLOS}, 
``Every subset of a tree $T$ will also be considered as a \emph{subtree} of $T$."
This is also as in~\cite[p.~27]{J-W}.
That is,
unlike in~\cite[Definition~III.5.3]{New-nen}, 
we do not require our subtrees to be \emph{downward closed}.

We use \emph{node} as a synonym for element of a tree,
following~\cite[p.~27]{J-W}, \cite[p.~204]{New-nen}, and implicitly~\cite[p.~244]{Jech},
but unlike~\cite[p.~240]{Stevo-TLOS} where \emph{node} has a different meaning.

For any tree $T$, a \emph{limit node} of $T$ is a node whose height is a limit ordinal,%
\footnote{Limit nodes correspond to the limit points of $T$, when we give $T$ the \emph{tree topology},
as we describe later in footnote~\ref{tree-topology}.}
while a \emph{successor node} is one whose height is a successor ordinal.

Following Kunen's notation
in~\cite[Definition~III.5.1]{New-nen},
we will use $\pred{t}$ (rather than $\hat t$ or pred$(t)$ or pr$(t)$) for the set of predecessors of the node
$t \in T$,
and $\cone t$ (rather than $T^t$) for the cone above $t$.
When discussing diagonal unions, it will be crucial that $\cone t$ be defined so as \emph{not} to include $t$.
\todo{
Make sure this is used consistently.  
}
However, as we will see later, 
it will be convenient to make an exception for the cone above the root node $\emptyset$,
to allow the root to be in the ``cone above'' \emph{some} node.%
\footnote{
Similar to~\cite[p.~8, footnote~1]{BTW}, the root node is ``an annoyance when dealing with diagonal unions''.  
}

Our notation for partition relations on trees (and on partially ordered sets in general) is based on~\cite{Stevo-PRPOS},
which generalizes the usual Erd\H os-Rado notation for linear orders as follows:

Suppose $\left<P, <_P \right>$ is any partial order.
If $\alpha$ is any ordinal, 
we write $[P]^\alpha$ to denote the set of all linearly ordered chains in $P$ of order-type $\alpha$.
If $\mu$ is any cardinal and $\alpha$ is any ordinal, the statement
\[
P \to \left(\alpha\right)^2_\mu
\]
means:
For any colouring (partition function) $c : [P]^2 \to \mu$, 
there is a chain $X \in [P]^\alpha$ 
%
%
that is $c$-homogeneous, 
that is, $c''[X]^2  = \{\chi\} 
$ for some colour $\chi<\mu$.

If $T$ is a tree and $c: [T]^2 \to \mu$ is a colouring, 
where $\mu$ is some cardinal, 
and $\chi < \mu$ is some ordinal (colour), and $t \in T$,
we define
\[
c_\chi(t) = \left\{ s <_T  t : c\{s, t\} = \chi \right\} \subseteq \pred t.
\]

For two subsets $A, B \subseteq T$, we will write $A <_T B$ to mean: 
for all $a \in A$ and $b \in B$ we have $a <_T b$.
In that case, the set $A \otimes B$ denotes
\todo{Check notation for consistency with
other papers and other chapters of thesis, especially $\otimes$.}
\[
\left\{ \{a, b\} : a \in A, b \in B \right\},
\]
which is a subset of $[T]^2$.

%% file: BHT_generalization_to_trees.tex


\section{Balanced Baumgartner-Hajnal-Todorcevic Theorem for Trees:
		Background and Motivation}


The remainder of this paper 
is devoted to our exposition of the Main Theorem, Theorem~\ref{BHT-trees-regular}. 
%
%
%
%
%
The Main Theorem 
is a generalization to trees of 
the Balanced Baumgartner-Hajnal-Todorcevic Theorem,%
\footnote{Jean Larson refers to it by that name in~\cite[p.~312, p.~326]{Larson-sets-extensions}.} 
which we recover by applying the Main Theorem~\ref{BHT-trees-regular} to the cardinal $(2^{<\kappa})^+$,
which is the simplest example of a non-$(2^{<\kappa})$-special tree:

\begin{corollary}[Balanced Baumgartner-Hajnal-Todorcevic Theorem] \cite[Theorem~3.1]{BHT} \label{BHT-cards}
Let $\kappa$ be any infinite regular cardinal.
Then for any ordinal $\xi$ such that $2^{\left|\xi\right|} < \kappa$, and any natural number $k$, we have
\[
\left(2^{<\kappa}\right)^+ \to \left(\kappa + \xi \right)^2_k.
\]
\end{corollary}


The case of the Main Theorem~\ref{BHT-trees-regular} where $k=2$ was proven by Todorcevic
in~\cite[Theorem~2]{Stevo-PRPOS}.
This was a generalization to trees of the corresponding result for cardinals by Shelah~\cite[Theorem~6.1]{Shelah}.

The Main Theorem~\ref{BHT-trees-regular} is a partial strengthening of the following result of Todorcevic,
which is itself a generalization to trees of the balanced Erd\H os-Rado Theorem for pairs:
\todo{Source this to the 1956 E-R paper, or is the general version first recorded in~\protect\cite{EHMR}?}

\begin{theorem}\label{ER-trees-Stevo}\cite[Corollary~25]{Stevo-PRPOS}
Let $\kappa$ be any infinite 
cardinal.
Then for any cardinal $\mu < \cf(\kappa)$, we have
\[
\text{non-$\left(2^{<\kappa}\right)$-special tree } \to \left(\kappa + 1 \right)^2_\mu.
\]
\end{theorem}

The Main Theorem~\ref{BHT-trees-regular} strengthens the result of Theorem~\ref{ER-trees-Stevo}
in the sense of providing a longer ordinal goal:  
$\kappa + \xi$ (for 
$\xi < \log\kappa$) instead of $\kappa + 1$.
However, this comes at a cost:
While Theorem~\ref{ER-trees-Stevo} applies to any infinite cardinal $\kappa$, 
the Main Theorem~\ref{BHT-trees-regular} applies to regular cardinals only
(see \autoref{section:singular-discussion} for discussion of the singular case);
and while Theorem~\ref{ER-trees-Stevo} allows any number of colours less than $\cf(\kappa)$,
the colourings in the Main Theorem~\ref{BHT-trees-regular} must be finite.

One of the main tools we will use in our proof of the Main Theorem~\ref{BHT-trees-regular} is
the technique of \emph{non-reflecting ideals determined by elementary submodels}.
This technique was introduced in~\cite{BHT}, 
where in Sections~1--3 it is used to prove the Balanced Baumgartner-Hajnal-Todorcevic Theorem
(our Corollary~\ref{BHT-cards}).
Those sections of~\cite{BHT} are reproduced almost verbatim in~\cite{Baumgartner}.
The basics of the technique are exposed in~\cite[Section~2]{Milner-esm},
and the method is developed in~\cite[Sections~3 and~4]{Hajnal-Larson-PR-Handbook}.
Some history of this technique is described in~\cite[pp.~312--313]{Larson-sets-extensions}.
In our \autoref{section:ideals-esm} below, 
we will 
explain the technique in detail,
while developing a more general form that works for trees rather than cardinals.



\section{Limitations, Conjectures, and Open Questions}

What possibilities are there for further extensions of the Main Theorem~\ref{BHT-trees-regular}?

First of all, we notice that the hypothesis that the tree is non-$(2^{<\kappa})$-special is necessary,
due to the combination of the following two theorems:

\begin{theorem}\label{neg-Galvin}
Let $\kappa$ be any infinite cardinal.
If $T$ is any
$\kappa$-special tree, then
\[
T \not\to 
\left( \kappa+1, \omega \right)^2.
\]
\end{theorem}

This theorem is a generalization to trees of the relation for
ordinals given in~\cite[Theorem~7.1.5]{Williams}.  The special case
where $\kappa = \omega$ is given in~\cite[Theorem~7]{Galvin}.

\begin{proof}
Let $f : T \to \kappa$ be a $\kappa$-specializing map 
for $T$.  So for any
$\left\{x,y\right\} \in [T]^2$, we clearly have $f(x) \neq f(y)$.
Define a colouring
\[
g: [T]^2 \to 2 = \{0,1\}
\]
by setting, for $\left\{x, y\right\} \in [T]^2$ with $x <_T y$,
\[
g (\left\{x, y\right\}) =
    \begin{cases}
        0 &\text{if $f(x) < f(y)$;} \\
        1 &\text{if $f(x) > f(y)$.}
    \end{cases}
\]

Suppose $A \subseteq T$ is a $0$-homogeneous chain for $g$.  Then
$\left<f(x):x \in A\right>$ is a sequence in $\kappa$ of the same order
type as $A$, so the order type of $A$ cannot be greater than
$\kappa$.

Suppose $B \subseteq T$ is a $1$-homogeneous chain for $g$.  Then
$\left<f(x):x \in B\right>$ is a decreasing sequence of ordinals, 
so $B$ cannot be infinite.
\end{proof}

\begin{theorem}\label{Sierpinski-trees}
Let $\kappa$ be any infinite regular cardinal, and suppose $2^{<\kappa} > \kappa$.
If $T$ is any $(2^{<\kappa})$-special tree, then
\[
T \not\to \left(\kappa\right)^2_2
\]

\begin{proof}
Since $2^{<\kappa} > \kappa$, we can find some $\mu < \kappa$ such that $2^\mu > \kappa$.
Using a Sierpinski partition \cite[bottom of p.~108]{EHMR} \cite[Lemma~9.4]{Jech} \cite[Theorem~15.12]{J-W},
we have
\[
2^\mu \not\to \left(\mu^+\right)^2_2.
\]
But $\kappa < 2^\mu$ and $\mu^+ \leq \kappa$, so this implies $\kappa \not\to (\kappa)^2_2$.
Then, since $\kappa$ is regular, \cite[Corollary~21.5(iii)]{EHMR} ensures that
\[
2^{<\kappa} \not\to \left(\kappa\right)^2_2.
\]
Combining a colouring $c : [2^{<\kappa}]^2 \to 2$ witnessing this last negative partition relation
with a specializing map $f : [T] \to 2^{<\kappa}$ in the obvious way
induces a colouring $c' : [T]^2 \to 2$ with the desired properties.
\end{proof}
\end{theorem}

\begin{corollary}\label{special-counterexample}
Let $\kappa$ be any infinite regular cardinal.
If $T$ is any $(2^{<\kappa})$-special tree, then
\[
T \not\to \left(\kappa +1, \kappa\right)^2.
\]

\begin{proof}
If $2^{<\kappa} = \kappa$, apply Theorem~\ref{neg-Galvin}.
Otherwise $2^{<\kappa} > \kappa$, so apply Theorem~\ref{Sierpinski-trees}.
\end{proof}
\end{corollary}

Next we consider:
Is there any hope of extending the ordinal goals beyond the ordinal $\kappa + \xi$
(where $\xi < \log \kappa$
) of the Main Theorem~\ref{BHT-trees-regular}?
Can we get a homogeneous chain of order-type $\kappa + \log \kappa$?
Alternatively, can we somehow combine the ordinal goals of the Main Theorem~\ref{BHT-trees-regular}
with the infinite number of colours in Theorem~\ref{ER-trees-Stevo}?

When $\kappa = \aleph_0$, both the Main Theorem~\ref{BHT-trees-regular} and Theorem~\ref{ER-trees-Stevo} are
subsumed by a stronger result of Todorcevic:
\begin{theorem}\cite[Theorem~1]{Stevo-PRPOS}\label{Stevo-trees-BH}
For 
all $\alpha<\omega_1$ and $k<\omega$ we have
\[
\text{nonspecial tree }  \to \left(\alpha\right)^2_k.
\]
\end{theorem}

(This itself is a generalization to trees of an earlier result of Baumgartner and Hajnal~\cite{B-H} 
for cardinals.)

What about uncountable values of $\kappa$?

The following theorem collects various results from~\cite[Section~3]{Kanamori-partitions}
that limit the possible extensions of our Main Theorem~\ref{BHT-trees-regular} 
that we can hope to prove without any special axioms:

\begin{theorem}\label{limits-from-Kanamori}
If $V = L$, then:%
\footnote{The negative partition relations proved in~\cite{Kanamori-partitions} are actually stronger
(the notation follows~\cite[p.~153]{Kanamori-partitions}):

For part~(1), we have
\[
\kappa^{+} \not\to \left[ \kappa : \log\kappa \right]^2_{\kappa},
\]
discussed for successor cardinals in~\cite[p.~161]{Kanamori-partitions} (this result was proven by Joseph Rebholz) 
and for inaccessible cardinals that are not weakly compact in~\cite[Theorem~3.6]{Kanamori-partitions} 
(this result was proven by Hans-Dieter Donder).

For part~(2), the result for successor cardinals is
\[
\kappa^{+} \not\to \left(\kappa : 2 \right)^2_{\log\kappa},
\]
given in~\cite[top of p.~163]{Kanamori-partitions}.
For a limit cardinal we have $\log\kappa = \kappa$, so the result follows from a Sierpinski relation
$2^\kappa \not\to (3)^2_\kappa$.}
\begin{enumerate}
\item If $\kappa$ is any regular uncountable cardinal that is not weakly compact, then
\[
\kappa^+ \not\to \left(\kappa + \log \kappa \right)^2_2.
\]
\item For any infinite 
cardinal $\kappa$, we have
\[
\kappa^+ \not\to \left(\kappa + 2 \right)^2_{\log\kappa}.
\]
\end{enumerate}
\end{theorem}

Recall that $V=L$ implies GCH, which in turn implies:
\begin{itemize}
\item $2^{<\kappa} = \kappa$;
\item $\log \kappa = \kappa^-$ for a successor cardinal $\kappa$
(where $\kappa^-$ is the cardinal $\mu$ such that $\mu^+ = \kappa$);
\item $\log \kappa = \kappa$ for a limit cardinal $\kappa$.
\todo{Decide whether we need these last two.}
\end{itemize}

So part~(1) of Theorem~\ref{limits-from-Kanamori} shows that 
(for regular uncountable cardinals that are not weakly compact) we cannot extend the ordinal goals of
the Main Theorem~\ref{BHT-trees-regular} without any special axioms.
That is, just as Corollary~\ref{BHT-cards} is described in~\cite[p.~142]{Hajnal-Larson-PR-Handbook},
our Main Theorem~\ref{BHT-trees-regular} is the best possible balanced generalization to trees of the
Erd\H os-Rado Theorem for finitely many colours to ordinal goals.

Furthermore, 
part~(2) of Theorem~\ref{limits-from-Kanamori} shows that
for successor cardinals $\kappa$ (where $\log \kappa < \kappa = \cf(\kappa)$)
we cannot 
combine the ordinal goals of the Main Theorem~\ref{BHT-trees-regular}
with the larger number of colours in Theorem~\ref{ER-trees-Stevo}.

This leaves open the following questions:

\begin{question}
For a regular cardinal $\kappa > \aleph_1$,
do we have%
\footnote{According to~\cite[top of p.~143]{Hajnal-Larson-PR-Handbook},
the only known result in this direction 
is a result of Shelah~\cite{Shelah-strongly-compact}
that
\[
\left(2^{<\kappa}\right)^+ \to \left(\kappa + \mu\right)^2_\mu
\]
(for regular $\kappa$) with the assumption that there exists a strongly compact cardinal $\sigma$
such that $\mu < \sigma \leq \kappa$.
We conjecture that this result is true when generalized to trees as well.}
\[
\text{non-$\left(2^{<\kappa}\right)$-special tree } \to \left(\kappa + \xi \right)^2_{\mu} 
\]
for $\xi, \mu < \log \kappa$ (or even $\mu^+ < \kappa$)?
\end{question}

The simplest case of the above question is when $\mu = \aleph_0$ and $\kappa = \aleph_2$, 
so that we ask:  Does 
\[
\left(2^{\aleph_1}\right)^+ \to \left(\omega_2 + 2 \right)^2_{\aleph_0}?
\]

We conjecture a \emph{yes} answer to the following question, 
generalizing the corresponding conjecture for cardinals in~\cite[p.~156]{Kanamori-partitions}:

\begin{question}
If $\kappa$ is a weakly compact cardinal, do we have,
for every $\alpha < \kappa^+$ and $n<\omega$,
\begin{align*}
\text{non-$\kappa$-special tree }	&\to \left(\alpha\right)^2_n ?
									\\
\text{non-$\kappa$-special tree }	&\to \left(\kappa^n\right)^2_{\aleph_0}?
\end{align*}
\end{question}

What about aiming for positive consistency results by avoiding $V=L$,
which caused the limitations in Theorem~\ref{limits-from-Kanamori} above?

For any fixed uncountable cardinal $\kappa$, if $2^{<\kappa} = 2^\kappa$, 
then applying Theorem~\ref{ER-trees-Stevo} to the cardinal $\kappa^+$ instead of $\kappa$
subsumes any extensions of the ordinal goals or number of colours that we would anticipate when applying our
Main Theorem~\ref{BHT-trees-regular} to $\kappa$.

So the question remains:

\begin{question}
Are any extensions of the Main Theorem~\ref{BHT-trees-regular} 
that are precluded when $V=L$ by Theorem~\ref{limits-from-Kanamori} consistent with $2^{<\kappa} < 2^\kappa$?%
\footnote{Some partial results in this direction are presented in~\cite[Section~2]{Kanamori-partitions},
and we conjecture that they can be generalized to trees.}
\end{question}

Singular cardinals are beyond the scope of this discussion.
In \autoref{section:singular-discussion} 
we will explain how our method of proof does not provide any results for singular cardinals.

\section{Examples}

Let us consider some examples of regular cardinals $\kappa$, 
to see what the Main Theorem~\ref{BHT-trees-regular} gives us in each case:

\begin{example}
Suppose 
$\kappa = \aleph_0$.  Then $2^{<\kappa} = \aleph_0$, and we have,
for 
any natural numbers $k$ and $n$,
\[
\text{nonspecial tree } \to \left(\omega + n\right)^2_k.
\]
\end{example}

Notice that our proof remains valid in this case; 
nowhere in the proof of the Main Theorem~\ref{BHT-trees-regular} do we require $\kappa$ to be uncountable.
However, 
as we have mentioned earlier, 
this case is already subsumed by the stronger Theorem~\ref{Stevo-trees-BH} of Todorcevic.


So we focus on uncountable values of $\kappa$.
The first case where we get something new is:

\begin{example}\label{example-main-aleph1}
Let $\kappa = \aleph_1$.  Then $2^{<\kappa} = \mathfrak c$, but
$\xi$ must still be finite, so we have,
for 
any natural numbers $k$ and $n$,
\[
\text{non-$\mathfrak c$-special tree } \to \left(\omega_1 + n\right)^2_k.
\]
\end{example}

Example~\ref{example-main-aleph1}
is the simplest example provided by the Main Theorem~\ref{BHT-trees-regular}. 
However, we can (consistently) strengthen Example~\ref{example-main-aleph1} by replacing $\omega_1$
with any regular cardinal $\kappa$ such that $2^{<\kappa} = \mathfrak c$.
For example:

\begin{example}
Suppose $\kappa = \mathfrak p$ (the pseudo-intersection number).  
Then by~\cite[Exercise~III.1.38]{New-nen}, $\kappa$ is regular, 
and by~\cite[Lemma~III.1.26]{New-nen},
$2^{<\kappa} = \mathfrak c$. 
So we have,
for 
any natural numbers $k$ and $n$,
\[
\text{non-$\mathfrak c$-special tree } \to \left(\mathfrak p + n\right)^2_k.
\]
\end{example}

\begin{example}
Setting $\kappa = \mathfrak c^+$, we have $2^{<\kappa} = 2^{\mathfrak c}$, so that 
for 
any ordinal
\footnote{In particular, $\xi$ can be any countable ordinal.
More generally, we know from~\cite[Lemma~III.1.26]{New-nen} that $\mathfrak p \leq \log (\mathfrak c^+)$,
so that any $\xi < \mathfrak p$ 
will work.}
$\xi < \log(\mathfrak c^+)$,
and any natural number $k$, we have
\[
\text{non-$2^{\mathfrak c}$-special tree }
 \to \left(\mathfrak c^+ + \xi\right)^2_k.
\]
\end{example}

If we assume CH, then Example~\ref{example-main-aleph1} 
becomes
(for finite $n$ and $k$)
\[
\text{non-$\aleph_1$-special tree } \to \left(\omega_1 + n\right)^2_k.
\]
If we further assume GCH, then $2^{<\kappa} = \kappa$, 
so the 
general statement of the Main Theorem~\ref{BHT-trees-regular} 
is simplified to
\[
\text{non-$\kappa$-special tree } \to \left(\kappa + \xi\right)^2_k,
\]
where the hypothesis $\xi < \log \kappa$ 
can be written as $\left|\xi\right|^+ <
\kappa$.  We will not assume these (or any) extra axioms in the proof of
the Main Theorem~\ref{BHT-trees-regular}, but if such assumptions%
\footnote{Older papers often assume GCH, or variants of it, when stating related results,
due to the lack of good notation for iterated exponentiation~\cite[p.~218]{Larson-sets-extensions}
and for the weak power $2^{<\kappa}$.
Even Section~1 of~\cite{BHT} and~\cite{Baumgartner} (though not the subsequent sections, 2 and~3)
unnecessarily assumes $2^{<\kappa} = \kappa$.}
help the reader's intuition in
following the proof then there is no harm in doing so.

\section{The Role of Regularity and Discussion of Singular Cardinals}
\label{section:singular-discussion}

In the statement of the Main Theorem (Theorem~\ref{BHT-trees-regular}),
why do we require $\kappa$ to be regular?
Where is the regularity of $\kappa$ used in the proof?
Furthermore, where in the proof do we use the fact that the tree is non-$(2^{<\kappa})$-special?

Suppose we fix \emph{any} infinite cardinal $\kappa$.
How tall must a non-special tree be in order to obtain the homogeneous sets in the 
conclusion of the Main Theorem~\ref{BHT-trees-regular}?

In order to apply the lemmas of \autoref{section:very-nice} 
(in particular, Lemmas~\ref{build-nice-collection}(3) and~\ref{eligible-and-complete}),
we will require our tree to be non-$\nu$-special, 
where $\nu$ is an infinite cardinal satisfying $\nu ^ {<\kappa} = \nu$. 
So we need to determine:
What is the smallest infinite cardinal $\nu$ for which $\nu ^ {<\kappa} = \nu$?
It is clear that we must have $\nu \geq 2^{<\kappa}$.
What happens if we set $\nu = 2^{<\kappa}$?

The following fact follows immediately from~\cite[Theorem~6.10(f), first case]{EHMR}:

\begin{theorem}\label{regular-2<kappa}
For any regular cardinal $\kappa$, we have
\[
\left(2^{<\kappa}\right)^{<\kappa} = 2^{<\kappa}.
\]
\end{theorem}

So for a \emph{regular} cardinal $\kappa$, 
we can set $\nu = 2^{<\kappa}$ to satisfy the requirement $\nu ^ {<\kappa} = \nu$, 
so that the non-$(2^{<\kappa})$-special tree in the statement of the Main Theorem~\ref{BHT-trees-regular}
is exactly what we need for the proof to work.

What about the case where $\kappa$ is a \emph{singular} cardinal?
It turns out that Theorem~\ref{regular-2<kappa} is the only consequence of regularity 
used in the proof of the Main Theorem~\ref{BHT-trees-regular}.
In fact, the proof of the Main Theorem~\ref{BHT-trees-regular} 
actually gives the following (apparently) more general version of it,
with weaker hypotheses:

\begin{theorem}
\label{BHT-trees-general}
Let $\nu$ and $\kappa$ be infinite cardinals such that $\nu ^ {<\kappa} = \nu$. 
Then for any ordinal $\xi$ such that $2^{\left|\xi\right|} < \kappa$, and any natural number $k$, we have
\[
\text{non-$\nu$-special tree } \to \left(\kappa + \xi \right)^2_k.
\]
\end{theorem}

For a singular cardinal $\kappa$,
we should be able to find some infinite cardinal $\nu$ satisfying the requirement $\nu ^ {<\kappa} = \nu$,
so that we can apply Theorem~\ref{BHT-trees-general} to a non-$\nu$-special tree for such $\nu$.
It would seem to be significant that the $\kappa$ in the conclusion does not need to be weakened to $\cf(\kappa)$.
It is tempting to conclude that we should present Theorem~\ref{BHT-trees-general} as our main result,
since it appears to have broader application than our Main Theorem~\ref{BHT-trees-regular}.

In particular,
depending on the values of the continuum function,
there may be some singular cardinals $\kappa$ for which the sequence
$\{2^\mu  : \mu < \kappa\}$ is eventually constant, 
in which case any such $\kappa$ would satisfy $\cf(2^{<\kappa}) \geq \kappa$ 
and $(2^{<\kappa})^{<\kappa} = 2^{<\kappa}$.
%
%
(Of course, this cannot happen under GCH.)
For such $\kappa$, we can apply Theorem~\ref{BHT-trees-general} with $\nu = 2^{<\kappa}$,
just as we do for regular cardinals.

However, it turns out that the singular case gives no new results, as we will see presently.

%

The following fact follows immediately from~\cite[Theorem~6.10(f), second case]{EHMR}:

\begin{theorem}\label{singular-nu^<kappa}
For any singular cardinal $\kappa$ and any cardinal $\nu \geq 2$, we have
\[
\left(\nu^{<\kappa}\right)^{<\kappa} = \nu^{\kappa}.
\]
\end{theorem}

Fixing any singular cardinal $\kappa$, 
suppose we choose some infinite cardinal $\nu$ satisfying $\nu ^ {<\kappa} = \nu$,
in order to 
apply Theorem~\ref{BHT-trees-general} to a non-$\nu$-special tree.
But then we also have (using Theorem~\ref{singular-nu^<kappa})
%
\[
2^{<\left(\kappa^+\right)} = 2^\kappa \leq \nu^\kappa = \left(\nu^{<\kappa}\right)^{<\kappa} = 
\nu^{<\kappa} = 
\nu,
\]
%
%
so that the non-$\nu$-special tree in Theorem~\ref{BHT-trees-general} 
is also non-$(2^{<(\kappa^+)})$-special.
Applying the Main Theorem~\ref{BHT-trees-regular} to the regular (successor) cardinal $\kappa^+$
gives us a longer homogeneous chain (of order-type $> \kappa^+$) 
than the one we get when applying Theorem~\ref{BHT-trees-general} to the original singular cardinal $\kappa$,
without 
requiring a taller tree.
Thus 
any result we can get by applying Theorem~\ref{BHT-trees-general} to a singular cardinal $\kappa$
is 
already subsumed 
by the Main Theorem~\ref{BHT-trees-regular}.

This explains how our Main Theorem~\ref{BHT-trees-regular} is the optimal statement of the result;
nothing is gained by attempting to state a more general result that includes singular cardinals.

\section{
From Trees to Partial Orders}
\label{trees-to-partial-orders}

In this section we derive a corollary of our Main Theorem~\ref{BHT-trees-regular},
using a result of Todorcevic~\cite[Section~1]{Stevo-PRPOS} that states that
partition relations for nonspecial trees imply corresponding partition relations for partially ordered sets in general.

First, we outline the main result of~\cite[Section~1]{Stevo-PRPOS}:


\begin{theorem} \label{tree-to-partial-order}
Let $r$ be any positive integer, let $\kappa$ and $\theta$ be cardinals, 
and for each $\gamma < \theta$ let $\alpha_\gamma$ be an ordinal.
If every non-$\kappa$-special tree $T$ satisfies
\[
T \to \left(\alpha_\gamma\right)^r_{\gamma < \theta},	\tag{$**$}	\label{general-partition}
\]
then every partial order $P$ satisfying $P \to (\kappa)^1_\kappa$ also satisfies 
the above partition relation~\eqref{general-partition}.

\begin{proof}
Suppose $\left<P, <_P \right>$ is any partial order satisfying $P \to (\kappa)^1_\kappa$.
Let $\sigma'P$ be the set of well-ordered chains of $P$ with a maximal element,
ordered by 
end-extension ($\sqsubseteq$).


Since $P \to (\kappa)^1_\kappa$, \cite[Theorem~9]{Stevo-PRPOS} tells us that 
$\sigma'P$ is not the union of $\kappa$ antichains.
But $\sigma'P$ is clearly a tree (as it is a collection of well-ordered sets, ordered by end-extension),
so this means that $\sigma'P$ is a non-$\kappa$-special tree.  
By the hypothesis of our theorem it follows that $\sigma'P$ satisfies~\eqref{general-partition}.

We now define a function $f : \sigma'P \to P$ by setting, for each $a \in \sigma'P$,
\[
f(a) = \max (a).
\]
It is clear that
\[
f : \left< \sigma'P, \sqsubseteq \right> \to \left< P, <_P \right>
\]
is an order-homomorphism.%
\footnote{\label{embeddable}%
This is often described by saying ``$\left< \sigma'P, \sqsubseteq \right>$ is $\left<P, <_P \right>$-embeddable",
but this is an unfortunate use of the term \emph{embeddable}, as we do not require $f$ to be injective,
so that it is not an embedding in the usual sense.}
\todo{In thesis, move footnote~\ref{embeddable} to introduction, and restore reference to it from here.}%


Since $\sigma'P$ satisfies~\eqref{general-partition} 
and $\left<\sigma'P, \sqsubseteq \right>$ is $\left<P, <_P\right>$-embeddable,
it follows from~\cite[Lemma~1]{Stevo-PRPOS} that $P$ satisfies~\eqref{general-partition} as well.
This is what we needed to show.
\end{proof}
\end{theorem}

\begin{theorem}
\textofCorBHTpo

\begin{proof}
Apply Theorem~\ref{tree-to-partial-order} to the Main Theorem~\ref{BHT-trees-regular}.
\end{proof}
\end{theorem}

\section{Non-Reflecting Ideals Determined by Elementary Submodels}
\label{section:ideals-esm}

In this section, we consider a fixed tree $T$ and a regular cardinal $\theta$ such that $T \in H(\theta)$.
We will consider elementary submodels $N \prec H(\theta)$ with $T \in N$, and 
use them to create certain algebraic structures on $T$.
Ultimately, for some nodes $t \in T$, we will use models $N$ to define ideals on $\pred t$.
We make no assumptions about the height of the tree $T$ at this point.

\begin{lemma}\label{N-algebra}
Suppose $N \prec H(\theta)$ is an elementary submodel such that
$T \in N$.  Then the collection $\mathcal P(T) \cap N$
is a field of sets (set algebra) over the set $T$.


\begin{proof}
The collection $\mathcal P (T) \cap N$ is clearly a collection
of subsets of $T$.  Furthermore:
\begin{description}
\item[Nonempty]
Clearly $\emptyset \in \mathcal P (T) \cap N$.
\item[Complements] Since $T \in N$, by
elementarity of $N$ it follows that $T \setminus B \in N$ for
any $B \in N$.
\item[Finite unions] Suppose $A, B \in \mathcal P(T) \cap N$.
The set $A \cup B$ is definable from $A$ and $B$, so by elementarity
of $N$ we have $A \cup B \in N$.  A union of subsets of $T$ is
certainly a subset of $T$, so we have $A \cup B \in \mathcal
P(T) \cap N$, as required to show that $\mathcal P(T)
\cap N$ is a set algebra.
%
%
%
\qedhere
\end{description}
\end{proof}
\end{lemma}

\begin{lemma}\label{ultrafilter-on-tree}
Suppose $N \prec H(\theta)$ is an elementary submodel such that $T \in N$, and let $t \in T$.
Then the collection
\[
\left\{ B \subseteq T : B \in N \text{ and } t \in B \right\}
\]
is an ultrafilter in the set algebra $\mathcal P(T) \cap N$,
and the collection
\[
\left\{ B \subseteq T : B \in N \text{ and } t \notin B \right\}
\]
is the corresponding maximal (proper) ideal in the same set algebra.
\end{lemma}

What we really want are algebraic structures on $\pred t$ determined by $N$.
So we now consider what happens when we 
intersect members of $N$ with $\pred t$:


\begin{definition}
Suppose $N \prec H(\theta)$ is an elementary submodel such that $T \in N$, and let $t \in T$.
Define a collapsing function
\[
\pi_{N,t} : \mathcal P(T) \cap N \to \mathcal P(\pred t)
\]
by setting, for $B \subseteq T$ with $B \in N$,
\[
\pi_{N,t} (B) = B \cap \pred t.
\]
We then define the collection
\[
\mathcal A_{N,t}
    = \range\left(\pi_{N,t}\right)
    = \left\{B \cap \pred t : B \in \mathcal P(T) \cap N \right\}
    = \left\{ B \cap \pred t : B \in N \right\}
    \subseteq \mathcal P\left(\pred t\right).
\]
\end{definition}

\begin{lemma}\label{homomorphism-set-algebra}
Suppose $N \prec H(\theta)$ is an elementary submodel such that $T \in N$, and let $t \in T$.

Then the collection $\mathcal A_{N,t}$ is a set algebra over the set
$\pred t$, and the collapsing function $\pi_{N,t}$ defines a
surjective homomorphism
of set algebras
\[
\pi_{N,t} : \left< \mathcal P (T) \cap N, \cup, \cap, \setminus, \emptyset, T \right> 
			\to 
		\left< \mathcal A_{N,t}, 
						\cup, \cap, \setminus, \emptyset, \pred t \right>.
\]


\begin{proof}
We will show that $\pi_{N,t} : \mathcal P(T) \cap N \to \mathcal P(\pred t) 
$ preserves the set-algebra 
operations:

We have $\pi_{N,t} (\emptyset) = \emptyset \cap \pred t = \emptyset$, and
$\pi_{N,t} (T) = T \cap \pred t = \pred t$, as required.

For every $B \in \mathcal P(T) \cap N$, we have
\begin{align*}
\pi_{N,t} (T \setminus B)
    &= (T \setminus B) \cap \pred t            \\
    &= (T \cap \pred t) \setminus (B \cap \pred t)             \\
    &= \pred t \setminus \pi_{N,t} (B),
\end{align*}
showing that $\pi_{N,t}$ preserves complements.

Let $\mathcal C \subseteq \mathcal P(T) \cap N$ be any
collection whose union is in $N$.  We then have
\begin{align*}
\pi_{N,t} \left( \bigcup_{B \in \mathcal C} B \right)
    &= \left( \bigcup_{B \in \mathcal C} B \right) \cap \pred t  \\
    &= \bigcup_{B \in \mathcal C} \left(B \cap \pred t \right)   \\
    &= \bigcup_{B \in \mathcal C} \pi_{N,t} (B),
\end{align*}
so that $\pi_{N,t}$ preserves unions.

Preservation of intersections follows from preservation of
complements and unions, using De Morgan's laws.

So we have shown that $\pi_{N,t} 
$ respects the set-algebra operations 
(there are no relations in the set-algebra structure; only constants
and functions), and is therefore a homo
morphism of set algebras onto
its range, which is $\mathcal A_{N,t}$.

The homomorphic image of a set algebra 
(where the range is a subset of a power-set algebra, with the usual set-theoretic operations) is a set algebra,
so it follows that $\mathcal A_{N,t}$ 
is a set algebra over the set 
$\pred t$.
%
\end{proof}
\end{lemma}

\begin{definition}
Suppose $N \prec H(\theta)$ is an elementary submodel such that $T \in N$, and let $t \in T$.
We define the collections
\begin{align*}
\mathcal G_{N,t} 	&= \left\{ \pi_{N,t}(A) : A \in \mathcal P(T) \cap N \text{ and } t \notin A \right\} 	\\
			&= \left\{ A \cap \pred t : A \in N \text{ and } t \notin A \right\}
			\subseteq \mathcal A_{N,t}, \text{ and }							\\
\mathcal G^*_{N,t} 	&= \left\{ \pi_{N,t}(B) : B \in \mathcal P(T) \cap N \text{ and } t \in B \right\} 	\\
			&= \left\{ B \cap \pred t : B \in N \text{ and } t \in B \right\}
			\subseteq \mathcal A_{N,t}.
\end{align*}

\end{definition}

\begin{lemma}\label{G_N-ideal}
Suppose $N \prec H(\theta)$ is an elementary submodel such that $T \in N$, and let $t \in T$.
Then the collection $\mathcal G_{N,t}$ is a (not necessarily proper) ideal in the set algebra
$\mathcal A_{N,t}$,
and $\mathcal G^*_{N,t}$ is the dual filter corresponding to $\mathcal G_{N,t}$.


\begin{proof}
The collections $\mathcal G_{N,t}$ and $\mathcal G^*_{N,t}$ are (respectively) the homomorphic images
(under $\pi_{N,t}$) of the ideal and filter in the algebra $\mathcal P(T) \cap N$ given by Lemma~\ref{ultrafilter-on-tree}.
More explicitly:

It is clear that both $\mathcal G_{N,t}$ and $\mathcal G^*_{N,t}$ are subcollections of the set
algebra $\mathcal A_{N,t}$.  Furthermore:
\begin{description}
\item[Nonempty]
$\emptyset \in \mathcal G_{N,t}$, since $\emptyset = \emptyset \cap
\pred t$, $\emptyset \in N$, and $t \notin \emptyset$.

\item[Subsets] Suppose $X$ and $Y$ are both in $\mathcal A_{N,t}$,
with $X \subseteq Y$ and $Y \in \mathcal G_{N,t}$.  We need to show that
$X \in \mathcal G_{N,t}$.

Since $X$ and $Y$ are both in $\mathcal A_{N,t}$,
we have $X = \pi_{N,t}(A)$ and $Y = \pi_{N,t}(B)$ for some $A, B \in \mathcal P(T) \cap N$.
Furthermore, since $Y \in \mathcal G_{N,t}$, we can choose $B$ so that $t \notin B$.
We then have $A \cap B \in \mathcal P(T) \cap N$, and $t \notin A \cap B$, and
\[
\pi_{N,t}(A \cap B)  = \pi_{N,t}(A) \cap \pi_{N,t}(B) = X \cap Y = X,
\]
showing 
%
that $X \in \mathcal G_{N,t}$, as required.


\item[Finite unions]
Suppose $X, Y \in \mathcal G_{N,t}$.  
We can choose $A, B \in \mathcal P(T) \cap N$ with $t \notin A, B$ such that $\pi_{N,t}(A) = X$ and $\pi_{N,t}(B) = Y$.
We then have $A \cup B \in \mathcal P(T) \cap N$, and $t \notin A \cup B$, and
\[
\pi_{N,t}(A \cup B)  = \pi_{N,t}(A) \cup \pi_{N,t}(B) = X \cup Y,
\]
%
and it follows that $X \cup Y \in \mathcal G_{N,t}$.
%

\item[Dual filter]
Finally, for any $X \in \mathcal A_{N,t}$, 
we have 
\begin{align*}
X \in \mathcal G_{N,t} 
	&\iff X = \pi_{N,t}(A) \text{ for some $A \in \mathcal P(T) \cap N$ with } t \notin A 	\\
	&\iff \pred t \setminus X = \pi_{N,t}(T \setminus A) \text{ for some $A \in \mathcal P(T) \cap N$ with } t \notin A\\
	&\iff \pred t \setminus X = \pi_{N,t}(B) \text{ for some $B \in \mathcal P(T) \cap N$ with } t \in B 	\\
	&\iff \pred t \setminus X \in \mathcal G^*_{N,t},
\end{align*}
showing that $\mathcal G_{N,t}$ and $\mathcal G^*_{N,t}$ are dual to each other.
\qedhere
\end{description}
\end{proof}
\end{lemma}

In general, there is no reason to expect that the homomorphism $\pi_{N,t}$ is injective,
as there can be many different subsets of $T$ in the model $N$ that share the same intersection with $\pred t$.
As we will see later (see Remark~\ref{isomorphism-when-cardinal}),
this is the new difficulty that arises when generalizing these structures from cardinals to trees.
In particular, it may happen that $\pi_{N,t}$ collapses the algebra to the extent that $\mathcal G_{N,t}$,
which is the image of a maximal proper ideal, 
is equal to the whole algebra $\mathcal A_{N,t}$ rather than a proper ideal in it.
That is, there may be some $A \in \mathcal P(T) \cap N$ with $t \notin A$ but $\pi_{N,t}(A) = \pred t$.
More generally, there may be some $B \in N$ with $t \in B$, but $B \cap \pred t \in \mathcal G_{N,t}$
because it is equal to $A \cap \pred t$ for some $A \in N$ with $t \notin A$.
We will need to avoid such combinations of models and nodes, and we will show later how to do so.
In the mean time:

\begin{lemma}\label{G_N-dichotomy}
Suppose $N \prec H(\theta)$ is an elementary submodel such that $T \in N$, and let $t \in T$.
Then
\[
\mathcal G_{N,t} \cup \mathcal G^*_{N,t} = \mathcal A_{N,t},
\]
so that 
exactly one of the following two alternatives is true:
\begin{enumerate}
\item $\mathcal G_{N,t} = \mathcal G^*_{N,t} = \mathcal A_{N,t}$, or
\item $\mathcal G_{N,t}$ is a maximal proper ideal in $\mathcal A_{N,t}$,
and $\mathcal G^*_{N,t}$ is the corresponding ultrafilter, 
so that $\mathcal G_{N,t} \cap \mathcal G^*_{N,t} = \emptyset$.
\end{enumerate}

\begin{proof}
This follows from the fact that 
$\mathcal G_{N,t}$ and $\mathcal G^*_{N,t}$ are (respectively) the homomorphic images
(under $\pi_{N,t}$) of the maximal proper ideal and ultrafilter in the algebra $\mathcal P(T) \cap N$ 
given by Lemma~\ref{ultrafilter-on-tree}.
\end{proof}

\end{lemma}

For 
elementary submodels $N \prec H(\theta)$ such that $T \in N$, and nodes $t \in T$,
recall that $\mathcal A_{N,t} \subseteq \mathcal P(\pred t)$ is a set
algebra over the set $\pred t$, and we defined a certain
ideal $\mathcal G_{N,t} \subseteq \mathcal A_{N,t}$.  We now consider the
ideal on $\pred t$ (that is, the ideal in the whole power set
$\mathcal P (\pred t)$) generated by $\mathcal G_{N,t}$:

\begin{definition}
Suppose $N \prec H(\theta)$ is an elementary submodel such that $T \in N$, and let $t \in T$.
We define
\[
I_{N,t} = \left\{ X \subseteq \pred t : X \subseteq Y \text{ for some }
Y \in \mathcal G_{N,t} \right\}.
\]
\end{definition}
We explore the properties of $I_{N,t}$:

\begin{lemma}\label{I_N-ideal}
Suppose $N \prec H(\theta)$ is an elementary submodel such that $T \in N$, and let $t \in T$.
Then the collection $I_{N,t}$ is a (not necessarily proper) 
ideal on $\pred t$,
that is, an ideal in the whole power set $\mathcal P(\pred t)$.
\todo{Show that $I_{N,t}$ is not a maximal ideal (unlike $\mathcal G_{N,t}$).}


\begin{proof}Since $I_N$ consists of all subsets of sets in
$\mathcal G_N$, where by Lemma~\ref{G_N-ideal} $\mathcal G_N$ is an
ideal in the algebra $\mathcal A_N \subseteq \mathcal P(\pred t)$,
it is clear that $I_N$ is an ideal on $\pred t$.
\end{proof}
\end{lemma}

Although we have defined the ideal $I_{N,t}$, we will be more
interested in the corresponding co-ideal.
Recall that for any ideal $\mathcal I$ in a set algebra $\mathcal A
\subseteq \mathcal P(Z)$, we define:
\begin{align*}
\mathcal I^+ &= \mathcal A \setminus \mathcal I = \left\{ X \in \mathcal A : X \notin \mathcal I \right\} \\
\mathcal I^* &= \left\{ X \in \mathcal A : Z \setminus X \in \mathcal I \right\}
\end{align*}
$\mathcal I^*$ is called the \emph{filter dual to $\mathcal I$}.
$\mathcal I^+$ is called the \emph{co-ideal} corresponding to
$\mathcal I$, and sets in $\mathcal I^+$ are said to be
\emph{$\mathcal I$-positive}.

\begin{lemma}\label{I_N-formulations}
Suppose $N \prec H(\theta)$ is an elementary submodel such that $T \in N$, and let $t \in T$.
Then we have the following facts:
\begin{gather*}
\mathcal A_{N,t} \cap I_{N,t} = \mathcal G_{N,t}     		\\
\mathcal A_{N,t} \cap I^*_{N,t} = \mathcal G^*_{N,t}	\\
\mathcal A_{N,t} \cap I^+_{N,t} = \mathcal G^+_{N,t}
\end{gather*}
It follows that for all $B \in N$, we have:
\begin{gather*}
B \cap \pred t \in I_{N,t} \iff B \cap \pred t \in \mathcal G_{N,t}	\\
B \cap \pred t \in I^*_{N,t} \iff B \cap \pred t \in \mathcal G^*_{N,t}	\\
B \cap \pred t \in I^+_{N,t} \iff B \cap \pred t \in \mathcal G^+_{N,t}
\end{gather*}

Furthermore,
we can express the ideal, co-ideal and filter
as follows:
\todo{Decide on most intuitive formulations, and how to optimize them for later use.} 
\begin{align*}
I_{N,t} &= \left\{ X \subseteq \pred t : X \subseteq A \text{ for some } A \in N \text{ with } t \notin A
%
\right\}                   \\
%
I^+_{N,t} &= \left\{ X \subseteq \pred t : \forall A \in N
\left[X \subseteq A \implies t \in A \right] \right\} \\
I^+_{N,t} &= \left\{ X \subseteq \pred t : \forall B \in N
\left[t \in B \implies X \cap B \neq \emptyset \right] \right\} \\
I^*_{N,t} &= \left\{ X \subseteq \pred t : X \supseteq Y \text{ for
some $Y \in \mathcal G^*_{N,t}$} \right\}                      \\
I^*_{N,t} &= \left\{ X \subseteq \pred t : X \supseteq B \cap \pred t
\text{ for some $B \in N$ with } t \in B \right\}
\end{align*}

Finally, $I_{N,t}$ is a proper ideal (in $\mathcal P(\pred t)$) iff 
$\mathcal G_{N,t}$ is a proper ideal (in $\mathcal A_{N,t}$).

\end{lemma}

When considering the ideal $I_{N,t}$, we will generally want to have $\pred t \subseteq N$.
The following lemma explains why:


\begin{lemma}\label{bounded-in-ideal}
Suppose $N \prec H(\theta)$ is an elementary submodel such that $T \in N$, and let $t \in T$.
Then:
\begin{enumerate}
\item If $A \subseteq \pred t$ and $A \in N$, then $A \in
\mathcal G_{N,t}$.
\todo{Is this fact ever used other than for the subsequent parts of this lemma?}
    \setcounter{condition}{\value{enumi}}
\end{enumerate}
Furthermore, if, in addition to the previous hypotheses, we have $\pred t \subseteq N$, then:
\begin{enumerate}
    \setcounter{enumi}{\value{condition}}
\item If 
$s <_T t$, 
then $\pred s \in \mathcal G_{N,t}$ and 
$\pred t \setminus \pred s \in \mathcal G^*_{N,t}$.
\item
If 
$X \subseteq \pred t$ is not cofinal%
\footnote{For limit nodes $t$, \emph{not cofinal in $\pred t$} is equivalent to \emph{bounded below $t$}.
However, for successor nodes there is a distinction, 
as for a successor node $t$ every subset of $\pred t$ has an upper bound below $t$, 
namely the node $s$ such that $\pred t = \pred s \cup \{s\}$.
In fact the stronger statement is true, 
that every $X \subseteq \pred t$ that is \emph{bounded below $t$} is in $I_{N,t}$,
but this requires a separate proof for successor nodes, 
and successor nodes are made irrelevant by Lemma~\ref{eligible-implications} anyway.}
in $\pred t$, that is, $X
\subseteq \pred s$ for some 
$s \in \pred t$, then $X \in
I_{N,t}$.  Equivalently, any set in $I^+_{N,t}$ must be cofinal in
$\pred t$.
\todo{Maybe this is never used either?  Yes, used in Lemma~\protect\ref{easy-reflection}.}
\item 
For any set $Y \subseteq \pred t$ and any 
$s \in \pred t$, we have
\[
Y \in I^+_{N,t} \iff Y \setminus \pred s 
					\in I^+_{N,t}.
\]

\end{enumerate}

\begin{proof}\hfill
\begin{enumerate}
\item Since $A \subseteq \pred t$, we certainly have $t \notin
A$.  Then, since $A \in N$, we also have
$\pi_{N,t}(A) = A \cap \pred t = A$, 
so it follows that $A \in \mathcal G_{N,t}$.
\item
Since $\pred t \subseteq N$, 
any 
$s <_T t$ is in $N$, and $\pred s$ is defined from $s$ and $T$, which are both in $N$,
so by elementarity we have $\pred s \in N$.
Clearly, $\pred s \subseteq \pred t$, so (1) gives us 
$\pred s \in \mathcal G_{N,t}$.
Then, the corresponding filter set to $\pred s$ is $
\pred t \setminus \pred s$, so it follows that $\pred t \setminus \pred s \in \mathcal G^*_{N,t}$.
\item
We have $X \subseteq \pred s$, where by part (2) we know $\pred s \in
\mathcal G_{N,t}$.  It follows by definition of $I_{N,t}$ that $X \in I_{N,t}$.
\item 
From part (2) and Lemma~\ref{I_N-formulations}, 
we have $\pred s \in \mathcal G_{N,t} \subseteq I_{N,t}$.  
Then $Y$ is equivalent to $Y \setminus \pred s$ modulo a set from the ideal $I_{N,t}$.
%
\qedhere
\end{enumerate}
\end{proof}
\end{lemma}

As mentioned earlier, in order to ensure that our ideals are proper, we want to avoid situations where
there may be some $B \in 
N$ with $t \notin B$ but $B \cap \pred t = \pred t$.
We will therefore impose an \emph{eligibility condition}:

\begin{definition}
Suppose $W$ is any collection of sets.
We say that a node $t \in T$ is \emph{$W$-eligible} if
\[
\forall B \in W \left[ \pred t \subseteq B \implies t \in B \right].
\]
\end{definition}


When $W$ is an elementary submodel $N$, 
the eligibility condition can be formulated in several ways, in terms of our structures on $\pred t$.
Particularly useful among the following is condition (\ref{intersection-determines-t}),
which states that for an $N$-eligible node $t$ and any $X \in \mathcal A_{N,t}$,
we can determine whether or not $X \in \mathcal G_{N,t}$ by choosing a single $A \in N$ with $\pi_{N,t}(A) = X$
and checking whether or not $t \in A$, rather than having to check every such $A$.

\begin{lemma}\label{eligible-equivalent}
Suppose $N \prec H(\theta)$ is an elementary submodel such that $T \in N$, and let $t \in T$.
Then the following are all equivalent:
\begin{enumerate}
\item $t$ is $N$-eligible;
\item $\nexists A \in N \left[\pred t \subseteq A \subseteq T \setminus \{t\}\right]$;
\item $\pred t \notin \mathcal G_{N,t}$, that is, $\mathcal G_{N,t}$ is a proper ideal in $\mathcal A_{N,t}$;
\item $\mathcal G_{N,t} \cap \mathcal G^*_{N,t} = \emptyset$;
\item $\mathcal G_{N,t}$ is a maximal proper ideal in $\mathcal A_{N,t}$;
\item $\mathcal G^*_{N,t}$ is an ultrafilter in $\mathcal A_{N,t}$;
\item $\mathcal G^*_{N,t} = \mathcal G^+_{N,t}$;
\item $\pred t \notin I_{N,t}$, that is, $I_{N,t}$ is a proper ideal on $\pred t$;
\item $I^*_{N,t} \subseteq I^+_{N,t}$;
\item 	\label{intersection-determines-t}
For all $A, B \in N$ with $A \cap \pred t = B \cap \pred t$,
we have $t \in A \iff t \in B$
(even if $\pi_{N,t}$ is not injective);
\item For all $B \in N$, we have
\[
t \in B \iff B \cap \pred t \in \mathcal G^+_{N,t}.
\]
\end{enumerate}

\begin{proof}\hfill
\begin{description}
\item[(1) $\implies$ (2)] Clear.
\item[$\neg$(1) $\implies \neg$(2)]
If $B$ witnesses that $t$ is not $N$-eligible, then 
let $A = B \cap T$.  Since $T \in N$, by elementarity of $N$ we have $A \in N$, violating (2).
\item[(1) $\iff$ (3)] From the definition of $\mathcal G_{N,t}$.
\item[(3) $\iff$ (4)] These are always equivalent for any ideal.
\item[(3) $\iff$ (5)] From the dichotomy given by Lemma~\ref{G_N-dichotomy}.
\item[(5) $\iff$ (6) $\iff$ (7)] These are always equivalent for any ideal.
\item[(3) $\iff$ (8)] From the last sentence of Lemma~\ref{I_N-formulations}.
\item[(8) $\iff$ (9)] These are always equivalent for any ideal.
\item[$\neg$(10) $\implies \neg$(1)] 
If there were $A, B \in N$ with $A \cap \pred t = B \cap \pred t$ but $t \in A \setminus B$,
then $B \cup (T \setminus A)$ violates (1).
\item[(10) $\implies$ (11)]
The $\impliedby$ implication in (11) is always true by definition of $\mathcal G_{N,t}$.
If some $B \in N$ violates the $\implies$ implication,
then we would have $B \cap \pred t = A \cap \pred t$ for some $A \in N$ with $t \notin A$, violating (10).
\item [(11) $\implies$ (3)]
Apply the $\implies$ implication of (11) to $T$.
\qedhere
\end{description}
\end{proof}

\end{lemma}

The eligibility condition has other consequences that are not, in general, equivalent to it:

\begin{lemma}\label{eligible-implications}
Suppose $N \prec H(\theta)$ is an elementary submodel such that $T \in N$.
If $t \in T$ is $N$-eligible, then:
\begin{enumerate}
\item $\pred t \notin N$.
\item $t \notin N$.
\item $\height_T(t) \notin N$.
\item If we also have $\pred t \subseteq N$, then 
\[
\height_T(t) = \min\left\{ \delta : \delta \text{ is an ordinal and } \delta \notin N \right\},
\]
so that in particular
$t$ must be a limit node
in that case.
\end{enumerate}

\begin{proof}\hfill
\begin{enumerate}
\item If $\pred t \in N$, then $\pred t$ itself would violate the $N$-eligibility of $t$.
\item If $t \in N$, then by elementarity we have $\pred t \in N$, contradicting (1).
\item Define
\[
B = \left\{ s \in T : \height_T(s) < \height_T(t) \right\}.
\]
If $\height_T(t) \in N$ then by elementarity we would have $B \in N$.
But $\pred t \subseteq B$ and $t \notin B$,
violating the $N$-eligibility of $t$.
\item For any $\beta < \height_T(t)$, there must be some $s <_T t$ with $\height_T(s) = \beta$.
But then $s \in N$ (since by assumption $\pred t \subseteq N$), 
so by elementarity we have $\beta = \height_T(s) \in N$.
Then (3) gives the desired equation.

For any ordinal $\beta \in N$, its successor $\beta \cup \{\beta\} \in N$ by elementarity,
so the smallest ordinal not in the model must be a limit ordinal.
\qedhere
\end{enumerate}
\end{proof}
\end{lemma}

The following corollary follows immediately from Lemma~\ref{eligible-implications}(4):

\begin{corollary}
Suppose $N \prec H(\theta)$ is an elementary submodel such that $T \in N$.
If $s, t \in T$ are two $N$-eligible nodes such that $\pred s, \pred t \subseteq N$, 
then $\height_T(s) = \height_T(t)$.
Furthermore, if the set
\[
\left\{t \in T : t \text{ is $N$-eligible and } \pred t \subseteq N \right\}
\]
is nonempty, then its nodes are all at the same height, 
and that height is the ordinal $\min\left\{ \delta : \delta \notin N \right\}$.
\end{corollary}

\begin{remark}\label{isomorphism-when-cardinal}
In the special case where $T$ is a cardinal $\lambda$ and $t \geq \sup (N \cap \lambda)$,
we have $N \cap T \subseteq \pred t$, so that elementarity of $N$ implies that $\pi_{N,t}$ is one-to-one,
giving an
isomorphism of set algebras
\[
\pi_{N,t} : \left< \mathcal P (\lambda) \cap N, \cup, \cap, \setminus,
\emptyset, \lambda \right> \cong \left< \mathcal A_{N,t}, \cup, \cap,
\setminus, \emptyset, \pred t \right>.
\]
In this case, $t$ is necessarily $N$-eligible, via condition (\ref{intersection-determines-t}) 
of Lemma~\ref{eligible-equivalent}.
So provided that $\sup (N \cap \lambda) < \lambda$ 
(such as when $\left|N\right| < \lambda$ for regular cardinal $\lambda$),
we can always choose an $N$-eligible node $t = \sup (N \cap \lambda)$ in this case.%
\footnote{\label{Mostowski}%
Furthermore, if $N \cap \lambda$ is downward closed and equal to an ordinal $\delta < \lambda$, 
then we can choose $t = \delta$,
and we have $\pred t \subseteq N$, 
and $\pi_{N,t}$ is just the Mostowski collapsing function of $\left<\mathcal P(T) \cap N, \in \right>$
onto its transitive collapse $\mathcal A_{N,t}.$}
In the general case of a tree $T$, it not clear that every model $N$ necessarily has an $N$-eligible node,
so we will have to work harder later on to show that such models and nodes exist.
\end{remark}

The significance of the following lemma will become apparent when we
introduce \emph{reflection points} later on.

\begin{lemma}\label{reflection-intersection}
Suppose $N \prec H(\theta)$ is an elementary submodel such that $T \in N$, and let $t \in T$.
Fix $X \in \mathcal A_{N,t}$, and $S \subseteq T$ such that $S \cap
\pred t \in I_{N,t}^+$.
Then
\todo{Consider moving this later as it is not needed here and may fit better with the description of reflection points.}
\[
X \in \mathcal G^+_{N,t} \iff X \cap S \in I^+_{N,t}.
\]

\begin{proof}
(Notice that since $S \cap \pred t \in I^+_{N,t}$, we have in particular that $I^+_{N,t} \neq \emptyset$,
so that $t$ is necessarily $N$-eligible.  
However, we do not use this fact formally in the proof.)
\begin{description}
\item[($\impliedby$)] If $X \cap S \in I^+_{N,t}$ then certainly $X
\in I^+_{N,t}$.  Since also $X \in \mathcal A_{N,t}$, Lemma~\ref{I_N-formulations}
gives $X \in \mathcal G^+_{N,t}$.
%
\item[($\implies$)] Suppose $X \in \mathcal G^+_{N,t}$.  
Then Lemma~\ref{G_N-dichotomy} gives $X \in \mathcal G^*_{N,t}$, and
then by Lemma~\ref{I_N-formulations} also $X \in I^*_{N,t}$.
By hypothesis, $S \cap \pred t \in I_{N,t}^+$.
The intersection of a co-ideal set and a filter set must be in the
co-ideal, and of course $X \subseteq \pred t$, so we have $X \cap S
\in I^+_{N,t}$, as required.\qedhere
\end{description}
\end{proof}
\end{lemma}

We now consider what happens to the algebraic structures on $\pred t$ when we build a new model by
\emph{fattening} an existing one, that is, by adding sets to the model:

\begin{lemma}\label{fatten}
Suppose $M, N \prec H(\theta)$ are two elementary submodels such that $T \in M, N$, 
and also $N \subseteq M$, and let $t \in T$.
Then we have:
\begin{gather*}
\mathcal A_{N,t} \subseteq \mathcal A_{M,t}     \\
\pi_{N,t} = \pi_{M,t} \upharpoonright \left( \mathcal P(T) \cap N\right)   \\
\mathcal G_{N,t} \subseteq \mathcal G_{M,t} \\	
I_{N,t} \subseteq I_{M,t}                           \\
I_{N,t}^* \subseteq I_{M,t}^* 		\\	
I_{N,t}^+ \supseteq I_{M,t}^+
\end{gather*}
Furthermore, if $t$ is $M$-eligible then $t$ is also $N$-eligible, and we have
\begin{gather*}
\mathcal G_{N,t} = \mathcal G_{M,t} \cap \mathcal A_{N,t}	\tag{$*$}	\label{fatten-G}	\\
\mathcal G_{N,t}^+ = \mathcal G_{M,t}^+ \cap \mathcal A_{N,t}     \\
I_{N,t}^* \subseteq I_{M,t}^* \subseteq I_{M,t}^+ \subseteq I_{N,t}^+
\end{gather*}

\begin{proof}
Mostly straight from the definitions and previous lemmas.
For~\eqref{fatten-G}:
Suppose $t$ is $M$-eligible, and $X \in \mathcal G_{M,t} \cap \mathcal A_{N,t}$.
We must show $X \in \mathcal G_{N,t}$.
Since $X \in \mathcal G_{M,t}$, we have $X = A \cap \pred t$ for some $A \in M$ with $t \notin A$.
Since $X \in \mathcal A_{N,t}$, we have $X = B \cap \pred t$ for some $B \in N \subseteq M$.
Since $t$ is $M$-eligible, and $A \cap \pred t = B \cap \pred t$ where $A, B \in M$ and $t \notin A$,
condition~(\ref{intersection-determines-t}) of Lemma~\ref{eligible-equivalent} gives $t \notin B$,
so that $X \in \mathcal G_{N,t}$, as required.
\end{proof}
\end{lemma}

We will want our algebraic structures defined using a model $N$ to be $\kappa$-complete, 
for some fixed cardinal $\kappa$.  
To ensure this, we impose the condition that the model $N$ must contain all of its subsets of size $<\kappa$, 
that is, we suppose $[N]^{<\kappa} \subseteq N$.  
What conditions does this impose on the combinatorial relationship
\todo{Maybe the next two lemmas should be in a separate section on cardinal arithmetic, 
along with other relevant facts.} 
between $\kappa$ and $
\left|N\right|$?  

\begin{lemma}\label{closed-subsets}
For any infinite set $A$ and any 
cardinal $\kappa$, if $[A]^{<\kappa} \subseteq A$, then 
$\left|A\right|^{<\kappa} = \left|A\right|$.


\end{lemma}

\begin{lemma}\label{nu-and-kappa}
Let $\nu$ and $\kappa$ be infinite cardinals.  If $\nu^{<\kappa} = \nu$, then 
$\kappa \leq \cf(\nu)$.

\begin{proof}
Clearly, for every $\mu < \kappa$ we have $\nu^\mu = \nu$, that is, 
$\left\{ \nu^\mu  : \mu < \kappa \right\}$ is constant.
Then by~\cite[Theorem~6.10(d)(i)]{EHMR}, we have $\cf(\nu^{<\kappa}) \geq \kappa$.
%
\end{proof}
\end{lemma}

Further consequences on the algebraic structure due to the elementarity of $N$ are:

\begin{lemma}\label{k-complete}
Suppose $N \prec H(\theta)$ is an elementary submodel such that $T \in N$, and let $t \in T$.
Let $\kappa$ be any cardinal.  If we have $[N]^{<\kappa} \subseteq N$,
then:
\begin{enumerate}
\item If $\mathcal B \in [N]^{<\kappa}$, then $\bigcup \mathcal B \in N$.
\item The set algebra $\mathcal P(T) \cap N$ 
is $\kappa$-complete.
\item The ultrafilter and maximal ideal of Lemma~\ref{ultrafilter-on-tree} are $\kappa$-complete.
\item The set algebra 
$\mathcal A_{N,t}$ is $\kappa$-complete.
\item The ideals $\mathcal G_{N,t}$ and $I_{N,t}$ are $\kappa$-complete.
\item If $\delta$ is the smallest ordinal not in $N$, then $\cf(\delta) \geq \kappa$.
\item If $t$ is $N$-eligible and $\pred t \subseteq N$, then
\[
\kappa \leq \cf\left(\height_T(t)\right).
\]
\end{enumerate}

\begin{proof}
Suppose the cardinal $\kappa$ satisfies $[N]^{<\kappa} \subseteq N$.
\begin{enumerate}
\item
Suppose $\mathcal B \in [N]^{<\kappa}$.
Since $[N]^{<\kappa} \subseteq N$, we have $\mathcal B \in N$.
Now $N \prec H(\theta)$, so $N$ models a sufficient fragment of
ZFC, including the union axiom, so it follows that $\bigcup \mathcal B \in N$, as required.
\item

A union of subsets of $T$ is certainly a subset of $T$, 
and from part~(1) we know that a union of fewer than $\kappa$ sets from $N$ is in $N$.

\item 
For any collection $\mathcal D$ of sets in the maximal ideal, 
we have $t \notin B$ for each $B \in \mathcal D$.
Certainly then,
\[
t \notin \bigcup_{B \in \mathcal D} B.
\]
So if the union $\bigcup \mathcal D$ is in the set algebra $\mathcal P(T) \cap N 
$, then it is in the maximal ideal as well.  
%
Since $\mathcal P(T) \cap N$ is $\kappa$-complete by part~(1), it follows
that the maximal ideal is $\kappa$-complete,
and so is the dual ultrafilter.

\item
By Lemma~\ref{homomorphism-set-algebra}, the set algebra $\mathcal
A_{N,t}$ is a homomorphic image of the $\kappa$-complete set algebra 
$\mathcal P(T) \cap N$, so it is also
$\kappa$-complete.

\item
The ideal $\mathcal G_{N,t}$ is the homomorphic image of a $\kappa$-complete ideal (from part~(3))
into a $\kappa$-complete set algebra $\mathcal A_{N,t}$, so it is $\kappa$-complete as well.

$\kappa$-completeness of $I_{N,t}$ follows easily.
\item Let $\delta$ be the smallest ordinal not in $N$,
and fix any cardinal $\mu < \kappa$. 
We must show that $\mu < \cf (\delta)$.

For each ordinal $\iota<\mu$, choose some ordinal $\gamma_\iota < \delta$.
Let
\[
\gamma = \sup_{\iota<\mu} \gamma_\iota = \bigcup_{\iota<\mu} \gamma_\iota,
\]
and we will show that $\gamma < \delta$.  

Since each $\gamma_\iota < \delta$, it is clear that $\gamma \leq \delta$.  But also each $\gamma_\iota \in N$, so
we have
\[
\left\{ \gamma_\iota : \iota < \mu \right\} \in \left[N\right]^\mu
\subseteq \left[N\right]^{<\kappa}, 
\]
so 
it follows from part~(1) that $\gamma \in N$.
Since $\delta \notin N$, it follows that $\gamma < \delta$,
as required. 
\item This follows immediately by combining the previous part with Lemma~\ref{eligible-implications}(4).
%
%
%
%
\qedhere
\end{enumerate}
\end{proof}
\end{lemma}


We now consider what effect our elementary submodels have on a given colouring $c : [T]^2 \to \mu$:

\begin{lemma}\label{some-colour-in-co-ideal}
Suppose we have cardinals  $\mu$ and $\kappa$, with $\mu < \kappa$,
and a colouring $c: [T]^2 \to \mu$. 
Suppose also that $N \prec H(\theta)$ is an elementary submodel such that 
$T \in N$, and also $[N]^{<\kappa} \subseteq N$, and let $t \in T$.
Then for any $X \subseteq \pred t$, we have%
\footnote{Recall the definition of the notation $c_\chi(t)$ in Section~\ref{Section:notation}.}
\[
X \in I^+_{N,t} \iff \exists \text{ some colour $\chi < \mu$ such that } X \cap c_\chi(t) \in I^+_{N,t}.
\]

\begin{proof}
For any $X \subseteq \pred t$,
we clearly have
\[
X = X \cap \bigcup_{ \chi<\mu} c_\chi(t) = \bigcup_{\chi<\mu} \left(X \cap c_\chi(t) \right).
\]
Since the model $N$ 
satisfies $[N]^{<\kappa} \subseteq N$, Lemma~\ref{k-complete}(5) tells us
that $I_{N,t}$ is $\kappa$-complete.  
Since $\mu < \kappa$,
the required result follows.
\end{proof}
\end{lemma}

The following lemma contains the crucial recursive construction of a homogeneous chain of length~$\kappa$ 
as a subset of an appropriate set from the co-ideal $I^+_{N,t}$:

\begin{lemma} (cf.~\cite[Claim before Lemma~2.2]{BHT}, \cite[Claim~2.2]{Baumgartner})
\label{homog-2.2}  
Suppose we have cardinals  $\mu$ and $\kappa$,
a colouring 
$c: [T]^2 \to \mu$, 
and some colour $\chi<\mu$. 
Suppose also that $N \prec H(\theta)$ is an elementary submodel such that 
$T, c, \chi \in N$, and also $[N]^{<\kappa} \subseteq N$. 
Let $t \in T$ be a node such that $\pred t \subseteq N$.  

If $X \subseteq c_\chi(t)$ is such that $X \in I^+_{N,t}$, then there is a $\chi$-homogeneous chain
$Y \in 
[X]^\kappa$. 

\begin{proof}
%
We will recursively construct a $\chi$-homogeneous chain
\[
Y = \left<y_\eta\right>_{\eta<\kappa} \subseteq X,
\]
of order type $\kappa$, as follows:

Fix some ordinal $\eta < \kappa$, and suppose we have constructed
$\chi$-homogeneous
\[
Y_\eta = \left<y_\iota\right>_{\iota<\eta} \subseteq X
\]
of order type $\eta$.  We need to choose $y_\eta \in X$ such that
$Y_\eta <_T \{y_\eta\}$ and $Y_\eta \cup \{y_\eta\}$ is
$\chi$-homogeneous.

Since $Y_\eta \subseteq X \subseteq \pred t \subseteq N$ and
$\left|Y_\eta\right|<\kappa$, 
the hypothesis that $[N]^{<\kappa} \subseteq N$ gives us $Y_\eta \in
N$.  Define
\[
Z = \left\{ s \in T : \left(\forall y_\iota \in
Y_\eta\right) \left[y_\iota <_T s \text{ and } c \left\{ y_\iota,
s \right\} = \chi \right] \right\}.
\]

Since $Z$ is defined from parameters $T, Y_\eta, c$, and
$\chi$ that are all in $N$, it follows by elementarity of $N$ that
$Z \in N$, so that $Z \cap \pred t \in \mathcal A_{N,t}$.

Since $Y_\eta \subseteq X \subseteq c_\chi(t)$, it follows
from
the definition of $Z$ that $t \in Z$.  But then 
we have $Z \cap \pred t \in 
\mathcal G^*_{N,t}
\subseteq I^*_{N,t}$.  By assumption we have $X \in I^+_{N,t}$. The
intersection of a filter set and a co-ideal set must be in the
co-ideal, so we have $X \cap Z \in I^+_{N,t}$.  In particular, this set
is not empty, so we choose $y_\eta \in X \cap Z$.  Because $y_\eta
\in Z$, we have $Y_\eta <_T \{y_\eta\}$ and $Y_\eta \cup \{y_\eta\}$
is $\chi$-homogeneous, as required.
\end{proof}
\end{lemma}

Two observations about Lemma~\ref{homog-2.2} will be demonstrated in the following corollary;
one about the hypotheses, and the other about the conclusion.

First, implicit in the hypothesis $X \in I^+_{N,t}$ of Lemmas~\ref{some-colour-in-co-ideal} and~\ref{homog-2.2} 
is the fact that $t$ is $N$-eligible.
Ultimately, it will be the existence of $N$-eligible nodes that will help us find suitable sets $X \in I^+_{N,t}$ 
to which we can apply Lemma~\ref{homog-2.2}.
\todo{Check that this matches the actual use!}

Second, the conclusion of
Lemma~\ref{homog-2.2} actually gives us
a $\chi$-homogeneous chain $Y \cup \{t\}$ of order-type $\kappa+1$.
This will be useful when we prove Theorem~\ref{ER-trees-Stevo} for regular cardinals 
in \autoref{section:ER}.

\begin{corollary}\label{homog-k+1}
Suppose we have cardinals  $\mu$ and $\kappa$, with $\mu < \kappa$,
and a colouring 
$c: [T]^2 \to \mu$. 
Suppose also that $N \prec H(\theta)$ is an elementary submodel such that 
$T, c \in N$, and also $[N]^{<\kappa} \subseteq N$.
Suppose $t \in T$ is $N$-eligible and $\pred t \subseteq N$.  
Then there is a chain
$Y \in [\pred t]^\kappa$ such that 
$Y \cup \{t\}$ is 
homogeneous for $c$.

\begin{proof}
Since $t$ is $N$-eligible, we have by Lemma~\ref{eligible-equivalent} that
$I_{N,t}$ is a proper ideal on $\pred t$,
so that $\pred t \in I^+_{N,t}$.
Applying Lemma~\ref{some-colour-in-co-ideal} to $\pred t$ itself,
we fix a colour $\chi < \mu$ such that $c_\chi(t) \in I^+_{N,t}$.


\begin{claim}
We have $\chi \in N$.

\begin{proof}
Since $\chi < \mu < \kappa \leq \cf(\height_T(t)) 
\leq \height_T(t)$ (using Lemma~\ref{k-complete}(7)),
there must be
some $s <_T t$ with $\height_T(s) = \chi$.
But then $s \in N$ (since by assumption $\pred t \subseteq N$), 
so by elementarity we have $\chi = \height_T(s) \in N$.
\end{proof}
\end{claim}

We now
apply Lemma~\ref{homog-2.2} to
$c_\chi(t)$ itself, to obtain a $\chi$-homogeneous chain $Y \in [c_\chi(t)]^\kappa$.  
Then $Y \cup \{t\}$ is a
$\chi$-homogeneous chain of order type $\kappa + 1$, as required.
\end{proof}
\end{corollary}

So if we could ensure the existence of models $N$ with some $N$-eligible nodes $t$ such that $\pred t \subseteq N$,
then we would be part way toward our goal of obtaining the long homogeneous chains we are looking for.
Having built up an algebraic structure based on hypothetical
elementary submodels $N \prec H(\theta)$, 
we would like to know:
Under what circumstances can we guarantee that some models $N$ will have some $N$-eligible nodes?

First, a counterexample:

\begin{example}
Fix any infinite regular cardinal $\kappa$.
Let $T$ be a $(2^{<\kappa})$-
special tree of height $(2^{<\kappa})^+$ (such as, for example, a special Aronszajn tree,
which we know exists for $\kappa = \aleph_0$
).
\todo{In thesis, restore reference here to theorem in introduction.}
Then by Corollary~\ref{special-counterexample} 
we have 
\[
T \not\to \left(\kappa+1, \kappa\right)^2.
\]
Fix a colouring $c : [T]^2 \to 2$ witnessing this negative partition relation.
In particular, there is no $c$-homogeneous chain in $T$ of order-type $\kappa+1$.

\begin{claim}
If $N \prec H(\theta)$ is any 
elementary submodel such that $T, c \in N$ and $[N]^{<\kappa} \subseteq N$,
then there is no $N$-eligible node $t \in T$ with $\pred t \subseteq N$,
even though (provided $\left|N\right| = 2^{<\kappa}$) 
there are nodes $t \in T$ with $ht_T(t) = \sup(N \cap (2^{<\kappa})^+)$.

\begin{proof}
Suppose $t \in T$ is $N$-eligible and $\pred t \subseteq N$.
Applying Corollary~\ref{homog-k+1} to the colouring $c$ (with 
$\mu = 2$),
we get a $c$-homogeneous chain of order-type $\kappa+1$,
contradicting our choice of $c$.
\end{proof}
\end{claim}
\end{example}

However, $(2^{<\kappa}$)-special trees are essentially the only counterexamples.
%
%
We will proceed in the next section 
to show how to construct collections of elementary submodels $N \prec H(\theta)$,
and to show that provided $T$ is a non-$(2^{<\kappa})$-special tree 
for some regular cardinal $\kappa$, 
we can 
guarantee that some of the models $N$ satisfying $[N]^{<\kappa} \subseteq N$ 
will have some $N$-eligible nodes $t$ such that $\pred t \subseteq N$.%
\footnote{In the special case where $T$ is the cardinal $(2^{<\kappa})^+$ for some regular cardinal $\kappa$,
and we have fixed a colouring $c : [T]^2 \to \mu$ for some cardinal $\mu < \kappa$,
we can use~\cite[Lemma~24.28 and Claim~24.23(b)]{J-W} or~\cite[p.~245]{Milner-esm}
to fix an elementary submodel $N \prec H(\theta)$ with $T, c \in N$, such that
$\left|N\right| = 2^{<\kappa}$, $[N]^{<\kappa} \subseteq N$, and
$N \cap (2^{<\kappa})^+ = \delta$ for some ordinal $\delta$ with $\left|\delta\right| = 2^{<\kappa}$.
Then, by Remark~\ref{isomorphism-when-cardinal} and footnote~\ref{Mostowski},
we can set $t = \delta$, so that $t$ is $N$-eligible and $\pred t \subseteq N$.
Then Corollary~\ref{homog-k+1} gives us a chain of order-type $\kappa+1$, homogeneous for $c$.
This is the proof of the balanced Erd\H os-Rado Theorem for regular cardinals 
given in~\cite[Section~2, Theorem~2.1]{BHT}.

In the slightly more general case of a tree $T$ of height $(2^{<\kappa})^+$ such that
every antichain of $T$ has cardinality $\leq 2^{<\kappa}$ 
(such as a $(2^{<\kappa})^+$-Souslin tree, if one exists),
we can modify the constructions of~\cite[Lemma~24.28]{J-W} and~\cite[p.~245]{Milner-esm}
so that, for each ordinal $\eta < \kappa$, we impose the extra condition
\[
\bigcup_{B \in N_\eta} \left\{ t\in T : \pred t \subseteq B \text{ and } t \notin B \right\} \subseteq N_{\eta+1}.
\]
Then the model $N = N_\kappa$ has the property that all nodes in $T \setminus N$ are $N$-eligible,
and in particular we have $T \cap N = \left\{ t \in T : \height_T(t) < \delta \right\}$, 
where $\delta = N \cap (2^{<\kappa})^+$.
We can then choose any $t \in T$ of height $\delta$, and as before,
Corollary~\ref{homog-k+1} gives us a chain of order-type $\kappa+1$, homogeneous for $c$,
proving the case of Theorem~\ref{ER-trees-Stevo} for the tree $T$.

However, in the general case of an arbitrary non-$(2^{<\kappa})$-special tree,
this method is insufficient.
We will need the full strength of the construction in \autoref{section:very-nice}
in order to find models $N$ with $N$-eligible nodes $t$ such that $\pred t \subseteq N$,
before we can return to proving Theorem~\ref{ER-trees-Stevo} (for regular cardinals $\kappa$)
in \autoref{section:ER}.

Furthermore, our ultimate goal is to prove the Main Theorem~\ref{BHT-trees-regular},
and for this we will need many models with eligible nodes, 
requiring the full strength of the subsequent constructions even in the simpler cases mentioned above.
%
}

%% file: very_nice_BHT.tex


\section{Very Nice Collections of Elementary Submodels}
\label{section:very-nice}

We will generalize 
Kunen's definition \cite[Definition~III.8.14]{New-nen} 
of a \emph{nice chain} of elementary submodels of $H(\theta)$:%
\footnote{Kunen's \emph{nice chains} are the special case of our \emph{nice collections of elementary submodels}
where $T = \lambda = \omega_1$, except that we require $N_\emptyset$ to be an elementary submodel,
rather than Kunen's $N_0 = \emptyset$.}
\todo{Compare also the \emph{approximation sequences} in Chapter 4 of
\emph{Introduction to Cardinal Arithmetic} (1999; Holz, Stefffens, Weitz).}

\begin{definition}
Let $\lambda$ be any regular uncountable cardinal, and let $T$ be a tree of height $\lambda$.
The collection $\left<W_t\right>_{t \in T}$ is called a 
\emph{nice collection of sets indexed by $T$} if:
\begin{enumerate}
\item For each $t \in T$, $\left|W_t\right| < \lambda$;
\item The collection is \emph{increasing}, 
\todo{Instead of saying ``for all $s <_T t$'', 
consider changing to ``for all $s, t \in T$ where $t$ is an immediate successor of $s$''.}
meaning that for $s, t \in T$ with $s <_T t$, $W_s \subseteq W_t$;
\item The collection is \emph{continuous} (with respect to its indexing),%
\footnote{
The continuity condition 
is consistent with the topological notion of continuity if we define the appropriate topologies:

The topology on $\mathcal P(W)$, where $W = \bigcup_{t \in T} W_t$,  
is the product topology on $\{0,1\}^W 
$,
that is, the topology of pointwise convergence on the set of characteristic functions on $W$, 
where of course $\{0,1\}$ has the discrete topology.

The topology on $T$ is the \emph{tree topology}, defined earlier in footnote~\ref{tree-topology}.
} 
meaning that for all limit nodes $t \in T$, 
\[
W_t = \bigcup_{s <_T t} W_s.
\]
    \setcounter{condition}{\value{enumi}}
\end{enumerate}
Suppose furthermore that
$\theta \geq \lambda 
$ is a regular cardinal such that%
\footnote{We could simplify the requirements on $\theta$ in this definition and in the subsequent lemmas 
by requiring the stronger condition $T \in H(\theta)$.
This implies both of the required conditions $T \subseteq H(\theta)$ and $\theta \geq \lambda$,
as well as the extra condition $\theta > \left|T\right|$, 
but this seems to be unnecessary.
} 
$T \subseteq H(\theta)$.
The collection $\left<N_t\right>_{t \in T}$ is called a 
\emph{nice collection of elementary submodels of $H(\theta)$ indexed by $T$} if,
in addition to being a nice collection of sets as above, we have:
\begin{enumerate}
    \setcounter{enumi}{\value{condition}}
\item For each $t \in T$, $N_t \prec H(\theta)$;
\item For each $t \in T$, $\pred{t} \subseteq N_t$;
\todo{If $\lambda$ is a successor cardinal, say $\lambda = \nu^+$, we can also require $\nu \subseteq N_0$,
so that (1) and (6) together 
would imply (2), 
and this also gives that $N_t \cap \lambda$ is always an ordinal.}
\item For $s, t \in T$ with $s <_T t$, $N_s \in N_t$. 
\todo{Consider using this condition:  
For $s, t \in T$ with $s <_T t$, $\left< N_r : r \leq_T s \right> \in N_t$.
We can easily accomplish this in the construction.
Together with condition (5) (or if $\left|\height_T(t)\right| \subseteq N_t$), it certainly implies (6), 
but not (2) 
unless we also impose $\left|N_s \right| \subseteq N_t$ for all $s <_T t$.
Decide whether we ever need it for this proof, or whether it is ever better to include this.}
    \setcounter{condition}{\value{enumi}}
\end{enumerate}
If $\kappa$ is an infinite cardinal, then we say $\left<N_t\right>_{t \in T}$ is a
\emph{$\kappa$-very nice collection of elementary submodels} if, in addition to the above conditions, we have
\begin{enumerate}
    \setcounter{enumi}{\value{condition}}
\item For $s, t \in T$ with $s <_T t$, $[N_s]^{<\kappa} \subseteq N_t$.
\end{enumerate}
If $\left<M_t\right>_{t \in T}$ and $\left<N_t\right>_{t \in T}$ are 
two nice collections of sets, 
then we say that
$\left<N_t\right>_{t \in T}$ is a \emph{fattening} of $\left<M_t\right>_{t \in T}$ if
for all $t \in T$ we have $M_t \subseteq N_t$.
\end{definition}


Notice that all nice collections of elementary submodels are $\aleph_0$-very nice collections, 
since any elementary submodel of $H(\theta)$ contains all of its finite subsets.

What condition on the combinatorial relationship between $\kappa$ and $\lambda$ is necessary for the existence
of a $\kappa$-very nice collection of elementary submodels indexed by a tree $T$ of height $\lambda$?

\begin{lemma}
Suppose $\lambda$ is any regular uncountable cardinal, $T$ is a tree of height $\lambda$, 
and
$\theta \geq \lambda$ is a regular cardinal such that 
$T \subseteq H(\theta)$.
Suppose 
$\kappa$ is any infinite cardinal.
If there exists a $\kappa$-very nice collection of elementary submodels of $H(\theta)$ indexed by $T$,
then we must have 
\[
\tag{$**$}
\label{necessary-k-very-nice}
\left(\forall \text{ cardinals } \nu < \lambda \right) \left[\nu^{<\kappa} < \lambda\right].
\]

\begin{proof}
Let $\left<N_t\right>_{t\in T}$ be a $\kappa$-very nice collection of elementary submodels of $H(\theta)$,
and fix any cardinal $\nu < \lambda$.
Since $T$ has height $\lambda$, 
we can choose some $s, t \in T$ with $\height_T(s) = \nu$ and $s <_T t$.
Then $\pred s \subseteq N_s$, so that $\left|N_s\right| \geq \left|\pred s\right| = \nu$.
Since the collection is $\kappa$-very nice, we must have $[N_s]^{<\kappa} \subseteq N_t$, so that
\[
\left| N_t \right| \geq \left| \left[N_s\right]^{<\kappa} \right| = \left|N_s\right|^{<\kappa} \geq \nu^{<\kappa}.
\]
But we must also have $\left|N_t\right| < \lambda$,
giving the requirement $\nu^{<\kappa} < \lambda$.
\end{proof}
\end{lemma}

In the intended applications, the height $\lambda$ of our tree will be a successor cardinal $\nu^+$.
In that case, 
condition~\eqref{necessary-k-very-nice}
becomes simply $\nu^{<\kappa} = \nu$,
from which we obtain 
%
the following chain of equations and inequalities (using
Lemma~\ref{nu-and-kappa} for one of them), which we will refer to
when necessary: 
\[
\kappa \leq \cf(\nu) \leq \nu = 
\nu^{<\kappa} < \nu^+ = \height (T) 
\]

In general, it turns out that the necessary 
condition~
\eqref{necessary-k-very-nice} is also sufficient,
as the following lemma shows (particularly, part~(3)):

\begin{lemma}\label{build-nice-collection}
Suppose $\lambda$ is any regular uncountable cardinal, $T$ is a tree of height $\lambda$, 
and
$\theta \geq \lambda$ is a regular cardinal such that 
$T \subseteq H(\theta)$.
Fix $X \subseteq H(\theta)$ with $\left|X\right| < \lambda$.
Then:
\begin{enumerate}
\item There is a nice collection $\left<N_t \right>_{t \in T}$ of elementary submodels of $H(\theta)$ 
such that $X \subseteq N_\emptyset$ (and therefore $X \subseteq N_t$ for every $t \in T$).
\item Given any nice collection $\left<M_t\right>_{t \in T}$ of elementary submodels of $H(\theta)$, 
we can {fatten} the collection to include $X$, that is, we can construct another nice collection 
$\left<N_t\right>_{t \in T}$ of elementary submodels of $H(\theta)$, 
that is a fattening of $\left<M_t\right>_{t \in T}$,
such that 
$X \subseteq N_\emptyset$.
\item If $\kappa$ is an infinite cardinal such that 
for all cardinals $\nu < \lambda$ we have $\nu^{<\kappa} < \lambda$,
then the nice collections we construct in parts (1) and (2) can be $\kappa$-very nice collections.
\end{enumerate}

\begin{proof}
We construct the nice collection recursively.  The Downward
L\"owenheim-Skolem-Tarski Theorem guarantees the existence of elementary
submodels of arbitrary infinite cardinality, and a version of it
given in~\cite[Theorem~I.15.10]{New-nen},
\cite[Corollary~24.13]{J-W},
\cite[Theorem~1.1]{Dow}, and~\cite[Theorem~2]{Milner-esm}
\todo{Include more references to~\protect\cite{Dow} and~\protect\cite{Milner-esm}.}
says that the submodel can even be guaranteed to
contain any number of specified items, up to the cardinality of the
desired submodel. This is our main tool for the construction, which
proceeds as follows:

\begin{description}
\item [For $\emptyset$]  \hfill
\begin{enumerate}
\item By the Downward L\"owenheim-Skolem-Tarski Theorem version
just mentioned, 
we can choose $N_\emptyset \prec H(\theta)$ such that
$X 
\subseteq N_\emptyset$,
with 
\[
\left|N_\emptyset\right| = \max \left\{ \left|X\right|, \aleph_0 \right\} < \lambda,
\] 
satisfying the required properties.

\item If we are fattening an already-existing collection, there is no
difficulty in ensuring as well that
$M_\emptyset 
\subseteq N_\emptyset$.
In this case we would have 
\[
\left|N_\emptyset\right| = \max \left\{ \left|X\right|, \left|M_\emptyset\right| \right\} < \lambda.
\] 

\item There is no additional requirement on $N_\emptyset$ in a $\kappa$-very nice collection.
\end{enumerate}
\item [For successor nodes]  \hfill
\begin{enumerate}
\item Fix $s \in T$, 
and assume that we
have already constructed $N_s \prec H(\theta)$ satisfying
the required properties,
and suppose that $t\in T$ is an immediate successor of $s$.
%
%
Again, by the Downward L\"owenheim-Skolem-Tarski Theorem,
we can choose
$N_{t} \prec H(\theta)$
such that 
%
\[
N_s \cup \left\{N_s\right\} \cup \pred{t} 
\subseteq N_{t},
\]
with
$\left|N_{t}\right| = \left| N_s \right| < \lambda$.
The required properties are easy to verify.

\item If we are fattening an already-existing collection, then again there
is no difficulty in ensuring as well that $M_{t} \subseteq
N_{t}$. 
In this case we would have 
\[
\left|N_t\right| = \max \left\{ \left|N_s\right|, \left|M_t\right| \right\} < \lambda.
\] 

\item Since $\left|N_s\right| < \lambda$, the extra hypothesis in this part gives us
%
\[
\left|\left[N_s\right]^{<\kappa}\right| \leq 
\left|N_s\right|^{<\kappa} < \lambda, 
\]
so that we can choose $N_t$ such that 
\[
\left[N_s\right]^{<\kappa} \subseteq N_{t},
\]
while still having $\left|N_t \right| < \lambda$.
\end{enumerate}

\item [For limit nodes]  Fix limit node $t \in T$, 
and assume that we have already constructed the chain $\left<N_s 
: s <_T t \right>$ satisfying the required properties.
%
Define
\[
N_t = \bigcup_{s <_T t} N_s.
\]
As the union of an increasing chain of elementary submodels is an
elementary submodel (\cite[Lemma~24.5]{J-W},
\cite[Corollary~1.3]{Dow}, and~\cite[top of p.~245]{Milner-esm}), we have $N_t \prec
H(\theta)$. 
Since $\lambda$ is a regular cardinal, and
$\height_T(t)<\lambda$ (so that $\left|\pred{t}\right| < \lambda$), and each $\left|N_s\right| < \lambda$,
it is clear that
$\left|N_t\right|  < \lambda$. 
The remaining properties are easy to verify.\qedhere
\end{description}
\end{proof}
\end{lemma}

Given any tree $T$ of height $\lambda$, any large enough $\theta$, 
and any cardinal $\kappa$ satisfying 
condition~\eqref{necessary-k-very-nice},
we can use Lemma~\ref{build-nice-collection} to construct a $\kappa$-very nice collection of elementary submodels
$\left<N_t\right>_{t\in T}$ of $H(\theta)$,
such that $N_\emptyset$ contains any relevant sets,
and in particular we can ensure that $T \in N_\emptyset$.
We can then use any node $t \in T$ and its associated model $N_t$ to build the algebraic structures defined in 
\autoref{section:ideals-esm},
including the ideal $I_{N_t,t}$ on $\pred t$.
By definition of our nice collections, we always have $N_t \prec H(\theta)$ and $\pred t \subseteq N_t$,
but in order to get the most value from these structures,
we will need to find nodes and models with two extra features:
eligibility and $\kappa$-completeness.

First, $\kappa$-completeness of the required algebraic structures (and, in
particular, the recursive construction of Lemma~\ref{homog-2.2})
depends on the additional condition $[N]^{<\kappa} \subseteq N$
introduced before Lemma~\ref{closed-subsets}. 
For any $s <_T t$ in $T$,
we have $[N_s]^{<\kappa} \subseteq N_{t}$.  
But we want to know: Which nodes $t \in T$ can we
choose so that the model $N_t$ satisfies the stronger condition
$[N_t]^{<\kappa} \subseteq N_{t}$?

\begin{lemma}\label{model-kappa-complete}
Suppose $\lambda$ is any regular uncountable cardinal, $T$ is a tree of height $\lambda$, 
and
$\theta \geq \lambda$ is a regular cardinal such that 
$T \subseteq H(\theta)$.
Suppose 
$\kappa$ is any infinite cardinal, 
and 
$\left<N_t\right>_{t \in T}$ is a $\kappa$-very nice collection of elementary submodels of $H(\theta)$.
%
Then
for every 
$t \in T$,
we have:
\begin{enumerate}
\item
if $\cf(\height_T(t)) \geq \kappa$ then $[N_t]^{<\kappa} \subseteq N_t$.
\item If $t$ is $N_t$-eligible, then%
\footnote{This part is not actually used.} 
\[
\cf\left(\height_T(t)\right) \geq \kappa \iff 
\left[N_t\right]^{<\kappa} \subseteq N_t.
\]
\end{enumerate}

\begin{proof}\hfill
\begin{enumerate}
\item
Fix $t \in T$ 
such that $\cf(\height_T (t)) \geq \kappa$.  
Fix a cardinal $\mu<\kappa$, and some collection
\[
\mathcal C = \left<A_\iota\right>_{\iota<\mu} \in
\left[N_t\right]^\mu.
\]
For each ordinal $\iota<\mu$, we have $A_\iota \in N_t$.  Since
$\cf(\height_T(t))\geq \kappa$, $t$ must be a limit node, 
so since the collection of models
is 
continuous, we have $A_\iota \in N_{s_\iota}$ for some
$s_\iota <_T t$. Then define
\todo{Formalize this notation, since the $\sup$ of a sequence of nodes
is defined only because we are clearly within a chain $\pred{t}$ (or possibly $\pred{t} \cup \{t\}$).}
\[
s = \sup_{\iota<\mu} s_\iota, 
\]
where the $\sup$ is taken along the chain $\pred t \cup \{t\}$.
Since each $s_\iota <_T t$ and $\mu < \kappa \leq
\cf(\height_T(t))$, we have $s <_T t$.  We then have, 
since the collection is $\kappa$-very nice,
\[
\mathcal C \in \left[N_s\right]^\mu \subseteq
\left[N_s\right]^{<\kappa} \subseteq N_t,
\]
as required.
\item This is simply a combination of the previous part with
Lemma~\ref{k-complete}(7). 
\qedhere
\end{enumerate}
\end{proof}
\end{lemma}

Provided we start with a non-special tree, 
this guarantees a large supply of $\kappa$-complete models:

\begin{lemma}\label{k-complete-stationary}
Suppose $\nu$ is any infinite cardinal, 
$T$ is a non-$\nu$-special tree (necessarily of height $\nu^+$), 
and
$\theta > \nu$ is a regular cardinal such that 
$T \subseteq H(\theta)$.
Suppose 
$\kappa$ is an infinite cardinal, 
and 
$\left<N_t\right>_{t \in T}$ is a $\kappa$-very nice collection of elementary submodels of $H(\theta)$.
Then the set 
\[
\left\{ t \in T : 
	\left[N_t\right]^{<\kappa} \subseteq N_t \right\}
\]
is a stationary subtree of $T$.

\begin{proof}
Since $T$ is a non-$\nu$-special tree,
Lemma~\ref{stationary subtree} gives%
\footnote{In the definition of this stationary subtree, 
we can replace $= \cf(\nu)$ with $\geq \cf(\nu)$, if desired, 
or even with $\geq \kappa$ (since $\cf(\nu) \geq \kappa$).  
In the special case where $T$ is the cardinal $\nu^+$,
the more general textbook theorem applies (see our comment before Theorem~\ref{stationary subtree}),
so that if $\kappa$ is regular, we can alternatively use $= \kappa$,
as is done in the definition of $S_0$ given in~\cite[p.~5]{BHT}.} 
\[
T \upharpoonright S^{\nu^+}_{\cf(\nu)} = 
	\left\{ t \in T : \cf(\height_T(t)) = \cf(\nu) \right\} \notin NS^T_\nu.
\]
Since there exists a $\kappa$-very nice collection of elementary submodels indexed by $T$,
condition~\eqref{necessary-k-very-nice} must be satisfied, so that $\nu^{<\kappa} = \nu$.
Then Lemma~\ref{nu-and-kappa} gives $\kappa \leq \cf(\nu)$.
For any $t \in T \upharpoonright S^{\nu^+}_{\cf(\nu)}$,
we have $\cf(\height_T(t)) = \cf(\nu) \geq \kappa$,
so by Lemma~\ref{model-kappa-complete}(1) we have $[N_t]^{<\kappa} \subseteq N_t$.
It follows that
\[
T \upharpoonright S^{\nu^+}_{\cf(\nu)} \subseteq
	\left\{ t \in T : \left[N_t\right]^{<\kappa} \subseteq N_t \right\},
\]
so that this last set is also a stationary subtree of $T$,
as required.
\end{proof}
\end{lemma}

Next, recall the earlier eligibility condition for nodes and models:
Given a nice collection of sets $\left<W_t\right>_{t \in T}$,
the node $t \in T$ is $W_t$-{eligible} if 
\[
\nexists B \in W_t [\pred t \subseteq B \text{ and } t \notin B].
\]

We would like to know that not too many nodes $t$ 
are $W_t$-\emph{ineligible}.

\begin{lemma}\label{most-are-eligible}
Suppose $\nu$ is any infinite cardinal, and let $T$ be a tree of height $\nu^+$.
Suppose $\left<W_t\right>_{t \in T}$ is a nice collection of sets.
Then
\[
\left\{ t \in T : t \text{ is not $W_t$-eligible } 
			 \right\} \in NS^T_\nu.
\]

\begin{proof}

For any fixed set $B$, the set
$\left\{t \in T : \pred t \subseteq B \text{ and } t \notin B \right\}$
is an antichain.
For any $s \in T$, we have $\left|W_s\right| \leq \nu$, so it follows that
\[
\bigcup_{B \in W_s} \left\{t \in T : \pred t \subseteq B \text{ and } t \notin B \right\}
\]
is a union of $\leq \nu$ antichains,
that is, it is a $\nu$-special subtree.

Since the set of successor nodes is always a nonstationary subtree by Lemma~\ref{clear subtrees}(3), 
we can consider only limit nodes.  
Suppose $t$ is a limit node.
Then by continuity of the nice collection $\left<W_t\right>_{t \in T}$,
if $B \in W_t$ then $B \in W_s$ for some $s <_T t$.
So
\begin{align*}
&\left\{ \text{limit nodes $t$ that are not $W_t$-eligible } \right\}		\\
	&= \left\{\text{limit } t \in T : 
		\exists s <_T t \exists B \in W_s \left[ \pred t \subseteq B \text{ and } t \notin B \right] \right\}	\\
	&= \bigtriangledown_{s \in T} \left\{\text{limit } t \in T : 
				   \exists B \in W_s \left[ \pred t \subseteq B \text{ and } t \notin B \right] \right\}	\\
	&= \bigtriangledown_{s \in T} \bigcup_{B \in W_s } \left\{\text{limit } t \in T : 
						         \left[ \pred t \subseteq B \text{ and } t \notin B \right] \right\}
				\in NS^T_\nu,
\end{align*}
and it follows that the set of 
nodes $t$ such that $t$ is not $W_t$-eligible is in $NS^T_\nu$, as required.
\end{proof}
\end{lemma}

Combining the last two lemmas,
we are guaranteed a large supply of nodes and models that satisfy 
both the $\kappa$-completeness and eligibility requirements,
provided we start with a non-special tree and that condition~\eqref{necessary-k-very-nice} holds:

\begin{corollary}\label{eligible-and-complete}
Suppose $\nu$ is any infinite cardinal, 
$T$ is a non-$\nu$-special tree (necessarily of height $\nu^+$), 
and
$\theta > \nu$ is a regular cardinal such that 
$T \subseteq H(\theta)$.
Suppose 
$\kappa$ is an infinite cardinal, 
and 
$\left<N_t\right>_{t \in T}$ is a $\kappa$-very nice collection of elementary submodels of $H(\theta)$.
Then the set 
\[
\left\{ t \in T : t \text{ is $N_t$-eligible and } \left[N_t\right]^{<\kappa} \subseteq N_t \right\}
\]
is a stationary subtree of $T$.

\begin{proof}

From Lemma~\ref{k-complete-stationary},
the set
\[
	\left\{ t \in T : \left[N_t\right]^{<\kappa} \subseteq N_t \right\},
\]
is 
a stationary subtree of $T$.
By Lemma~\ref{most-are-eligible}, we have
\[
\left\{ t \in T : t \text{ is not $N_t$-eligible } \right\} \in NS^T_\nu.
\]
Our desired set is obtained by subtracting a nonstationary subtree from a stationary subtree, so it must be stationary,
as required.
\end{proof}
\end{corollary}


%
%

\section{Erd\H os-Rado Theorem for Trees 
						}
\label{section:ER}

To demonstrate the power of the tools we have developed in the previous two sections,
we now (similarly to~\cite[Section~2]{BHT}) divert our attention from the Main Theorem to show how the
machinery we have developed allows us to prove Theorem~\ref{ER-trees-Stevo} for the case
where $\kappa$ is a regular cardinal:



\begin{proof}[Proof of Theorem~\ref{ER-trees-Stevo} for regular $\kappa$]
As in the hypotheses, 
{fix} 
an infinite regular cardinal $\kappa$,
%
and a non-$(2^{<\kappa})$-special tree $T$ (necessarily of height $(2^{<\kappa})^+$), and
let $c : [T]^2 \to \mu$ be a colouring, 
where $\mu < \kappa$.  
We are looking for a chain of order type $\kappa +1$, homogeneous for $c$.

Let $\theta$ be any regular cardinal large enough so that $T \in H(\theta)$.
Using Theorem~\ref{regular-2<kappa}, we have $(2^{<\kappa})^{<\kappa} = 2^{<\kappa}$.
Then we use Lemma~\ref{build-nice-collection} (parts~(1) and~(3)) to
fix a $\kappa$-very nice collection $\left<N_t\right>_{t \in T}$ of elementary submodels of $H(\theta)$
such that $T, c \in 
N_\emptyset$.

Since $T$ is a non-$(2^{<\kappa})$-special tree,
Corollary~\ref{eligible-and-complete} tell us that
the set 
\[
\left\{ t \in T : t \text{ is $N_t$-eligible and } \left[N_t\right]^{<\kappa} \subseteq N_t \right\}
\]
is a stationary subtree of $T$,
so we can choose some node $t \in  T$ such that
$t$ is $N_t$-eligible and $[N_t]^{<\kappa} \subseteq N_t$.
Fix such a node $t$. 
Since $N_t$ was taken from a nice collection of elementary submodels, 
we have $\pred t \subseteq N_t$.
Then Corollary~\ref{homog-k+1} gives us a chain of order-type $\kappa+1$ homogeneous for $c$, as required.
\end{proof}

%
%

%% file: BHT_main_proof.tex


\section{Proof of the Main Theorem}

In this section, we will prove the Main Theorem, Theorem~\ref{BHT-trees-regular}.


As in the hypotheses of the Main Theorem~\ref{BHT-trees-regular}, 
{fix} 
an infinite regular cardinal $\kappa$.
Let $\nu = 2^{<\kappa}$.
Fix a non-$\nu$-special tree $T$ (necessarily of height $\nu^+$), 
a natural number $k$, and a colouring 
\[
c : \left[ T\right]^2 \to k.
\]

Since $\kappa$ is regular,
Theorem~\ref{regular-2<kappa} gives us $\nu^{<\kappa} = \nu$,
a fact that will be essential in the proof.

Let $\theta$ be any regular cardinal large enough so that $T \in H(\theta)$.
Using Lemma~\ref{build-nice-collection} (parts~(1) and~(3)) and the fact that $\nu^{<\kappa} = \nu$,
fix a $\kappa$-very nice collection $\left<N_t\right>_{t \in T}$ of elementary submodels of $H(\theta)$
such that $T, c \in N_\emptyset$.

The proof of Theorem~\ref{ER-trees-Stevo} in \autoref{section:ER} relied on Corollary~\ref{homog-k+1}, 
where we were able to obtain a
homogeneous set of order type $\kappa +1$ relatively easily using
the co-ideal $I_{N,t}^+$ constructed from a single elementary submodel
$N$. However, to obtain a homogeneous set of order type $\kappa +
\xi$, where $\xi > 1$, we need to do some more work.  In particular,
it will not be so easy to determine, initially, which $i<k$ will be
the colour of the required homogeneous set, so we must devise a
technique for describing sets that simultaneously include
homogeneous subsets for several colours. For this, we will need some
more machinery.


Recall that in \autoref{section:ideals-esm} 
we used 
nodes $t \in T$ and models $N$ to create algebraic structures on $\pred t$,
including the ideal $I_{N,t}$.
Now that we have fixed a nice collection 
of elementary submodels indexed by $T$,
we will generally 
allow the node $t$ to determine the model $N_t$ and therefore the corresponding structures on $\pred t$.
We will therefore simplify our notation as follows:

\begin{definition}
We define, for each $t \in T$:
\begin{gather*}
\pi_t = \pi_{N_t,t};                   \\
\mathcal A_t = \mathcal A_{N_t,t};       \\
\mathcal G_t = \mathcal G_{N_t,t};       \\
I_t = I_{N_t,t}.
\end{gather*}
Furthermore, we will say that $t$ is \emph{eligible} if it is $N_t$-eligible.
\end{definition}


\begin{lemma}\label{models-in-cone}
Fix $r \in T$ and $B \in N_r$.
Then:
\begin{enumerate}
\item For all $s \geq_T 
r$, we have $B \in N_s$ and $B \cap \pred s \in \mathcal A_s$.
\item For all eligible nodes $s \geq_T 
r$, we have
\[
s \in B \iff B \cap \pred s \in \mathcal G^+_s \iff B \cap \pred s \in I^+_s.
\]
\end{enumerate}

\begin{proof}\hfill
\begin{enumerate}
\item This result follows from the fact that
the nice collection of models is 
increasing, as well as
the definition of $\mathcal A_s$. 
\item Since $B \in N_s$ from part~(1), 
this result follows from Lemmas~\ref{eligible-equivalent} and~\ref{I_N-formulations}.
\qedhere
\end{enumerate}
\end{proof}

\end{lemma}


\begin{definition}
Let $S \subseteq T$ be any subtree, and suppose $t \in T$.
If $S \cap \pred t \in I^+_t$, then $t$ is called a
\emph{reflection point} of $S$.
\end{definition}

Some easy facts about reflection points:

\begin{lemma}\label{easy-reflection}
\hfill
\begin{enumerate}
\item If $t \in T$ is a reflection point of some subtree $S \subseteq T$, then $t$ is eligible.
\item If $t \in T$ is a reflection point of $S$, then $t$ is a limit point of $S$.
\item If $R \subseteq S \subseteq T$ and $t \in  T$ is a reflection point of $R$, then $t$ is a reflection point of $S$.
\end{enumerate}

\begin{proof}\hfill
\begin{enumerate}
\item Since $S \cap \pred t \in I^+_t$, we have in particular that $I^+_t \neq \emptyset$,
which is equivalent by Lemma~\ref{eligible-equivalent} to $t$ being eligible.
\item Since $t$ is eligible by part~(1), and also $\pred t \subseteq N_t$,
Lemma~\ref{eligible-implications}(4) tells us that $t$ must be a limit node.
By Lemma~\ref{bounded-in-ideal}(3), since $S \cap \pred t \in
I^+_t$, $S \cap \pred t$ must be cofinal in $\pred t$.  
It follows that $t$ must
be a limit point of $S \cap \pred t$, and therefore also of $S$.
\item $I^+_t$ is a co-ideal and therefore closed under
supersets. \qedhere
\end{enumerate}
\end{proof}

\end{lemma}

We want to be able to know when some eligible $t \in T$ is
a reflection point of some subtree $S \subseteq T$.  Is it enough
to assume that $t \in S$?  If $S \in N_t$ and $t \in
S$ is eligible, then we have $S \cap \pred t \in \mathcal G^*_t = G^+_t 
\subseteq 
I^+_t$ by Lemma~\ref{eligible-equivalent}, so 
that $t$ is
a reflection point of $S$.  Furthermore, if $S \in N_t$ for
some $t \in T$, 
the combination of Lemma~\ref{models-in-cone}(2) and Lemma~\ref{easy-reflection}(1) tells us 
precisely which $u \in \cone t$ are reflection points of
$S$, namely those eligible $u \in \cone t$ such that $u \in S $.  
But what if $S \notin N_t$? Then we can't
guarantee that every eligible $t \in S$ is a reflection point of $S$,
but we can get close.  The following lemma will be applied several
times 
throughout the proof of the Main Theorem:

\begin{lemma}\label{reflection}(cf.~\cite[Lemma~3.2]{BHT})
For any $S \subseteq T$, we have
\[
\left\{ t \in S : S \cap \pred t \in I_t \right\} \in NS^T_\nu.
\]

\end{lemma}

In the case where $S$ itself is a nonstationary subtree, 
Lemma~\ref{reflection} is trivially true. 
But then the result is also
useless, as we do not obtain any reflection points.  The
significance of the lemma is when $S$ is stationary in $T$. 
In that case, the lemma 
tells us that ``almost all'' points of $S$ are
reflection points: 
the set of points of $S$ that are \emph{not}
reflection points is a nonstationary subtree of $T$.  
In \cite[Lemma~3.2]{BHT} (dealing with the special case where the tree is a cardinal), 
the lemma is stated for stationary sets S, and the conclusion is worded differently, but the
fact that $S$ is stationary is not actually used at all in the
proof.

\begin{proof}[Proof of Lemma~\ref{reflection}]
Recall that the problem was that $S$ is not necessarily in any of
the models $N_t$ already defined.  We therefore \emph{fatten}
the models to include $\{S\}$.  
That is, we use
Lemma~\ref{build-nice-collection}(2) to construct another nice collection 
$
\left<M_t \right>_{t \in T}$ of
elementary submodels of $H(\theta)$ that is a fattening of the collection 
$\left< N_t \right>_{t \in T}$ 
(meaning that for all $t \in T$ we have $N_t \subseteq M_t$),
and such that $S \in M_\emptyset$ (so that $S \in M_t$ for all $t \in T$). 
The idea is that while initially there may be some eligible nodes in $S$ that are not reflection points of $S$,
by fattening the models to contain $S$ all of those points will become ineligible 
with respect to the new collection of models, showing that there cannot be too many of them.


More precisely:  
Fix any $t \in T$.  We have $S \in M_t$.
If $t$ is $M_t$-eligible, then we have (using Lemmas~\ref{eligible-equivalent} ans~\ref{I_N-formulations})
\[
t \in S \iff S \cap \pred t \in \mathcal G^*_{M_t,t} = \mathcal G^+_{M_t,t} \iff S \cap \pred t \in I^+_{M_t,t}.
\]
Since $N_t \subseteq M_t$, we apply Lemma~\ref{fatten} to get
\[
I^+_{M_t,t} \subseteq I^+_{N_t,t} = I^+_t.
\]
It follows that if $t \in S$ is $M_t$-eligible, then $S \cap \pred t \in I^+_t$,
so that $t$ is a reflection point of $S$.
Equivalently, if $t \in S$ satisfies $S \cap \pred t \in I_t$, then $t$ must not be $M_t$-eligible.

%

Applying Lemma~\ref{most-are-eligible} to the nice collection $\left< M_t \right>_{t\in T}$,
we then have
\[
\left\{ t \in S : S \cap \pred t \in I_t \right\} \subseteq
	\left\{ t \in T : t \text{ is not $M_t$-eligible } \right\} \in NS^T_\nu.
\]
which is the required result.
\end{proof}


\begin{definition}
We define subtrees $S_n \subseteq T$, for $n \leq
\omega$, by recursion on $n$, as follows:%
\footnote{
It may be possible to omit the requirement ``$t$ is eligible'' from the definition of $S_0$.
Ultimately, this should not be a problem,
as the subsequent sets $S_n$ (for $n >0$) will consist only of reflection points (by Lemma~\ref{Sn-properties}(2)),
which are eligible by Lemma~\ref{easy-reflection}(1).
However, some of the subsequent lemmas will have to be qualified,
such as Lemma~\ref{equivalence-cone-above} and
Corollary~\ref{some-ideal-is-proper}.}

First, define 
\[
S_0 = \left\{ t \in T : t \text{ is eligible and } 
\left[N_t\right]^{<\kappa} \subseteq N_t \right\}.
%
\]
Then, for every $n <\omega$, define
\todo{Say something about the changing the condition to $t$ is a \emph{limit point} of $S_n$, as done in~\cite{BHT},
or possibly removing the condition altogether as done in~\cite{Hajnal-Larson-PR-Handbook}.}
\[
S_{n+1} = \left\{ t \in S_n : S_n \cap \pred t \in I^+_t
\right\}.
\]
Finally, define
\[
S_\omega = \bigcap_{n<\omega} S_n.
\]
\end{definition}

\begin{lemma}\label{Sn-properties}
The sequence $\left<S_n\right>_{n\leq\omega}$ satisfies the
following properties:
\begin{enumerate}
\item The sequence is decreasing, that is,
\[
T \supseteq S_0 \supseteq S_1 \supseteq S_2 \supseteq \dots \supseteq S_\omega.
\]
\item For all $n<m\leq\omega$, each $t \in S_m$ is a reflection
point of $S_n$, and therefore also a limit point of $S_n$.
\item For all $n<m\leq\omega$, the set $S_n \setminus S_m$ is a
nonstationary subtree of $T$.
\item For all $n\leq\omega$, $S_n$ is a stationary subtree of $T$.
\end{enumerate}

\begin{proof}\hfill
\begin{enumerate}

\item Straight from the definition.
\item For every $n<m\leq\omega$ we have $S_m \subseteq S_{n+1}$ from
(1), and $S_{n+1}$ consists only of reflection points of $S_n$ by
definition.  By Lemma~\ref{easy-reflection}(2), any reflection point
of $S_n$ must also be a limit point of $S_n$.
\item For each $j<\omega$, we have
\[
S_j \setminus S_{j+1} = \left\{ t \in S_j : S_j \cap \pred t \in I_t \right\},
\]
and this subtree is nonstationary by Lemma~\ref{reflection}.
We then have, for any $n < m \leq \omega$,
\[
S_n \setminus S_m = \bigcup_{n \leq j < m} \left(S_j \setminus
S_{j+1}\right),
\]
so this subtree is nonstationary, as it is the union of at most
countably many nonstationary subtrees.
\item Since $T$ is a non-$\nu$-special tree,
the fact that $S_0$ is stationary is Corollary~\ref{eligible-and-complete}.

For $0 < n \leq \omega$, we know from (3) that $S_0 \setminus S_n$
is nonstationary, so it follows that $S_n$ is stationary. \qedhere
\end{enumerate}
\end{proof}
\end{lemma}

\begin{lemma}\label{equivalence-cone-above}\hfill
\begin{enumerate}
\item
For every $r \in T$ and $B \in N_r$,
we have
\[
S_0 \cap \left( \{r\} \cup \cone r \right) \cap B
	= \left\{ s \in S_0 \cap \left( \{r\} \cup \cone r \right) : B \cap \pred s \in \mathcal G^+_s \right\}.
\]
\item For every $r \leq_T t$ in $T$ and $B \in N_r$,
we have
\[
S_0 \cap \left( \pred t \setminus \pred r \right) \cap B
	= \left\{ s \in S_0 \cap \left( \pred t \setminus \pred r \right) : B \cap \pred s \in \mathcal G^+_s \right\}.
\]
\end{enumerate}

\begin{proof}\hfill
\begin{enumerate}
\item
This follows from Lemma~\ref{models-in-cone}(2),
using the fact that nodes in $S_0$ are eligible.
\item This follows from part~(1), since (for $r \leq_T t$)
\[
\pred t \setminus \pred r \subseteq \{r\} \cup \cone r.
\qedhere
\]
\end{enumerate}
\end{proof}

\end{lemma}

We will now define ideals $I(t, \sigma)$ and $J(t,
\sigma)$ on $\pred t$, for certain nodes $t \in T$
and finite sequences of colours $\sigma \in
k^{<\omega}$.  
We continue to follow the convention as explained in Section~\ref{Section:notation},
that
properness is \emph{not} required for a collection of sets to be
called an ideal (or a filter).  
Some of the ideals we are about to define may not be proper.

Though we define the ideals $I(t,\sigma)$ and
$J(t,\sigma)$, our intention will be to focus on the
corresponding co-ideals, just as we said earlier regarding the
co-ideals $I^+_t$ corresponding to the ideals $I_t$.  
As
we shall see (Lemma~\ref{homog-from-coideal}), 
for a set to be in
some co-ideal $I^+(t,\sigma)$ implies that it will include
homogeneous sets of size $\kappa$ for every colour in the sequence
$\sigma$.  This gives us the flexibility to choose later which
colour in $\sigma$ will be used when we combine portions of such
sets to get a set of order type $\kappa + \xi$, homogeneous for the
colouring $c$. 
When reading the definitions, it will help the
reader's intuition to think of the co-ideals rather than the ideals.

\begin{definition}

We will define
\todo{Can we define $I(\dots)$ and $J(\dots)$ more
generally, not just for $t \in S_n$?  
See~\cite{Hajnal-Larson-PR-Handbook}.
}
ideals $J(t,\sigma)$ and $I(t,\sigma)$ jointly by
recursion on the length of the sequence $\sigma$.
The collection $J(t, \sigma)$ will be defined for all $\sigma
\in k^{<\omega}$ but only when $t \in S_{\left|\sigma\right|}$,
while the collection $I(t, \sigma)$ will be defined only for
nonempty sequences $\sigma$ but for all $t \in
S_{\left|\sigma\right|-1}$.


(When $\sigma \in k^n$ we say $\left|\sigma\right| = n$.)

\begin{itemize}
\item Begin with the empty sequence, $\sigma = \left<\right>$.  For
$t \in S_0$, we define
\[
J \left( t, \left<\right> \right) = I_t.
\]
\item Fix $\sigma \in k^{<\omega}$ and $t \in
S_{\left|\sigma\right|}$, and assume we have defined
$J(t,\sigma)$.  Then, for each colour $i<k$, we define
$I(t, \sigma^\frown{}\left<i\right>) \subseteq \mathcal
P(\pred t)$ by setting, for $X \subseteq \pred t$,
\[
X \in I(t, \sigma^\frown\left<i\right>) \iff X \cap c_i(t)
\in J(t, \sigma).
\]
\item Fix $\sigma \in k^{<\omega}$ with $\sigma\neq\emptyset$,
and assume we have defined $I(s,\sigma)$ for all $s \in
S_{\left|\sigma\right|-1}$.  Fix $t \in
S_{\left|\sigma\right|}$. We define $J(t,\sigma) \subseteq
\mathcal P(\pred t)$ by setting, for $X \subseteq \pred t$,
\[
X \in J(t, \sigma) \iff \left\{ s \in
S_{\left|\sigma\right|-1} \cap \pred t : X \cap \pred s \in I^+(s,
\sigma) \right\} \in I_t.
\]
\end{itemize}
\end{definition}

We have introduced the intermediary $J(t,\sigma)$ as an
intuitive aid to understand the recursive construction and
subsequent inductive proofs.  
As can be seen by examining the definition,
each $I$-ideal is defined from a $J$-ideal by extending the sequence of colours $\sigma$,
without changing the node $t$;
but each $J$-ideal is defined by considering $I$-ideals from nodes lower down in the tree,
without changing the sequence of colours.

In fact the collections
$I(t,\sigma)$ can be described explicitly without the
intermediary $J(t,\sigma)$ (
as in~\cite{BHT}), by saying, for $\sigma \neq \emptyset$,
\[
X \in I(t, \sigma^\frown\left<i\right>) \iff \left\{ s \in
S_{\left|\sigma\right|-1} \cap \pred t : X \cap c_i(t) \cap
\pred s \in I^+(s, \sigma) \right\} \in I_t.
\]
However, we have changed the base case from \cite{BHT}:  We do not
define $I(t,\left<\right>)$, and by setting
$J(t,\left<\right>) = I_t$, we eliminate an unnecessary
reflection step for the sequences of length 1.  This way,
$I(t,\left<i\right>)$ is defined in a more intuitive way, by
setting (for $X \subseteq \pred t$)
\[
X \in I \left(t, \left<i\right> \right) \iff X \cap c_i(t)
\in I_t,
\]
and this definition is valid for all $t \in S_0$, not just in
$S_1$.
Subsequent lemmas are easier to prove with this definition, and this
seems to be how $I(t,\sigma)$ is intuitively understood, even
in \cite{BHT}.

In particular, Corollary~\ref{some-ideal-is-proper} will be proven much
more easily, and its meaning is the intended intuitive one, which
was not the case under the original definition of
$I(t,\left<i\right>)$ in \cite{BHT}.  Furthermore,
Lemma~\ref{reflection-gives-nonempty-sigma} would not have been true
using the original definition in \cite{BHT}.

We now investigate some properties of the various collections
$I(t,\sigma)$ and $J(t,\sigma)$ and the relationships
between them.

\begin{lemma}\label{I-J-k-complete}
For each sequence $\sigma$ and each relevant $t$, the
collections $I(t,\sigma)$ and $J(t,\sigma)$ are
$\kappa$-complete ideals on $\pred t$ (though not necessarily
proper).

\begin{proof}
Easy, by induction over the length of the sequence
$\sigma$,
\todo{Consider writing this out.} 
using the fact that
we are using only nodes $t \in S_0$, so that each $I_t$
is a $\kappa$-complete ideal on $\pred t$ (Lemma~\ref{k-complete}(5)).
\end{proof}
\end{lemma}

We are going to need to take sets from co-ideals, so it would help
us to get a sense of when the ideals are proper.
\todo{Maybe focus less on properness and more on when
sets from the co-ideal $I^+_t$ are in some other co-ideal.}

For a fixed $t$ and $\sigma$, there is no particular
relationship between $I(t,\sigma)$ and any
$I(t,\sigma^\frown\left<i\right>)$.  In fact there is no
relationship between $I(t,\sigma)$ and $J(t,\sigma)$, as
the latter is defined in terms of $I(s,\sigma)$ for
$s<_T t$ only.  We need to explore relationships that do exist
between various 
ideals.

\begin{lemma}\label{J-intersect-I}
For each sequence $\sigma \in k^{<\omega}$ and each $t \in
S_{\left|\sigma\right|}$, we have
\[
J(t,\sigma) = \bigcap_{i<k}
I(t,\sigma^\frown\left<i\right>),
\]
and equivalently,
\[
J^+(t,\sigma) = \bigcup_{i<k}
I^+(t,\sigma^\frown\left<i\right>).
\]
In particular, if $J(t,\sigma)$ is proper, then for at least
one $i<k$, $I(t,\sigma^\frown\left<i\right>)$ must be proper.

\begin{proof}
For $X \subseteq \pred t$, we have
\begin{align*}
X \in J(t, \sigma)
    &\iff X \cap \bigcup_{i<k} c_i(t) \in J(t,\sigma)
            &&\left(\text{since } \pred t =
                                \bigcup_{i<k} c_i(t) \right)
                                                                \\
    &\iff \bigcup_{i<k}
        \left(X \cap c_i(t)\right) \in J(t,\sigma) \\
    &\iff \forall i<k
        \left[X \cap c_i(t) \in J(t,\sigma) \right]
            &&\text{(since $J(t,\sigma)$ is an ideal)}  \\
    &\iff \forall i<k
        \left[X \in I(t, \sigma^\frown\left<i\right>)\right] \\
    &\iff X \in \bigcap_{i<k}
        I(t, \sigma^\frown\left<i\right>)
            &&\qedhere
\end{align*}
\end{proof}
\end{lemma}

The following special case of Lemma~\ref{J-intersect-I} where $\sigma = \left<\right>$ can be thought of as
a reformulation of Lemma~\ref{some-colour-in-co-ideal} using the terminology of our new ideals
$J(t, \left<\right>)$ and $I(t, \left<i\right>)$:

\begin{corollary}\label{some-colour-in-coideal-reformulated}
For each $t \in S_0$, we have
\[
I^+_t = 
J^+(t,\left<\right>) = \bigcup_{i<k}
I^+(t,\left<i\right>).
\]
\end{corollary}

In particular, since $S_0$ consists only of eligible nodes, any $t \in S_0$ satisfies
$\pred t \in I^+_t$ 
by Lemma~\ref{eligible-equivalent},
so applying Corollary~\ref{some-colour-in-coideal-reformulated} to $\pred t$ gives:

\begin{corollary}\label{some-ideal-is-proper}
For each $t \in S_0$, there is some colour $i<k$ such that
$I(t,\left<i\right>)$ is proper, that is,
$I^+(t,\left<i\right>)$ is nonempty.%
\footnote{This result is not actually used,
but it corresponds to the sentence at the bottom of~\cite[p.~5]{BHT}.}

\end{corollary}

Recall that each $S_{n+1}$ consists of reflection points of $S_n$.
This is for a good reason:  If we were to allow $t \in S_{n+1}$
that was not a reflection point of $S_n$, then we would have $S_{n}
\cap \pred t \in I_t$, so for all $\sigma \in k^{n+1}$ we would
have $\pred t \in J(t,\sigma)$, so that $J(t,\sigma)$ could
not be proper, regardless of the sequence $\sigma$.  In contrast,
since we allow only reflection points in $S_{\left|\sigma\right|}$
each time we lengthen $\sigma$, we obtain the following lemma:


\begin{lemma}\label{ideal-intersection-of-levels}\hfill
\begin{enumerate}
\item
For all $n\geq0$ and all $t \in S_n$, we have%
\footnote{We do 
not really need the full strength of this lemma, 
though it is 
one of the elegant results from the ideals being defined the way they are.  
The problem is that it requires all sequences of a given length, including sequences with repeated colours.
The only consequence of this lemma that we will actually need is the inclusion described in 
Lemma~\ref{ideal-inclusion}. 
}
\[
\mathcal G^+_t = \bigcup_{\sigma \in k^n} J^+(t, \sigma)
\cap \mathcal A_t,
\]
and equivalently,
\[
\mathcal G_t = \bigcap_{\sigma \in k^n} J(t, \sigma) \cap
\mathcal A_t.
\]
\item Similarly, For all $n\geq1$ and all $t \in S_{n-1}$, we
have
\[
\mathcal G^+_t = \bigcup_{\sigma \in k^n} I^+(t, \sigma)
\cap \mathcal A_t,
\]
and equivalently,
\[
\mathcal G_t = \bigcap_{\sigma \in k^n} I(t, \sigma) \cap
\mathcal A_t.
\]
\end{enumerate}

\begin{proof}
We prove parts (1) and (2) jointly by induction on $n$.
\begin{description}
\item[Base case for (1)]  When $n=0$, the only $\sigma \in k^0$ is the
empty sequence $\left<\right>$, and for any $t \in S_0$,
$J(t,\left<\right>)$ is defined to equal $I_t$, so the
equality reduces to Lemma~\ref{I_N-formulations}.

\item[Induction step, (1) $\implies$ (2)]  Fix $n\geq0$ and $t
\in S_{n}$, and assume that we have
\[
\mathcal G^+_t = \bigcup_{\sigma \in k^n} J^+(t, \sigma)
\cap \mathcal A_t.
\]
We need to show that
\[
\mathcal G^+_t = \bigcup_{\tau \in k^{n+1}} I^+(t, \tau)
\cap \mathcal A_t.
\]
But we have
\begin{align*}
\bigcup_{\tau \in k^{n+1}} I^+(t, \tau)
    &= \bigcup_{\sigma \in k^{n}} \bigcup_{i < k}
                        I^+(t, \sigma^\frown\left<i\right>) \\
    &= \bigcup_{\sigma \in k^{n}} J^+(t, \sigma)
            &&\text{(by Lemma~\ref{J-intersect-I}).}
\end{align*}
The conclusion now follows easily from the Induction Hypothesis.

\item[Induction step, (2) $\implies$ (1)]  Fix $n\geq1$, and assume
that for all $s \in S_{n-1}$ we have
\[
\mathcal G^+_s = \bigcup_{\sigma \in k^n} I^+(s, \sigma)
\cap \mathcal A_s.
\]
We now fix $t \in S_{n}$, and we must show that
\[
\mathcal G^+_t = \bigcup_{\sigma \in k^{n}} J^+(t, \sigma)
\cap \mathcal A_t.
\]

Fix $X \in \mathcal A_t$.  We need to show that
%
\[
X \in \mathcal G^+_t \iff X \in \bigcup_{\sigma \in k^{n}}
J^+(t, \sigma).
\]
Since $X \in \mathcal A_t$, we can fix some
$B \in \mathcal P(T) \cap N_t$ such that $X = B \cap \pred t$.  
It follows that for every
$s \leq_T t$ we have $X \cap \pred s = B \cap \pred s$.

Since $t \in S_n$ (where $n\geq1$), $t$ is a reflection point of $S_{n-1}$, that is,
$S_{n-1} \cap \pred t \in I^+_t$.
By Lemma~\ref{easy-reflection}(2), $t$ is also a limit point of $S_{n-1}$,
%
so that 
$t$ must also be a limit node of $T$.
By continuity of the nice collection of submodels,
we can fix some $r^{\min} <_T t$ such that $B \in N_{r^{\min}}$.
%
So Lemma~\ref{models-in-cone}(1)
gives 
$B \cap \pred s \in \mathcal A_s$
for all 
nodes $s \geq_T r^{\min}$,
and in particular, for all nodes $s \in \pred t \setminus \pred{r^{\min}}$.


Using all of these facts, we have
\begin{align*}
&X \in \mathcal G^+_t                  \\
    &\iff X \cap S_{n-1} \in I^+_t
            &&\text{(Lemma~\ref{reflection-intersection}
                )}     \\
    &\iff X \cap S_{n-1} \setminus \pred {r^{\min}} 
        \in I^+_t
            &&\text{(Lemma~\ref{bounded-in-ideal}(4))}
                                                    \\
    &\iff B \cap S_{n-1} \cap \left( \pred t \setminus \pred {r^{\min}} \right) 
        \in I^+_t
		&&(
			X = B \cap \pred t)
                                                    \\
    &\iff \left\{ s \in S_{n-1} \cap \left( \pred t \setminus \pred {r^{\min}} \right) 
			: B \cap \pred s \in \mathcal G^+_s \right\}
        \in I^+_t
            &&\text{(Lemma~\ref{equivalence-cone-above}(2))}    \\
    &\iff \left\{ s \in S_{n-1} \cap \left( \pred t \setminus \pred {r^{\min}} \right) 
			: B \cap \pred s \in \bigcup_{\sigma \in k^n}
        I^+(s, \sigma) \right\} \in I^+_t
            &&\text{(by Ind.\ Hyp.)}      \\
    &\iff \bigcup_{\sigma \in k^n} \left\{ s \in S_{n-1} \cap
        \left( \pred t \setminus \pred {r^{\min}} \right) 
			: B \cap \pred s \in
        I^+(s, \sigma) \right\} \in I^+_t  \\
    &\iff \bigcup_{\sigma \in k^n} \left\{ s \in S_{n-1} \cap
        \pred t : B \cap \pred s \in
        I^+(s, \sigma) \right\} \in I^+_t 
            &&\text{(Lemma~\ref{bounded-in-ideal}(4))}
                                                    \\
%
    &\iff \exists \sigma \in k^n \left(\left\{ s \in S_{n-1} \cap
        \pred t : B \cap \pred s \in
        I^+(s, \sigma) \right\} \in I^+_t \right)
                                                    \\
    &\iff \exists \sigma \in k^n \left(\left\{ s \in S_{n-1} \cap
        \pred t : X \cap \pred s \in
        I^+(s, \sigma) \right\} \in I^+_t \right)  \\
    &\iff \exists \sigma \in k^n
        \left( X \in J^+(t, \sigma) \right)    \\
    &\iff X \in
        \bigcup_{\sigma \in k^n} J^+(t, \sigma),
\end{align*}
as required.\qedhere
\end{description}
\end{proof}
\end{lemma}

Since (for $t \in S_0$) $\mathcal G_t$ is a proper ideal in $\mathcal A_t
\subseteq \mathcal P(\pred t)$,
Lemma~\ref{ideal-intersection-of-levels} implies the following
result:  For each $n\geq0$ and $t \in S_n$, there is some
$\sigma \in k^n$ such that $J(t,\sigma)$ is proper; similarly,
for each $n\geq1$ and $t \in S_{n-1}$, there is some $\sigma
\in k^n$ such that $I(t,\sigma)$ is proper.  When $n=1$, this
gives Corollary~\ref{some-ideal-is-proper}.  However, for larger $n$,
this fact is not as useful, because there is no way to guarantee
that the relevant sequence does not contain repeated colours.


Our main use of Lemma~\ref{ideal-intersection-of-levels} will be the
following:

\begin{lemma}\label{ideal-inclusion}(cf.~\cite[Lemma~3.3]{BHT}) 
\begin{enumerate}
\item
For all $\sigma \in k^{<\omega}$ and all $t \in
S_{\left|\sigma\right|}$, we have
\[
I_t \subseteq J(t, \sigma),
\]
and equivalently,
\[
I^+_t \supseteq J^+(t, \sigma), \text{ and } I^*_t
\subseteq J^*(t, \sigma).
\]
\item Similarly,
for all 
nonempty $\sigma \in k^{<\omega}$ and all $t \in
S_{\left|\sigma\right|-1}$, we have
\[
I_t \subseteq I(t, \sigma),
\]
and equivalently,
\[
I^+_t \supseteq I^+(t, \sigma), \text{ and } I^*_t
\subseteq I^*(t, \sigma).
\]
\end{enumerate}

\begin{proof}\hfill
\begin{enumerate}
\item
Fix $\sigma \in k^{<\omega}$ and $t \in
S_{\left|\sigma\right|}$.  From
Lemma~\ref{ideal-intersection-of-levels}(1), we see that
$\mathcal G_t \subseteq J(t, \sigma)$.
But $\mathcal G_t$ is a generating set for the ideal $I_t$
on $\pred t$.
%
%
Since $J(t,\sigma)$ is an ideal on $\pred t$, it follows that
$I_t \subseteq J(t, \sigma)$,
as required.
%
\item
Fix nonempty $\tau \in k^{<\omega}$ and $t \in
S_{\left|\tau\right|-1}$.  We write $\tau =
\sigma^\frown\left<i\right>$ for some $\sigma \in
k^{\left|\tau\right|-1}$ and $i<k$.  We then have
\begin{align*}
I_t    &\subseteq J(t,\sigma)
                    &&\text{(from part (1))}                    \\
            &\subseteq I(t,\sigma^\frown\left<i\right>)
                    &&\text{(from Lemma~\ref{J-intersect-I})}   \\
            &= I(t,\tau),
\end{align*}
%
%
as required.\qedhere
\end{enumerate}
\end{proof}
\end{lemma}

Lemma~\ref{ideal-inclusion} will be used several times in what
follows.

The following lemma is of slight interest in characterizing the
intersections of the ideals with $\mathcal A_t$ in the case
that they are proper, though we will 
not particularly need to use it:

\begin{lemma}
For all $\sigma \in k^{<\omega}$ and all $t \in
S_{\left|\sigma\right|}$, either
$J(t, \sigma) = \mathcal P(\pred t)$
or
$J(t, \sigma) \cap \mathcal A_t = \mathcal G_t$.
%
Similarly, for all nonempty $\sigma \in k^{<\omega}$ and all $t
\in S_{\left|\sigma\right|-1}$, either
$I(t, \sigma) = \mathcal P(\pred t)$
or
$I(t, \sigma) \cap \mathcal A_t = \mathcal G_t$.

\begin{proof}
Let $K$ be either $J(t,\sigma)$ or $I(t,\sigma)$
for some fixed 
$t$ and $\sigma$ satisfying the relevant
hypotheses.  By Lemma~\ref{I-J-k-complete}, $K$ is an ideal
on $\pred t$.  By Lemma~\ref{homomorphism-set-algebra} we know
that $\mathcal A_t$ is a set algebra over $\pred t$, so it
follows that $K \cap \mathcal A_t$ is an ideal in
$\mathcal A_t$.  
By Lemma~\ref{ideal-intersection-of-levels} we have
$\mathcal G_t \subseteq K \cap \mathcal A_t$.  
But Lemma~\ref{G_N-dichotomy} tells us that
either $\mathcal G_t = \mathcal A_t$ or
$\mathcal G_t$ is a maximal proper ideal in $\mathcal
A_t$.  So
it follows that $K \cap \mathcal A_t$ must equal either
$\mathcal G_t$ or $\mathcal A_t$.  In the first case we
are done. In the second case, we have $\pred t \in \mathcal A_t
= K \cap \mathcal A_t \subseteq K$, so that
$K = \mathcal P(\pred t)$, and we are done.
\end{proof}
\end{lemma}

\begin{lemma}\label{homog-from-coideal}(cf.~\cite[Lemma~3.4]{BHT})
\begin{enumerate}
\item Fix $\sigma \in k^{<\omega}$ and $t \in
S_{\left|\sigma\right|}$.  If $X \subseteq \pred t$ and $X \in
J^+(t,\sigma)$, then for all $j \in \range(\sigma)$ there is a
$j$-homogeneous chain $W \in [X]^\kappa$. 
\item 
Fix nonempty $\sigma \in k^{<\omega}$ and $t \in
S_{\left|\sigma\right|-1}$.  If $X \subseteq \pred t$ and $X \in
I^+(t,\sigma)$, then for all $j \in \range(\sigma)$ there is a 
$j$-homogeneous chain $W \in [X]^\kappa$. 
\end{enumerate}

\begin{proof}
We prove parts (1) and (2) jointly by induction over the length of
the sequence $\sigma$.

\begin{description}
\item[Base case for (1)]  If $\sigma = \left<\right>$ then
$\range(\sigma) = \emptyset$ so there is nothing to
show.

\item[Induction step, (1) $\implies$ (2)]  Fix $\sigma \in
k^{<\omega}$ and $t \in S_{\left|\sigma\right|}$, and assume
that for all
$Z \in J^+(t,\sigma)$ and all $j \in \range(\sigma)$ there is
$W \subseteq Z$ such that $\left|W\right| = \kappa$ and $W$ is
$j$-homogeneous.  We then fix $i<k$,
$X \in I^+(t, \sigma^\frown\left<i\right>)$, and $j \in
\range(\sigma^\frown\left<i\right>)$, and we must find $W \subseteq
X$ such that $\left|W\right| = \kappa$ and $W$ is $j$-homogeneous.

Since $X \in I^+(t, \sigma^\frown\left<i\right>)$, we have $X
\cap c_i(t) \in J^+(t, \sigma)$.

There are two cases to consider:
\begin{itemize}
\item $j \in \range(\sigma)$:  Since $X
\cap c_i(t) \in J^+(t, \sigma)$,
we use the Induction Hypothesis to find $W \subseteq X \cap
c_i(t)$ such that $\left|W\right| = \kappa$ and $W$ is
$j$-homogeneous.  But then $W \subseteq X$ and we are done.
\item $j=i$:  From Lemma~\ref{ideal-inclusion}(1) we have
$J^+(t,\sigma) \subseteq I^+_t$, so it follows that $X
\cap c_i(t) \in I^+_t$.  Applying Lemma~\ref{homog-2.2} to
the set $X \cap c_i(t)$, we get $i$-homogeneous $W \subseteq X
\cap c_i(t)$ of size $\kappa$, as required.
\end{itemize}

\item[Induction step, (2) $\implies$ (1)]
Fix nonempty $\sigma \in k^{<\omega}$ and assume that for all $s
\in S_{\left|\sigma\right|-1}$ and all $Z \subseteq \pred s$ such that
$Z \in I^+(s,\sigma)$ and all $j \in \range(\sigma)$ there is $W
\subseteq Z$ such that $\left|W\right| = \kappa$ and $W$ is
$j$-homogeneous.  We then fix $t \in S_{\left|\sigma\right|}$,
$X \in J^+(t, \sigma)$, and $j \in \range(\sigma)$, and we must
find $W \subseteq X$ such that $\left|W\right| = \kappa$ and $W$ is
$j$-homogeneous.

Since $X \in J^+(t, \sigma)$, we have
\[
\left\{ s \in S_{\left|\sigma\right|-1} \cap \pred t : X \cap
\pred s \in I^+(s, \sigma) \right\} \in I^+_t.
\]
In particular, this set, being in the co-ideal $I^+_t$, must be
non-empty.  So we fix $s \in S_{\left|\sigma\right|-1} \cap
\pred t$ such that $X \cap \pred s \in I^+(s, \sigma)$.  Then we
use the Induction Hypothesis to find $W \subseteq X \cap \pred s$ such
that $\left|W\right| = \kappa$ and $W$ is $j$-homogeneous.  But then
$W \subseteq X$ and we are done.\qedhere
\end{description}
\end{proof}
\end{lemma}

Until here, we have focused on 
describing co-ideals from which we can extract homogeneous sets of 
order-type $\kappa$.
Ultimately we will fix an ordinal $\xi < \log\kappa$, and our 
strategy will be to find some node $s \in T$ and chains $W, Y \subseteq T$ 
such that
\[
W <_T \{s\} <_T Y, 
\]
where $W$ has order type $\kappa$, $Y$ has order type $\xi$, and $W
\cup Y$ is homogeneous for the colouring $c$.  
We will now work on building structures from which we will be able to extract homogeneous sets of order-type $\xi$.

\begin{definition}
For any ordinal $\rho$ and sequence $\sigma \in k^{<\omega}$, we
consider chains in $T$ of order type
$\rho^{\left|\sigma\right|}$, and we define, by recursion over the
length of $\sigma$, what it means for such a chain to be
$(\rho,\sigma)$-good:
\todo{Can we come up with an alternative,
intuitive, non-recursive definition?  Something along the lines of:
$X$ has order-type $\rho^{\left|\sigma\right|}$, and when
enumerating the elements of $X$ in increasing order, as sums of
powers of $\rho$, if $s$ and $\gamma$ differ in the $j$th
coordinate but not in any higher coordinate, then $c\{s,\gamma\}
= \sigma(j)$.  See how Jones defines these sets. Then modify the
proofs of the following three lemmas accordingly.}
\begin{itemize}
\item Beginning with the empty sequence $\left<\right>$, we say that
every singleton set is $(\rho,\left<\right>)$-good.
\item Fix a sequence $\sigma \in k^{<\omega}$, and suppose we have
already decided which chains in $T$ are $(\rho,\sigma)$-good.
Fix a colour $i<k$.  We say that a chain $X \subseteq T$ of
order type $\rho^{\left|\sigma\right|+1}$ is
$(\rho,\sigma^\smallfrown\left<i\right>)$-good if
\[
X = \bigcup_{\eta<\rho} X_\eta,
\]
where the sequence $\left<X_\eta : \eta < \rho\right>$ satisfies the
following conditions:
\begin{enumerate}
\item for each $\eta<\rho$, the chain $X_\eta$ is
$(\rho,\sigma)$-good,
\item for each $\iota<\eta<\rho$, we have\footnote{The definition in 
\cite[p.~6]{BHT} uses $\sup X_\iota < \inf X_\eta$, which is
slightly stronger but seems not to be necessary.} 
$X_\iota <_T X_\eta$, and
\item for each $\iota<\eta<\rho$,
\[
c''\left(X_\iota \otimes X_\eta\right) = \{i\}.
\]
that is, for each $s \in X_\iota$ and $t \in X_\eta$, we have $c(\{s,t\}) = i$.
\end{enumerate}
\end{itemize}
\end{definition}

%

\begin{lemma}(cf.~\cite[Lemma~3.5]{BHT})
\label{homog-from-good}
Fix $\sigma \in k^{<\omega}$ and ordinal $\rho$.  If a chain $X \subseteq T$ is
$(\rho,\sigma)$-good, then for all $j \in \range(\sigma)$ there is
$Y \subseteq X$ such that $Y$ is $j$-homogeneous for $c$ and has
order-type $\rho$.

\begin{proof}
By induction over the length of the sequence $\sigma$.
\begin{description}
\item[Base case]  If $\sigma = \left<\right>$ then $\range(\sigma) =
\emptyset$ so there is nothing to show.

\item[Induction step]  Fix $\sigma \in k^{<\omega}$ and assume that
for every $(\rho,\sigma)$-good set $Z$ and all $j \in
\range(\sigma)$ there is $Y \subseteq Z$ such that $Y$ has order
type $\rho$ and $Y$ is $j$-homogeneous.  We then fix $i<k$,
$(\rho,\sigma^\frown\left<i\right>)$-good set $X$,
and $j \in \range(\sigma^\frown\left<i\right>)$, and we must find $Y
\subseteq X$ such that $Y$ has order type $\rho$ and $Y$ is
$j$-homogeneous.

There are two cases to consider:
\begin{itemize}
\item $j \in \range(\sigma)$:  By definition of $X$ being
$(\rho,\sigma^\frown\left<i\right>)$-good, $X$ includes some set
$X_0$ that is $(\rho,\sigma)$-good.
Then by the Induction Hypothesis, there is $Y \subseteq X_0$  with
order type $\rho$ that is $j$-homogeneous.  But then $Y \subseteq X$
and we are done.
\item $j=i$:  We decompose $X$ into its component subsets $X_\eta$,
$\eta<\rho$.  For each $\eta<\rho$, choose $\gamma_\eta \in X_\eta$.
Then the set
\[
Y = \left<\gamma_\eta \right>_{\eta<\rho}
\]
is
$i$-homogeneous and satisfies the required conditions.
\qedhere
\end{itemize}
\end{description}
\end{proof}
\end{lemma}

\begin{lemma}\label{exist-good}(cf.~\cite[Lemma~3.6]{BHT}) 
\begin{enumerate}
\item
Fix $\sigma \in k^{<\omega}$ and $t \in
S_{\left|\sigma\right|}$.  If $X \in J^+(t,\sigma)$ then for
all $\rho<\kappa$ there is $Y \subseteq X$ that is
$(\rho,\sigma)$-good.
\item 
Fix nonempty $\sigma \in k^{<\omega}$ and $t \in
S_{\left|\sigma\right|-1}$.  If $X \in I^+(t,\sigma)$ then for
all $\rho<\kappa$ there is $Y \subseteq X$ that is
$(\rho,\sigma)$-good.
\end{enumerate}

\begin{proof}
Fix any ordinal $\rho < \kappa$.  We prove parts (1) and (2) jointly
by induction over the length of the sequence $\sigma$.

\begin{description}
\item[Base case for (1)]
If $X \in J^+(\sigma,\left<\right>)$ then
$X$ is certainly nonempty, so choose any $u \in X$, so that
$\{u\}$ is $(\rho,\left<\right>)$-good.
\item[Induction step, (1) $\implies$ (2)]
Fix $\sigma \in k^{<\omega}$ and $t \in
S_{\left|\sigma\right|}$, and assume that for all
$Z \in J^+(t,\sigma)$ 
there is $W \subseteq Z$ such that $W$ is $(\rho,\sigma)$-good.
We then fix $i<k$, 
and $X \in I^+(t, \sigma^\frown\left<i\right>)$, and
we must find $Y \subseteq X$ that is
$(\rho,\sigma^\frown\left<i\right>)$-good.

Since $X \in I^+(t, \sigma^\frown\left<i\right>)$, we have $X
\cap c_i(t) \in J^+(t, \sigma)$.

We will recursively construct a sequence $\left<Y_\eta : \eta <
\rho\right>$  of subsets of $X \cap c_i(t)$ that satisfies the
requirements for its union to be
$(\rho,\sigma^\frown\left<i\right>)$-good.

Fix an ordinal $\eta < \rho$ and assume that we have constructed a
sequence $\left<Y_\iota : \iota < \eta\right>$ satisfying the
required properties.  We show how to construct $Y_\eta$.

Let
\[
V = \bigcup_{\iota<\eta} Y_\iota.
\]
Since $\eta<\rho<\kappa$, $\left|\sigma\right|$ is finite, $\kappa$
is infinite, and for each $\iota < \eta$ we have
$\left|Y_\iota\right| = \left|\rho^{\left|\sigma\right|}\right|$,
it follows that
\[
\left|V\right|
    = \left|\bigcup_{\iota<\eta} Y_\iota \right|
    = \sum_{\iota<\eta} \left|Y_\iota \right|
    = \sum_{\iota<\eta} \left|\rho^{\left|\sigma\right|}\right|
    = \left|\eta\right| \cdot\left|\rho^{\left|\sigma\right|}\right|
    < \kappa.
\]
Of course $V \subseteq \pred t \subseteq N_t$, so that $V \in
[N_t]^{<\kappa}$. Since $t \in S_{\left|\sigma\right|}
\subseteq S_0$, we have $[N_t]^{<\kappa} \subseteq N_t$,
giving us $V \in N_t$.

Define
\[
B = \left\{ u \in T : \left(\forall s \in V \right)
\left[s <_T u \text{ and } c \left\{ s, u \right\} =
i \right] \right\}.
\]
Since $B$ is defined from parameters $T, V, c$, and $i$ that
are all in $N_t$, it follows by elementarity of $N_t$ that
$B \in N_t$.

Since $V \subseteq c_i(t)$, it follows from the definition of
$B$ that $t \in B$.  But then 
we have $B \cap \pred t \in 
\mathcal G^*_t
\subseteq I^*_t$.  By Lemma~\ref{ideal-inclusion}(1), we then
have $B \cap \pred t \in J^*(t,\sigma)$.  Recall that $X \cap
c_i(t) \in J^+(t, \sigma)$.  The intersection of a filter
set and a co-ideal set must be in the co-ideal, so we have $B \cap X
\cap c_i(t) \in J^+(t,\sigma)$.  We now apply the
Induction Hypothesis, obtaining $(\rho,\sigma)$-good
\[
Y_\eta \subseteq B \cap X \cap c_i(t).
\]
Since $Y_\eta \subseteq B$, we clearly have $V <_T Y_\eta$ and $c''(V
\otimes Y_\eta) = \{i\}$, as required, and we have completed the
recursive construction.

We now let
\[
Y = \bigcup_{\eta<\rho} Y_\eta
\]
so that $Y \subseteq X$ is $(\rho,
\sigma^\frown\left<i\right>)$-good, as required.
%
%
%

\item[Induction step, (2) $\implies$ (1)]
Fix nonempty $\sigma \in k^{<\omega}$ and assume that for all $s
\in S_{\left|\sigma\right|-1}$ and all $Z \subseteq \pred s$
such that $Z \in I^+(s,\sigma)$ 
there is $Y \subseteq Z$ such that $Y$ is $(\rho,\sigma)$-good.
We then fix $t \in S_{\left|\sigma\right|}$ and $X \in
J^+(t, \sigma)$, and
we must find $Y \subseteq X$ that is $(\rho,\sigma)$-good.

Since $X \in J^+(t, \sigma)$, we have
\[
\left\{ s \in S_{\left|\sigma\right|-1} \cap \pred t : X \cap
\pred s \in I^+(s, \sigma) \right\} \in I^+_t.
\]
In particular, this set, being in the co-ideal $I^+_t$, must be
non-empty.  So we fix $s \in S_{\left|\sigma\right|-1} \cap
\pred t$ such that $X \cap \pred s \in I^+(s, \sigma)$.  Then we
use the Induction Hypothesis to find $(\rho,\sigma)$-good $Y
\subseteq X \cap \pred s$.  But then $Y \subseteq X$ and we are
done.\qedhere
\end{description}
\end{proof}

\end{lemma}

\begin{lemma}(cf.~\cite[Lemma~3.7]{BHT})\label{refine-good}
Fix $\sigma \in k^{<\omega}$ and%
\footnote{This lemma remains true with $m$ replaced by any cardinal, but we need only the finite case.}
$m<\omega$.  If $\rho$ and $\xi$
are any two ordinals such that
\[
\rho \to \left(\xi\right)^1_m,
\]
if $X \subseteq T$ is $(\rho,\sigma)$-good, and $g : X \to m$
is some colouring, then there is some $Y \subseteq X$, homogeneous
for $g$, such that $Y$ is $(\xi,\sigma)$-good.

\begin{proof}
Fix $m< \omega$ and ordinals $\rho$ and $\xi$ satisfying the
hypothesis.  We prove the lemma by induction over the length of the
sequence $\sigma$.
\begin{description}
\item[Base case]  If $X$ is $(\rho,\left<\right>)$-good, then it is a
singleton.  Any colouring $g$ on a singleton must go to only one
colour, so $X$ is homogeneous for $g$, and being a singleton it is
also $(\xi,\left<\right>)$-good.

\item[Induction step]  Fix $\sigma \in k^{<\omega}$ and assume that
for every $(\rho,\sigma)$-good $Z \subseteq T$ and colouring
$g : Z \to m$ there is a $(\xi,\sigma)$-good $W \subseteq Z$
homogeneous for $g$.  We then fix a colour $i<k$,
$(\rho,\sigma^\frown\left<i\right>)$-good $X \subseteq T$, and
a colouring $g : X \to m$, and we must find some $(\xi,
\sigma^\frown\left<i\right>)$-good $Y \subseteq X$ that is
homogeneous for $g$.

Since $X$ is $(\rho,\sigma^\frown\left<i\right>)$-good, we fix a
sequence $\left<X_\eta : \eta < \rho\right>$ satisfying the
conditions in the definition for
\[
X = \bigcup_{\eta<\rho} X_\eta
\]
to be $(\rho,\sigma^\frown\left<i\right>)$-good.

Consider any $\eta<\rho$.  The set $X_\eta$ is $(\rho,\sigma)$-good,
so we apply the Induction Hypothesis to $X_\eta$ and the restricted
colouring $g \upharpoonright X_\eta : X_\eta \to m$.  This gives us
$(\xi,\sigma)$-good $Y_\eta \subseteq X_\eta$ and a colour $j_\eta <
m$ such that $g'' Y_\eta = \{j_\eta\}$.

Now for each colour $j<m$, define the set
\[
V_j = \left\{ \eta < \rho : j_\eta = j \right\}.
\]
We now have a partition
\[
\rho = \bigcup_{j<m} V_j,
\]
so we can fix some $j<m$ and a set $H \subseteq V_j$ of order type
$\xi$.

Now set
\[
Y = \bigcup_{\eta \in H} Y_\eta.
\]
It is clear that $Y \subseteq X$ is $(\xi,
\sigma^\frown\left<i\right>)$-good and $j$-homogeneous for
$g$.\qedhere
\end{description}
\end{proof}
\end{lemma}

\begin{lemma}\label{exist-ordinal-pigeonhole}
Fix $m<\omega$.  For any infinite cardinal $\tau$, and any ordinal
$\xi<\tau$, there is some ordinal $\rho$ with $\xi\leq\rho<\tau$
such that
\[
\rho \to \left(\xi\right)^1_m.
\]

\begin{proof}
To see this, consider two cases:
\begin{itemize}
\item Suppose $\tau = \omega$.  In this case, $\xi <\tau$ is necessarily finite,
and we have
\[
(\xi-1) \cdot m + 1 \to \left(\xi\right)^1_m
\]
so we can let $\rho = (\xi-1)\cdot m+1$.
\item Otherwise, $\tau$ is an uncountable cardinal.  
(This is the case assumed in~\cite[Lemma~3.7]{BHT}.)
For $\xi<\tau$,
let $\rho = \omega^\xi$ (where the operation here is 
ordinal exponentiation).  We clearly have $\xi\leq\rho<\tau$.  Any ordinal
power of $\omega$ is
%
\emph{indecomposable}, that is,
\[
\left(\forall m<\omega\right) \left[\omega^\xi \to
\left(\omega^\xi\right)^1_m\right],
\]
giving us a homogeneous chain even longer than required.  
\qedhere
\end{itemize}
\end{proof}
\end{lemma}

From here onward, we will generally be working within the subtree
\[
S_\omega = \bigcap_{n<\omega} S_n,
\]
as defined earlier.
Notice that if $t \in S_\omega$, then (because $S_\omega \subseteq S_n$ for all $n<\omega$) 
$I(t,\sigma)$ and $J(t,\sigma)$ are 
defined for all 
$\sigma \in k^{<\omega}$
(provided $\sigma \neq \emptyset$ for defining $I(t,\sigma)$).

Also, rather than considering all possible finite sequences of
colours $\sigma \in k^{<\omega}$, we will consider only those
sequences that are:
\begin{itemize}
\item non-empty (to ensure that we can obtain a homogeneous set
of some colour), and 
\item one-to-one (distinct colours; without repetition --- to ensure that its length cannot be longer
than the number of colours, so that the collection of such sequences is finite).%
\footnote{We could have started from the beginning by allowing only sequences without repeated colours
in the definition of $I(t,\sigma)$ and $J(t,\sigma)$.
Some of the lemmas as stated would be problematic, 
such as Lemmas~\ref{J-intersect-I} and~\ref{ideal-intersection-of-levels},
but they are the ones whose full strength we are not using anyway.
%
}
\end{itemize}

\begin{definition}
We begin by defining
\[
\Sigma_0 = \left\{ \sigma \in k^{<\omega} : \sigma \neq \emptyset
\text{ and $\sigma$ is one-to-one} \right\}.
\]

For a stationary subtree $S \subseteq S_\omega$ and $t \in S$,
define
\todo{Why does this have to be defined only for stationary $S$?  And why only for $S \subseteq S_\omega$?}
\[
\Sigma(t,S) = \left\{ \sigma \in \Sigma_0 : S \cap \pred t \in
I^+(t,\sigma) \right\}.
\]
\end{definition}



For any $\sigma \in \Sigma_0$ it is clear that $1\leq
\left|\sigma\right| \leq k$.  We then have%
\todo{Can't use
$\sum$-notation for the sum, because of confusion with the $\Sigma$
symbol!}
\[
\left|\Sigma_0\right| = k + k(k-1) + \dots + k!
\]
which is finite.  
Since for any $t, S$ we have
$\Sigma(t,S) \subseteq \Sigma_0$, there are only finitely many distinct
sets $\Sigma(t,S)$.

\begin{lemma}\label{subset-gives-smaller-sigma}
For any stationary $R,S \subseteq S_\omega$, if $t \in R
\subseteq S$ then $\Sigma(t,R) \subseteq \Sigma(t,S)$.

\begin{proof}
If $R \subseteq S$ then certainly $R \cap \pred t \subseteq S \cap
\pred t$.  For any sequence $\sigma \in \Sigma_0$, we then have
\begin{align*}
\sigma \in \Sigma(t,R)
    &\implies R \cap \pred t \in I^+(t,\sigma)   \\
    &\implies S \cap \pred t \in I^+(t,\sigma)   \\
    &\implies \sigma \in \Sigma(t,S),
\end{align*}
as required.
\end{proof}
\end{lemma}

For any stationary subtree $S \subseteq S_\omega$, recall that $t$
is called a reflection point of $S$ if $S \cap \pred t \in
I^+_t$.  Also recall that by Lemma~\ref{reflection}, we have
\[
\left\{ t \in S : S \cap \pred t \in I_t \right\} \in NS^T_\nu.
\]


\begin{lemma}\label{reflection-gives-nonempty-sigma}
Fix any stationary subtree $S \subseteq S_\omega$.  
For any $t \in S$, the following are equivalent:
\begin{enumerate}
\item $S \cap \pred t \in I^+_t$;
\item There is some colour $i<k$ such that $\left<i\right> \in \Sigma(t,S)$;
\item $\Sigma(t,S) \neq \emptyset$.
\end{enumerate}
%
It follows that
\[
\left\{t \in S : \Sigma(t,S) = \emptyset \right\}
\]
must be a nonstationary subtree

\begin{proof}\hfill
\begin{description}
\item[(1) $\implies$ (2)]
%
Let $t$ be any reflection point of $S$.
We have
\[
S \cap \pred t
    \in I^+_t
    = \bigcup_{i<k} I^+(t,\left<i\right>)
\]
{by Corollary~\ref{some-colour-in-coideal-reformulated}}.
So there is some colour $i<k$ such that $S \cap \pred t \in
I^+(t,\left<i\right>)$.  But then $\left<i\right> \in
\Sigma(t,S)$, 
as required.
\item[(2) $\implies$ (3)] Clear.
\item[(3) $\implies$ (1)] 
Suppose $\Sigma(t,S) \neq \emptyset$, and choose $\sigma \in \Sigma(t,S)$.
So $S \cap \pred t \in I^+(t,\sigma)$.
Then Lemma~\ref{ideal-inclusion} gives $S \cap \pred t \in I^+_t$, as required.
\end{description}

The final statement then follows from Lemma~\ref{reflection}.
\end{proof}

\end{lemma}

For any stationary subtree, Lemma~\ref{reflection-gives-nonempty-sigma}
tells us that ``almost all'' of its points have nonempty $\Sigma$,
but we
would like to have a large set on which $\Sigma$ is constant.
Only the case $R_0 = S_\omega$ of the following lemma will be used in our proof 
of the Main Theorem~\ref{BHT-trees-regular}, 
but there is no extra effort in stating it with greater generality:

\begin{lemma} (cf.~\cite[Lemma~3.8]{BHT})\label{constant-sigma}
For every stationary subtree $R_0 \subseteq S_\omega$, there are a
stationary subtree $R \subseteq R_0$ and a fixed $\Sigma \subseteq
\Sigma_0$ such that for all stationary $S \subseteq R$ 
we have
\[
\left\{ t \in S : \Sigma(t,S) \neq \Sigma \right\} \in NS^T_\nu.
\]


\begin{proof}
Fix a stationary subtree $R_0 \subseteq S_\omega$, and recall that
$\Sigma_0$ is defined previously.

We will attempt to construct, recursively,
\todo{Decide whether to start at 0 or 1, 
and whether to end at finite $m$ or to present them as infinite sequences.} 
decreasing sequences
\todo{Settled on using symbols $\subsetneqq$ and $\supsetneqq$,
rather than the variants 
$\subsetneq, \varsubsetneq$, and $\varsubsetneqq$ (and their mirror-images).
}
\[
R_0 \supseteq R_1 \supseteq R_2 \supseteq R_3 \supseteq \cdots
\text{ and } 
\Sigma_0 \supsetneqq \Sigma_1 \supsetneqq \Sigma_2 \supsetneqq \Sigma_3 \supsetneqq \cdots
\]
satisfying the following properties for all $n\geq0$:
\begin{enumerate}
\item $R_n$ is stationary; and
\item
for all $t \in R_{n}$, we have%
\footnote{This condition was misstated in~\cite{BHT} and~\cite{Baumgartner}, leading to some confusion.} 
$\Sigma(t, R_n) \subseteq \Sigma_n$
%
\end{enumerate}

When $n=0$, we see that $R_0$ and $\Sigma_0$ satisfy the required
properties because $R_0$ was chosen to be stationary and every
possible $\Sigma(t,R_0)$ is a subset of $\Sigma_0$.


Fix $n\geq0$, and assume we have constructed $R_n$ and $\Sigma_n$
satisfying the requirements.  We attempt to choose $R_{n+1}$ and
$\Sigma_{n+1}$, as follows:


Consider any stationary set $S \subseteq R_n$.
%
%
For each $\Gamma \subseteq \Sigma_{n}$ define
\[
S^\Gamma = \left\{ t \in S : \Sigma(t,S) = \Gamma
\right\}.
\]

There are now two possibilities:
\begin{itemize}
\item If there is some stationary $S \subseteq R_n$ and
$\Gamma \subsetneqq \Sigma_n$ such that $S^\Gamma$ is stationary,
then let $\Sigma_{n+1} = \Gamma$ and $R_{n+1} = S^\Gamma$.  
For each
$t \in R_{n+1}$, since $R_{n+1} \subseteq S$, we have (using
Lemma~\ref{subset-gives-smaller-sigma})
\[
\Sigma(t, R_{n+1}) \subseteq \Sigma(t, S) = \Gamma = \Sigma_{n+1},
\]
so it is clear that $R_{n+1}$ and $\Sigma_{n+1}$ satisfy the
properties required for our decreasing sequences.

Recall that $\Sigma_0$ is finite.  A strictly decreasing sequence of
subsets of a finite set cannot be infinite, so after some finite
$m$, this alternative will be impossible.

\item Otherwise, for all stationary $S \subseteq R_n$ and all $\Gamma
\subsetneqq \Sigma_n$, $S^\Gamma$ is nonstationary.  So we set $R =
R_n$ and $\Sigma = \Sigma_n$ and we verify that these sets satisfy
the conclusion of the lemma:

Fix a stationary subtree $S \subseteq R_n$.
For any $t \in S$, Lemma~\ref{subset-gives-smaller-sigma} and property~(2) above give
\[
\Sigma(t,S) \subseteq \Sigma(t,R_n) \subseteq \Sigma_n, 
\]
so that we have
%
\[
\left\{ t \in S : \Sigma(t,S) \neq \Sigma_n \right\} 
	= \bigcup_{\Gamma \subsetneqq \Sigma_n} S^\Gamma.
\]
There are only finitely many subsets of $\Sigma_n$,
so this set is is the union of finitely many nonstationary subtrees, so it must be nonstationary, 
as required.
%
%
\qedhere
\end{itemize}
\end{proof}

\end{lemma}

From Lemma~\ref{reflection-gives-nonempty-sigma}, it follows that
any $\Sigma$ obtained from Lemma~\ref{constant-sigma} must be
nonempty.  Since any $\Sigma \subseteq \Sigma_0$ is also finite, it
is reasonable to consider a sequence of colours $\sigma \in \Sigma$
that is maximal by inclusion.  Here we explore the consequences of
$\sigma$ being maximal.


\begin{lemma}(cf.~\cite[Lemma~3.9]{BHT})\label{ideal-from-maximal-sigma}
Suppose $S \subseteq S_\omega$ is stationary, and there is $\Sigma
\subseteq \Sigma_0$ 
such that
\[
\left\{ t \in S : \Sigma(t,S) \neq \Sigma \right\} \in NS^T_\nu.
\]
Suppose also that $\sigma \in \Sigma$ is maximal by inclusion.  
Then
there are $s \in S$ with $\Sigma(s,S) = \Sigma$ and stationary $R \subseteq
S$, with $\{s\} <_T R$, such that
\[
\left(\forall t \in R \right) 
	\Biggl[ S \cap \pred s \cap \bigcup_{i \notin \range(\sigma)} c_i(t) \in I(s,\sigma) \Biggr].
\]

\begin{proof}
We define
\[
S' = \left\{ t \in S : \Sigma(t,S) = \Sigma \right\}
\]
and 
\[
S'' = \left\{ t \in S' : S' \cap \pred t \in I^+_t \right\}.
\]
By hypothesis, $S$ is stationary, and $\left\{ t \in S : \Sigma(t,S) \neq \Sigma \right\}$ is nonstationary,
so $S'$ is stationary.
Applying Lemma~\ref{reflection} to $S'$ gives us that $\left\{ t \in S' : S' \cap \pred t \in I_t \right\}$ is nonstationary,
so it follows that $S''$ is stationary.



By assumption, $\sigma$ is maximal in $\Sigma$.  That is, $\sigma
\in \Sigma$ but
\[
\left(\forall i\notin\range(\sigma)\right)
\left[\sigma^\frown\left<i\right> \notin \Sigma \right].
\]

Now consider any $t \in S''$.  
Since $t \in S'$, we have $\Sigma(t,S) = \Sigma$.
For every $i \notin
\range(\sigma)$, we have $\sigma^\frown\left<i\right> \notin
\Sigma(t,S)$, meaning that $S \cap \pred t \notin I^+(t,
\sigma^\frown\left<i\right>)$, equivalently $S \cap \pred t \in
I(t, \sigma^\frown\left<i\right>)$, and $S \cap \pred t \cap
c_i(t) \in J(t, \sigma)$.  It follows that
\[
\bigcup_{i\notin\range(\sigma)} S \cap c_i(t) \in J(t,
\sigma),
\]
meaning that
\todo{This is the step that would fail if we had
added $\cap X$ to the definition of $J(t,\sigma)$.}
\[
\left\{ s \in S_{\left|\sigma\right|-1} \cap \pred t :
\bigcup_{i\notin\range(\sigma)} S \cap c_i(t) \cap \pred s \in
I^+(s, \sigma) \right\} \in I_t.
\]
Since $t \in S''$, we have $S' \cap \pred t \in I^+_t$.
Then, since $S' \subseteq S \subseteq S_\omega \subseteq S_{\left|\sigma\right|-1}$,
we can choose $s_t \in S' \cap 
\pred t$
such that
\[
\bigcup_{i\notin\range(\sigma)} S \cap c_i(t) \cap \pred {s_t}
\in I(s_t, \sigma).
\]

So for every $t \in S'' 
$ (a stationary subtree of $T$), we have chosen
$s_t <_T t$ with $s_t \in S' 
$,
satisfying the formula immediately above.  This defines a regressive
function on a stationary subtree, 
so by Theorem~\ref{nonstationary-normal} it must have a constant value $s \in S'$ on some stationary subtree
$R \subseteq S''$ with $\{s\} <_T R$, meaning that for all $t \in R$, we have $s_t = s$,
giving
\[
\bigcup_{i\notin\range(\sigma)} S \cap c_i(t) \cap \pred s \in
I(s, \sigma).
\]
Since $R \subseteq S'' \subseteq S' \subseteq S$, this implies the required result.
\end{proof}
\end{lemma}

Now it's time to put everything together to get the required
homogeneous sets.  Fix an ordinal%
\footnote{Recall that $\log\kappa$ is the smallest cardinal $\tau$ such that
$2^\tau \geq \kappa$.}
\todo{Consider not fixing $\xi$ until the next page.} 
$\xi < \log \kappa$. 
Recall that $T$ is a non-$\nu$-special tree (where $\nu = 2^{<\kappa}$), and $c : [T]^2 \to k$, and
we need to find a chain $X \subseteq T$ of
order type $\kappa + \xi$ that is homogeneous for the partition $c$.


Recall that $S_\omega$ is stationary (Lemma~\ref{Sn-properties}(4)).
Using Lemma~\ref{constant-sigma}, 
we fix stationary $S \subseteq
S_\omega$ and $\Sigma \subseteq \Sigma_0$ such that for all
stationary $R \subseteq S$ 
we have
\[
\left\{ u \in R : \Sigma(u,R) \neq \Sigma \right\} \in NS^T_\nu.
\]


Using Lemma~\ref{reflection-gives-nonempty-sigma}, 
$\Sigma \neq
\emptyset$.  Fix $\sigma \in \Sigma$ that is maximal by inclusion,
and let $m = \left|\sigma\right|$.


We now apply Lemma~\ref{ideal-from-maximal-sigma} 
to $S$, $\Sigma$, and $\sigma$.  
This gives us $s \in S$ with $\Sigma(s,S) = \Sigma$ and
stationary $R \subseteq S$, with $\{s\} <_T R$, such that
\[
\left(\forall u \in R \right) \left[ S \cap \pred s \cap
\bigcup_{i \notin \range(\sigma)} c_i(u) \in I(s,\sigma)
\right].
\]
Our goal will be to find chains $W \subseteq S \cap \pred s$ and $Y \subseteq R$ such that
$W$ has order-type $\kappa$, $Y$ has order-type $\xi$, and $W \cup Y$ is homogeneous for $c$.
%
That is, we require the chains
$W$ and $Y$ to satisfy
\[
\left[W\right]^2 \cup \left(W \otimes Y\right) \cup \left[Y\right]^2
\subseteq c^{-1}(\{i\})
\]
for some $i<k$.

Since 
$\Sigma(s,S) = \Sigma$, we have 
$\sigma \in \Sigma(s,S)$, meaning
\[
S \cap \pred s \in I^+(s,\sigma).
\]

Since $R \subseteq S$, 
by choice of $S$ we have
\[
\left\{ u \in R : \Sigma(u,R) \neq \Sigma \right\} \in NS^T_\nu,
\]
and $R$ is stationary, so we can fix
$u \in R$ such that $\Sigma(u,R) = \Sigma$, 
so that $\sigma \in \Sigma = \Sigma(u,R)$, giving
\[
R \cap \pred u \in I^+(u,\sigma).
\]

We have $\xi < \log \kappa \leq \kappa$, where of course $\log\kappa$ is infinite.

We apply Lemma~\ref{exist-ordinal-pigeonhole} to the ordinal $\xi$, 
obtaining
an ordinal $\rho$ with $\xi\leq\rho<\log\kappa$ such that
\[
\rho \to \left(\xi\right)^1_m.
\]

We then apply Lemma~\ref{exist-good} to $R \cap \pred u$ and the ordinal $\rho$, to 
obtain $Z \subseteq R \cap \pred u$ that is
$(\rho,\sigma)$-good.  Since $Z \subseteq R$, we have $\{s\} <_T
Z$ and for every $t \in Z$ we have
\[
S \cap \pred s \cap \bigcup_{i \notin \range(\sigma)} c_i(t) \in
I(s,\sigma).
\]

Since $Z$ is $(\rho,\sigma)$-good, it
has order type $\rho^m$, and therefore $\left|Z\right| =
\left|\rho^m\right| < \log\kappa \leq \kappa$.  Since $I(s,\sigma)$
is a $\kappa$-complete ideal (Lemma~\ref{I-J-k-complete}), 
it follows that
\[
\bigcup_{t \in Z} \left( S \cap \pred s \cap \bigcup_{i \notin
\range(\sigma)} c_i(t) \right) \in I(s,\sigma),
\]
or
\[
    S \cap \pred s \cap \bigcup_{t \in Z}
\left(
    \bigcup_{
    \substack{
              i \notin \range(\sigma)}}
    c_i(t)
\right)
    \in I(s,\sigma).
\]
We now let
\[
H =
    S \cap \pred s \setminus \bigcup_{t \in Z}
\left(
    \bigcup_{
    \substack{
              i \notin \range(\sigma)}}
    c_i(t)
\right),
\]
and since $S \cap \pred s \in I^+(s,\sigma)$, it follows that
\[
H \in I^+(s,\sigma).
\]
We can also write
\[
H = \left\{ r \in S \cap \pred s :
    \left(\forall t \in Z \right)
    \left[c(\{r,t\}) \in \range(\sigma) \right]
\right\}.
\]

For each $r \in H$, we define a function $g_r : Z \to
\range(\sigma)$ by setting, for each $t \in Z$,
\[
g_r (t) = c(\{r,t\}).
\]
How many different functions from $Z$ to $\range(\sigma)$ can there
be?  At most $\left|\sigma\right|^{\left|Z\right|}$.  
But $\left|Z\right| < \log\kappa$ and $\sigma$ is finite, so $\left|\sigma\right|^{\left|Z\right|} < \kappa$.

For each function $g : Z \to \range(\sigma)$, define
\[
H_g = \left\{ r \in H : g_r = g \right\}.
\]
There are fewer than $\kappa$ such sets, and their union is all of
$H$, which is in the $\kappa$-complete co-ideal $I^+(s,\sigma)$,
so there must be some function $g$ such that $H_g \in
I^+(s,\sigma)$.  Fix such a function $g : Z \to \range(\sigma)$.

We then apply Lemma~\ref{refine-good} 
to the colouring $g$, and we
obtain $Z' \subseteq Z$, homogeneous for $g$, that is
$(\xi,\sigma)$-good.  That is, we have a $(\xi,\sigma)$-good $Z'
\subseteq Z$ and a fixed colour $i \in \range(\sigma)$ such that for
all $t \in Z'$ we have $g(t) = i$.
But this means that for all $r \in H_g$ and all $t \in Z'$
we have
\[
c(\{r,t\}) = g_r(t) = g(t) = i,
\]
showing that $H_g \otimes Z' \subseteq c^{-1}(\{i\})$.

Now $Z'$ is $(\xi,\sigma)$-good and $i \in \range(\sigma)$, so using Lemma~\ref{homog-from-good} 
we fix $Y \subseteq Z'$ that is
$i$-homogeneous for $c$ and has order type $\xi$.

Also, applying Lemma~\ref{homog-from-coideal} to $H_g$, 
we get $W
\subseteq H_g$ such that $\left|W\right| = \kappa$ and $W$ is
$i$-homogeneous for $c$.

So then $W \cup Y$ is $i$-homogeneous of order type $\kappa + \xi$,
as required.  This completes the proof of the Main Theorem, Theorem~\ref{BHT-trees-regular}.